\title{\large{\textbf{HIGH-DIMENSIONAL ASYMPTOTICS FOR PERCOLATION OF GAUSSIAN FREE FIELD LEVEL SETS}}}

\date{}

\documentclass[a4paper,11pt, leqno]{article}
\usepackage[english]{babel}
\usepackage[a4paper,left=2.6cm,right=2.6cm,top=2.5cm,bottom=2.8cm]{geometry}
\usepackage{amsfonts, amsmath, amssymb, amsthm, mathrsfs}

\usepackage{verbatim}
\usepackage{color}
\usepackage{paralist}
\usepackage{enumitem}

\usepackage[hang,flushmargin]{footmisc} 

\usepackage{psfrag}
\usepackage{graphicx}
\usepackage[font=footnotesize]{caption}

\numberwithin{equation}{section}

\newtheorem{theorem}{Theorem}[section]
\newtheorem{lemma}[theorem]{Lemma}

\newtheorem{proposition}[theorem]{Proposition}

\newtheorem{corollary}[theorem]{Corollary}

\theoremstyle{remark}

\newtheorem*{proof2}{Proof of Theorem \ref{T:ASYMPTOTICS_UB}}
\newtheorem*{proof3}{Proof of Lemma \ref{L:LOCAL_GLOBAL_q}}
\newtheorem*{proof4}{Proof of Lemma \ref{L:subt_comp_dom}}
\newtheorem*{proof5}{Proof of Lemma \ref{L:subst_comp_gwbound}}
\newtheorem*{proof6}{Proof of Lemma \ref{L:giant_event_H}}
\newtheorem*{proof7}{Proof of Theorem \ref{T:ASYMPTOTICS_LB}}
\newtheorem*{proof8}{Proof of Theorem \ref{T:ASYMPTOTICS_DENSITY}}
\newtheorem*{proof9}{Proof of Lemma \ref{L:MASTER_FORMULA_BD}}

\theoremstyle{definition}
\newtheorem{remark}[theorem]{Remark}

\addtocounter{section}{-1}

\begin{document}

\maketitle

\begin{center}
\vspace{-1cm}
Alexander Drewitz\footnote{\noindent Department of Mathematics, Columbia University, 2990 Broadway, New York City, NY 10027, USA. E-mail: drewitz@math.columbia.edu.} and Pierre-Fran\c cois Rodriguez\footnote{\noindent Departement Mathematik, ETH Z\"urich, R\"amistrasse 101, CH-8092 Z\"urich, Switzerland. E-mail: pierre.rodriguez@math.ethz.ch. This research was supported in part by the grant ERC-2009-AdG 245728-RWPERCRI.} 

\end{center}
\vspace{0.3cm}
\begin{abstract}
\centering
\begin{minipage}{0.8\textwidth}
We consider the Gaussian free field on $\mathbb{Z}^d$, $d \geq 3$, and prove that the critical density for  percolation of its level sets behaves like $1/d^{1 + o(1)}$ as $d$ tends to infinity. Our proof gives the principal asymptotic behavior of the corresponding critical level $h_*(d)$. Moreover, it shows that a related parameter $h_{**}(d) \geq h_*(d)$ introduced by Rodriguez and Sznitman in \cite{RS} is in fact asymptotically equivalent to $h_*(d)$.
\end{minipage}
\end{abstract}
\thispagestyle{empty}


\newpage
\mbox{}
\thispagestyle{empty}
\newpage

\section{Introduction}

When studying the statistical mechanics of random interfaces which typically arise between coexisting phases of a ($d+1$-dimensional) physical system in equilibrium, one often considers so-called \textit{effective} models, which aim at describing the $d$-dimensional surface \textit{itself}, free from its surroundings. Arguably the most notorious example in this class is the massless harmonic crystal, or Gaussian free field (precise definitions will follow, see \eqref{phi} below). A natural approach in trying to gain some insight into the geometry of this field is to inquire about its level sets, say, above a given height $h \in \mathbb{R}$. In case the underlying space is the cubic lattice $\mathbb{Z}^d$, with $d \geq 3$, and due to the presence of strong correlations (the susceptibility is infinite), this gives rise to an interesting percolation model, which was originally introduced by Lebowitz and Saleur in \cite{LS}, and has since then been investigated in \cite{BLM}, \cite{DRS}, \cite{Ga}, \cite{Ma}, \cite{RP} and \cite{RS}, see also \cite{AS}, \cite{MS} for related results.

Only recently has it been shown in \cite{RS} that the associated phase transition is non-trivial in \textit{all} dimensions $d \geq 3$ (partial results were already obtained in \cite{BLM} and \cite{Ga}). Our main focus in the present work is to examine the limiting behavior of certain critical parameters associated to this transition in high dimension. Of fundamental importance in this context is the heuristic principle by which this percolation model ought to fall in the ``domain of attraction'' of a corresponding model on the $(2d)$-regular tree. Our results indicate that this is indeed the case, and in fact, this paradigm permeates more or less explicitly many of the proofs below. Ideally, we would also like to compare these results with corresponding ones in the tree-case \textit{directly}. We hope to return to this point elsewhere.

A first step in the direction of high-dimensional asymptotics is given by Theorem 3.3 of \nolinebreak \cite{RS}, which asserts that the critical height $h_*(d)$ for percolation of Gaussian free field level sets, defined in \eqref{h_*} below, is strictly positive when $d$ is sufficiently large. In fact, a careful inspection of the proof of this theorem yields that $h_*(d) \to \infty$ as $d\to \infty$, which amounts to saying that the critical density for this percolation model converges to $0$ as $d$ tends to infinity. We will considerably refine this result by providing the leading  asymptotic behavior of the critical density, as well as principal asymptotics for $h_*(d)$ and a related critical parameter $h_{**}(d)$ (see \eqref{h_**} below) as $d$ becomes large. Our proof follows in its broad lines the general strategy underlying similar results in Bernoulli (see \cite{Go} and \cite{ABS}, Section 4) and interlacement (see \cite{SLB} and \cite{SUB}) percolation. However, the implementation of this program highly depends on the specific nature of the model, as it crucially relies on a precise understanding of its dependence structure in high dimension. In particular, in the present context, the random walk representation of the Gaussian free field, which will lead us to its ``perturbative''  expansion around a suitable \textit{independent} field, will play a pivotal role in allowing for a precise understanding of the \textit{local} connectivity of the level set (we will explain this in greater detail below, see the discussion around \eqref{EQ:INTRO_PERT}). Moreover, a severe technical obstruction is the absence of a BK-type inequality (companion to the long-range dependence). In the independent case, this inequality underlies the successful deployment of such elaborate tools as the lace expansion, see \cite{HS1} and \cite{Sl}, Chapter 9. In our set-up, as a partial substitute, we develop suitable decoupling inequalities, much in the spirit of \cite{R}, \cite{RS}, with the notable difference that they will need to work ``\textit{uniformly} well'' for all sufficiently large $d$. These inequalities will typically allow for a certain kind of (static) renormalization procedure, to which the dimension $d$ will be inextricably tied.

We now describe our results and refer to Section  \ref{NOTATION} for details. We consider the lattice \nolinebreak $\mathbb{Z}^d$, $d \geq 3$, endowed with the usual nearest-neighbor graph structure, and investigate the Gaussian free field on $\mathbb{Z}^d$, with canonical law $\mathbb{P}$ on $ \Omega = \mathbb{R}^{\mathbb{Z}^d}$ (equipped with the product $\sigma$-algebra) such that,
\begin{equation} \label{phi}
\begin{split}
&\text{under $\mathbb{P}$, the canonical field $\varphi$ = $(\varphi_x)_{x \in \mathbb{Z}^d}$ is a centered Gaussian} \\
&\text{field with covariance $\mathbb{E}[\varphi_x \varphi_y] = g (x,y)$}, \text{ for all } x,y \in \mathbb{Z}^d,
\end{split}
\end{equation}
where $g(\cdot, \cdot)$ denotes the Green function of simple random walk on $\mathbb{Z}^d$, see \eqref{GreenFunction}. Note that, for fixed $d$, this model exhibits rather strong correlations, (in particular, the susceptibility is infinite for all $d\geq 3$, cf. \eqref{1.15} below). However, as $d$ grows, $\varphi$ ``approaches'' an independent field, in a sense to be made precise below. 

For any \textit{level} $h \in \mathbb{R}$, we introduce the random subset of $\mathbb{Z}^d$
\begin{equation} \label{Ephi}
E_{\varphi}^{\geq h}  = \{ x \in \mathbb{Z}^d ; \; \varphi_x \geq h \},
\end{equation}
often referred to as the \textit{excursion set} (or \textit{level set}) of the field $\varphi$ above height $h$. In order to study its percolative properties, we set $\eta(h) = \mathbb{P}[0 \stackrel{\geq h}{\longleftrightarrow} \infty]$, the probability that the origin lies in an infinite cluster (i.e. connected component) of $E_{\varphi}^{\geq h}$. The function $\eta(\, \cdot \, )$ being decreasing, we define the critical parameter for level-set percolation as
\begin{equation}\label{h_*}
h_*(d) = \inf \{ h \in \mathbb{R} \, ; \, \eta(h) =0 \}
\end{equation}
(with the convention $\inf \emptyset = \infty$). Following $(0.6)$ of \cite{RS}, we also introduce a second critical point
\begin{equation} \label{h_**}
h_{**}(d) = \inf \big\{ h \in \mathbb{R} \, ; \,  \lim_{L\to \infty} \mathbb{P} [ B(0,L) \stackrel{\geq h}{\longleftrightarrow} S(0, 2L) ] = 0 \big\}, 
\end{equation}
where the event $\{ B(0,L) \stackrel{\geq h}{\longleftrightarrow} S(0, 2L) \}$ refers to the existence of a nearest-neighbor path in $E_\varphi^{\geq h}$ connecting $B(0,L)$, the ball of radius $L$ around $0$ in the $\ell^\infty$-norm, to $S(0, 2L)$, the $\ell^\infty$-sphere of radius $2L$ around $0$ (in fact \eqref{h_**} does not exactly coincide with (0.6) in \cite{RS}, which requires the relevant probability to decay at least polynomially in $L$; the two are in fact equivalent, and \eqref{h_**} can be further weakened, see \cite{RP}, Theorem 2.1). The definitions \eqref{h_*} and \eqref{h_**} immediately yield that $h_*(d) \leq h_{**}(d)$ for all $d \geq 3$. It is presently known that 
\begin{equation} \label{KNOWN_RESULTS}
 0 \leq h_{*}(d) \text{ and } h_{**}(d) < \infty, \text{ for all $d \geq 3$},
\end{equation}
which implies that percolation of $E_{\varphi}^{\geq h}$ exhibits a non-trivial phase transition (see \cite{BLM}, Corollary 2 for the former and \cite{RS}, Theorem 2.6 for the latter result in \eqref{KNOWN_RESULTS}; see also Theorem 3 in \cite{BLM} for a proof of $h_*(3) < \infty$). In particular, for all $h > h_*$, $E_{\varphi}^{\geq h}$ only contains finite clusters, and a (unique, see \cite{RS}, Remark 1.6) infinite cluster for all $h< h_*$. The parameter $h_{**}$ is an important quantity because it characterizes a \textit{strongly} subcritical regime. For all $h> h_{**}$ and $d \geq 3$, the probability $\mathbb{P}[0 \stackrel{\geq h }{\longleftrightarrow} S(0,L)]$ decays exponentially in $L$ as $L \to \infty$ (with logarithmic corrections when $d=3$), as follows from Theorem 2.1. in \cite{RP}. Moreover, as mentioned in the expository paragraph (see \cite{RS}, Theorem 3.3),
\begin{equation}\label{KNOWN_RESULTS2}
\text{ $h_*(d)$ is strictly positive when $d$ is large enough.}
\end{equation}
Our main goal in this paper is to prove the following asymptotic result concerning the critical density of this percolation model.
\begin{theorem} \label{T:ASYMPTOTICS_DENSITY}
\begin{equation} \label{EQ:ASYMPTOTICS_DENSITY}
\mathbb{P}[\varphi_0 \geq h_*(d)] =\frac{1}{d^{1+o(1)}}, \textnormal{ as $d \to \infty$}.
\end{equation}
\end{theorem}
\noindent In fact, Theorem \ref{T:ASYMPTOTICS_DENSITY} will be an easy consequence of the following two results regarding the principal asymptotics of the critical parameters $h_*$ and $h_{**}$ in high dimension. For future reference, we let
\begin{equation} \label{EQ:h_as}
h_{\text{as}} (d) = \sqrt{2 g(0) \log d},
\end{equation}
where $g(0)$ refers to the Green function at the origin, cf. \eqref{GreenFunction} below. We will show the following.
\begin{theorem} \label{T:ASYMPTOTICS_UB} (Upper bound) 
\begin{equation}
\limsup_{d \to \infty} h_{**}(d)/ h_{\textnormal{as}} (d) \leq 1. \label{EQ:ASYMPTOTICS_UB} 
\end{equation}
\end{theorem}
\begin{theorem} \label{T:ASYMPTOTICS_LB} (Lower bound) 
\begin{equation}
\liminf_{d\to \infty} h_{*}(d)/ h_{\textnormal{as}} (d) \geq 1. \label{EQ:ASYMPTOTICS_LB}
\end{equation}
Moreover, for all $\varepsilon > 0$, there exists a finite positive constant $c(\varepsilon)$ such that for all $d \geq c(\varepsilon)$,
\begin{equation} \label{EQ:ASYMPTOTICS_LB_COR}
\mathbb{P}[E_\varphi^{\geq h_{\textnormal{as}}(1-\varepsilon)} \cap ( \mathbb{H} + \mathbb{Z}^2) \textnormal{ contains an infinite cluster}]=1,
\end{equation}
where $ \mathbb{H} \stackrel{\textnormal{def.}}{=} \{ 0,1\}^d$ and $\mathbb{Z}^2$ is viewed as the subset $\mathbb{Z}^2 \times \{ 0 \}^{d-2}$ of $\mathbb{Z}^d$.
\end{theorem}
\noindent(N.B.: the presence of the factor $\sqrt{g(0)}$ in the definition of $h_{\text{as}} (d) $ is for esthetic purposes only, since $g(0) \to 1$ as $d\to \infty$, cf. \eqref{g(0)} below).

Before describing our methods, we make a few remarks concerning these results and the heuristics lurking behind them. As alluded to above, a recurrent theme in the proofs below will be that, in high dimensions, the free field at \textit{small} scales (at which the geometry felt by the random walk is roughly tree-like) can be viewed as a \textit{perturbation} of an independent Gaussian field, and a considerable effort will go into understanding the precise effect of this perturbation on the connectivity of the level sets around the asymptotic value $h_{\text{as}}$ (see the comments below, and the proofs of Theorems \ref{T:LOCAL_CB} and \ref{P:SUBST_COMP}). This perturbative behavior was already hinted at in the proof of \eqref{KNOWN_RESULTS2} in \cite{RS}, which was based on a decomposition of the covariance $g(x,y)$ for $x, y$ belonging to a (lower-dimensional) subspace $\mathbb{Z}^{d'}$, with $d' \ll d$, into the sum of a dominant \textit{diagonal} part with entries close to $1$ (thus inducing a field of independent Gaussians) and a ``small noise'' (cf. Lemma 3.1. of \cite{RS} for a precise statement). Admitting this (local) resemblance of $\varphi$ to an independent field, it is reasonable to compare \eqref{EQ:ASYMPTOTICS_DENSITY} to $p_c^{\text{site}}(\mathbb{Z}^d)$, the critical parameter for Bernoulli site percolation on $\mathbb{Z}^d$, which is known to be asymptotically equal to $1/2d$ as $d \to \infty$, see \cite{ABS}, \cite{BH}, \cite{Go}, \cite{Kes}.

Regarding Theorems \ref{T:ASYMPTOTICS_UB} and \ref{T:ASYMPTOTICS_LB}, using $h_* \leq h_{**}$, \eqref{EQ:ASYMPTOTICS_UB} and \eqref{EQ:ASYMPTOTICS_LB} imply that
\begin{equation*}
h_{*}(d) \sim h_{**}(d) \ (\sim h_{\text{as}} (d)), \text{ as $d \to \infty$}
\end{equation*}
(we write $f(x)\sim g(x)$ as $x\to a$ if $\lim_{x \to a} f(x)/g(x)=1$). 
It is at present an important unresolved question whether both critical parameters are actually equal (in any dimension).

\bigskip

We now comment on the proofs. The proof of the upper bound follows a strategy inspired by that used in \cite{RS} to prove finiteness of $h_{**}(d)$, for arbitrary, but \textit{fixed} dimension $d$ (see also \cite{R}, \cite{SS}, \cite{S1}, \cite{S3}, \cite{SUB} for similar ideas in the context of random interlacements). In particular, we use a variant of the renormalization scheme developed therein. However, the present task requires a scheme which works ``uniformly in $d$'' as $d$ becomes large. Using careful estimates on the behavior of the Green function of simple random walk on a high-dimensional lattice developed in \cite{SLB, SUB}, we obtain a ``decoupling inequality,'' which enables us to propagate bounds on the relevant crossing events in $E_\varphi^{ \geq h}$ (cf. the definition \eqref{h_**} of $h_{**}$), where $h=h_\text{as}(1 + \varepsilon)$ for some $\varepsilon>0$, at small scale (the so-called seed estimates) to controls of such crossing probabilities at arbitrarily large scale in $E_\varphi^{ \geq h +\varepsilon}$ (this is in fact much more than we need). A significant part of the problem is to produce sufficiently sharp seed estimates, in order to initiate the renormalization, see Remark \ref{R:renorm}, 2) below (note that, in contrast to the proof of the finiteness of $h_{**}(d)$ in \cite{RS}, which allowed one to look for a corresponding regime at arbitrary large $h$, we are now constrained to remain in the vicinity of $h_{\text{as}}$).

Obtaining the desired bounds at small scales involves controlling the probability to see a crossing in $E_\varphi^{ \geq h}$ from a given point $x$ to the boundary of the $\ell^1$-ball centered at $x$ of radius $R=c(\varepsilon)d$, see \eqref{EQ:B_n,x} and Theorem \ref{T:LOCAL_CB} below (we emphasize that the use of the $\ell^1$-norm is essential here, as this distance controls the short-range behavior of $g(\cdot)$ in high dimensions, where the tree-like nature of the lattice manifests itself, cf. \eqref{return_probas1}). This is the first instance where the aforementioned (local) perturbative expansion comes into play. In rough terms, the domain Markov property for the free field $\varphi$ allows us to ``discover'' it along any given path (in the present case, one joining $x$ to $S_1(x,R)$) ``dynamically,'' starting from a suitable \textit{independent} Gaussian field $\psi$, and introducing the required dependence at each step. Specifically, if $K= \{x_1,\dots,x_n\} \subset \mathbb{Z}^d$, $n \geq 1$, denotes the trace of the path in question, we represent
\begin{equation} \label{EQ:INTRO_PERT}
\varphi_{x_k} \; \text{``$=$''} \;   \psi_{x_k} + \text{error}(\psi_{x_1}, \dots, \psi_{x_{k-1}}  ), \text{ for all $1 \leq k \leq n$},
\end{equation}
where the error term is a suitable linear combination of $\psi_{x_1}, \dots, \psi_{x_{k-1}}$, cf. Lemma \ref{L:conds_exps2} below for a precise statement. This procedure enables us to ``pass'' from $\varphi$ to the independent field $\psi$, provided we carefully keep track of the error we make in each step.

For the lower bound, Theorem \ref{T:ASYMPTOTICS_LB}, our method mimics in its broad lines the approach of \cite{Go} and \cite{ABS}, Section 4, to the corresponding problem in Bernoulli (bond and site) percolation on $\mathbb{Z}^d$, involving ideas of \cite{AKS} from hypercube percolation (see also \cite{Kes} for a completely different solution, and \cite{SLB} for a corresponding result in interlacement percolation, following a similar spirit). The proof essentially comprises two parts, which we briefly detail. 

In the first part, we show the following finite-size criterion. Suppose we partition the set $ \mathbb{H} + \mathbb{Z}^2$ into translates of the hypercube.  Roughly speaking, we show in Theorem \ref{T:LOCAL_GLOBAL} below that if $ E_{\varphi}^{\geq h}$, with $h = h_{\text{as}}(d)(1 - \varepsilon)$, $\varepsilon > 0$,
\begin{itemize}[noitemsep]
\item[i)] contains a ``giant'' connected component in each of the sets $\mathbb{H}$ and its four neighboring translates, and
\item[ii)] all these components are connected in $ E_{\varphi}^{\geq h}$ within the union of $\mathbb{H}$ and its four neighboring translates,
\end{itemize}
with sufficiently high probability, then $ E_{\varphi}^{\geq h- \varepsilon}$ percolates (the ``additive'' sprinkling in $h$ is more than enough for the sake of proving \eqref{EQ:ASYMPTOTICS_LB}). Due to the long-range dependence, this reduction step does not follow from standard stochastic domination arguments (see for example \cite{Scho}), and Theorem \ref{T:LOCAL_GLOBAL} is established by means of a (two-dimensional) renormalization argument.

In the second (and more difficult) part of the proof of Theorem \ref{T:ASYMPTOTICS_LB}, we show that this criterion actually holds. In the independent setting of  \cite{Go} and \cite{ABS}, the validity of such a criterion is essentially guaranteed by the analysis in \cite{AKS} of (independent) percolation in the hypercube. Our procedure essentially comprises two steps. First, drawing inspiration from \nolinebreak \cite{AKS}, cf. in particular Lemma 1 therein, we grow \textit{substantial} connected components in $E_{\varphi}^{\geq h} \cap \mathbb{H}$ (with $h = h_{\text{as}}(d)(1 - \varepsilon)$), which have cardinality growing polynomially in $d$. To achieve this, we embed into $\mathbb{H}$ a deterministic $r(\varepsilon,d)$-regular tree $\mathbb{T}$, rooted at $0$, with $r(\varepsilon, d)$ comparable to $d$ for every $\varepsilon > 0$, of depth depending on $\varepsilon$ only. Here again, the ``perturbative representation'' of \eqref{EQ:INTRO_PERT} crucially enters in allowing us to compare $E_{\varphi}^{\geq h_{\text{as}}(d) (1 - \varepsilon)} \cap \mathbb{T}$ to a certain (supercritical) Galton-Watson process on the same tree in order to derive a meaningful lower bound for the probability to see a substantial component at the origin. The precise statement is the object of Theorem \ref{P:SUBST_COMP}, and an easy consequence is that most vertices in $\mathbb{H}$ are in fact either neighboring or contained in such a substantial component, see Corollary \ref{C:number_subst_comps}.

The second step then consists of gluing all substantial components within the hypercube to a giant one (i.e. a connected component whose closure in $\mathbb{H}$ contains at least $(1-d^{-2})2^d$ points), and subsequently connecting neighboring giant components, which is achieved in Theorem \nolinebreak \ref{T:giant_comp} by two successive sprinkling operations. Quantifying that it is highly unlikely for the substantial components not to merge after sprinkling requires isoperimetric considerations in $\mathbb{H}$ (as were used in the proofs of \cite{AKS}, Theorem 1 and \cite{SLB}, Theorem 4.2). However, the (wealth of) edges which are hereby deduced to be pivotal for a giant component to emerge or not, are not necessarily ``well spread-out'' within $\mathbb{H}$ and might therefore influence each other rather strongly. A slightly careful bookkeeping of this mutual influence is required in order to show that it is too costly for many pivotal sites to remain ``closed'' (cf. the proof of Lemma \ref{L:giant_event_H}). This then completes the proof of the lower bound \eqref{EQ:ASYMPTOTICS_LB}.

\bigskip

Let us now describe the organization of this article. In Section \ref{NOTATION}, we introduce some notation and review a few known results concerning simple random walk on a high-dimensional lattice and the Gaussian free field. Section \ref{S:UPPER_BOUND} is devoted to the proof of Theorem \ref{T:ASYMPTOTICS_UB}. Subsection \nolinebreak \ref{S:UB_RS} introduces the renormalization scheme. Its main result is Proposition \ref{P:DEC_INEQ}, which entails the induction step (one-step renormalization). Subsection \ref{S:UB_LCB} contains the local estimates on the connectivity of $E_\varphi^{\geq h_{\text{as}}(1+ \varepsilon)}$, $\varepsilon > 0$. The main result, Theorem \nolinebreak \ref{T:LOCAL_CB}, is of independent interest, but it is needed crucially to establish the required seed estimate which enables us to trigger the renormalization scheme. Finally, in Subsection \ref{S:DENOUEMENT}, we put these two ingredients together to complete the proof of \eqref{EQ:ASYMPTOTICS_UB}. The proof of the lower bound \eqref{EQ:ASYMPTOTICS_LB} is the object of Section \nolinebreak \ref{S:LOWER_BOUND}. First, we show the abovementioned finite-size criterion, which requires another (simpler) renormalization scheme. The main result is Theorem \ref{T:LOCAL_GLOBAL}, which can be found in Subsection \ref{S:FINITE_SIZE_CRITERION}. Having established this reduction step, we prove that most points in the hypercube are either contained or neighboring a substantial component in the level set of interest. This is done in Subsection \nolinebreak \ref{S:SUBST_COMP}, see in particular Theorem \ref{P:SUBST_COMP} and Corollary \ref{C:number_subst_comps} therein. Finally, Theorem \ref{T:giant_comp} in Subsection \nolinebreak\ref{S:CONNECTING_SUBST_COMPS} connects these substantial components, thereby completing the proof of the criterion. The lower bound \eqref{EQ:ASYMPTOTICS_LB} then follows readily by collecting the pieces, and Subsection \nolinebreak \ref{S:CONNECTING_SUBST_COMPS} also contains the proof of Theorem \ref{T:ASYMPTOTICS_DENSITY}, which follows straightforwardly from \eqref{EQ:ASYMPTOTICS_UB} and \eqref{EQ:ASYMPTOTICS_LB}.

\bigskip

We conclude this introduction with a remark concerning our convention regarding constants: we denote by $c,c',\dots$ positive constants with values changing from place to place. Numbered constants $c_0,c_1,\dots$ and $c_0', c_1', \dots$ are defined at the place they first occur within the text and remain fixed from then on until the end of the article. All constants are numerical, and their dependence on any additional parameter, including, most importantly, the dimension $d$, will always appear in the notation. The only exceptions to this rule are the last two Sections \ref{S:SUBST_COMP} and \ref{S:CONNECTING_SUBST_COMPS}, in which constants may implicitly depend on a parameter $\varepsilon > 0$. The Reader will be reminded of this exception in due time.

\section{Notation and useful results} \label{NOTATION}

In this section, we introduce some notation to be used in the sequel, collect some important estimates related to simple random walk on a high-dimensional lattice, and review a few useful facts concerning the Gaussian free field.

We denote by $\mathbb{N}= \{0,1,2,\dots\}$ the set of natural numbers, and by $\mathbb{Z}$ the set of integers. We write $\mathbb{R}$ for the set of real numbers, abbreviate $x \wedge y = \min \{x,y\}$ and $x \lor y = \max\{ x,y\}$ for any two numbers $x,y \in \mathbb{R}$, denote by $\lfloor x \rfloor = \max \{ n \in \mathbb{N}  ; \, n \leq x \}$ the integer part of $x$, for any $x \geq 0$, and let $\lceil x \rceil = \min \{ n \in \mathbb{N} \; \, n \geq x \}$. We consider the lattice $\mathbb{Z}^d$, and tacitly assume that $d \geq 3$. Given a subset $K$ of $\mathbb{Z}^d$, $K^c = \mathbb{Z}^d \setminus K$ stands for the complement of $K$ in $\mathbb{Z}^d$, and $\vert K \vert$ for the cardinality of $K$. Moreover, in writing $K\subset \subset \mathbb{Z}^d$, we mean that $K$ is a finite subset of $\mathbb{Z}^d$. Finally, we denote by $K+K' = \{x+ y \, ; \, x \in K, y \in K' \}$ the Minkowski sum of arbitrary sets $K, K' \subset \mathbb{Z}^d$.

On $\mathbb{Z}^d$, we respectively denote by $\vert \cdot \vert_1$, $\vert \cdot \vert_2$ and $\vert \cdot \vert_\infty$ the $\ell^1$, Euclidean and $\ell^\infty$-norms. The three norms are equivalent, and satisfy the relations
\begin{equation} \label{norm_equivalence}
|\cdot|_2 \leq |\cdot|_1 \leq \sqrt{d} |\cdot|_2, \text{ and }  |\cdot|_\infty \leq |\cdot|_2 \leq \sqrt{d} |\cdot|_\infty, \text{ for all $d \, (\geq 1)$. }
\end{equation}
For any $x \in \mathbb{Z}^d$, $r \geq 0$, and $p=1,2, \infty$, we let $B_p(x,r) = \{ y \in \mathbb{Z}^d  ; \ \vert y-x \vert_p \leq r \}$ and $S_p(x,r) = \{ y \in \mathbb{Z}^d  ; \ \vert y-x \vert_p = r \}$ stand for the the (closed) $\ell^p$-ball and $\ell^p$-sphere of radius $r$ centered at $x$. We also note for later purposes the following bound on the cardinality of a $d$-dimensional $\ell^1$-sphere,
\begin{equation} \label{S_1_bound}
|S_1(0,n)| \leq e^{n+2d}, \text{ for all $n \geq 0$, $d \geq 3$,}
\end{equation}
which is easily computed by considering the generating function of $|S_1(0,n)|$, see for example \cite{SUB}, Lemma 3.2 (i) for a proof. Given two arbitrary sets $K, K' \subset \mathbb{Z}^d$, we define their $\ell^p$-distance as $d_p(K,K')= \inf \{ |x-y|_p  ;  \; x \in K, \; y \in K' \}$, for $p=1,2,\infty$, and simply write $d_p(x,K')$ when $K = \{ x \}$ is a singleton.  Moreover, we denote the interior ($\ell^1$-)boundary of $K$ as $\partial_{\text{int}} K = \{ x \in K ; \; \exists y \in K^c, \; \vert y-x \vert_1 =1 \}$, the outer boundary of $K$ as $\partial K = \partial_{\text{int}} (K^c)$, and write $\overline{K}= K \cup \partial_{\text{out}}K$ for the ($\ell^1$-)closure of $K$ and $\overline{K}^{K'}= \overline{K} \cap K'$ for the relative closure of $K$ in $K'$.

We now introduce the discrete-time simple random walk on $\mathbb{Z}^d$. We endow the lattice $\mathbb{Z}^d$ with the usual nearest-neighbor graph structure, and we will frequently use $x \sim y$ instead of $|x-y|_1=1$ to denote two neighboring vertices  $x,y \in \mathbb{Z}^d$. Moreover, $x,y \in \mathbb{Z}^d$ will be called  $*$-nearest neighbors if $|x-y|_\infty=1$. A $\mathbb{Z}^d$-valued nearest-neighbor path is a (finite or infinite) sequence $(x_n)_{n}$ of vertices in $ \mathbb{Z}^d$ satisfying $x_{n+1} \sim x_n$ for all $n \geq 0$. We define its length as the number of edges it traverses. When its length is infinite, we will often use the term \textit{trajectory} instead of \textit{path}. A $*$-path is defined accordingly.  Let $W$ denote the space of nearest-neighbor trajectories, and let $\mathcal{W}$, $(X_n)_{n \geq 0}$ and $(\theta_n)_{n\geq 0}$, stand for the canonical $\sigma$-algebra, the canonical process and the canonical shifts on $W$, respectively. We write $P_x$ for the canonical law of the walk starting at $x \in \mathbb{Z}^d$ and $E_x$ for the corresponding expectation. We denote by $g(\cdot, \cdot)$ the Green function
of the walk, i.e.
\begin{equation}\label{GreenFunction}
g(x,y) = \sum_{n \geq 0} P_x [X_n = y], \text{ for } x,y \in \mathbb{Z}^d,
\end{equation}
which is finite (since $d \geq3$) and symmetric. Moreover, $g(x,y)= g(x-y,0) \stackrel{\text{def.}}{=} g(x-y)$ due to translation invariance. We further recall that (see for example \cite{L}, Theorem 1.5.4)
\begin{equation}
g(x) \sim  c(d) |x|_2^{2-d}, \text{ as } |x|_2 \to \infty, \text{ for all $d \geq 3$}   \label{1.15}. 
\end{equation}
Given $K \subset \mathbb{Z}^d$, we denote the entrance time in $K$ by $H_K = \inf \{ n\geq 0 ; X_n \in K \}$ and the hitting time of $K$ by $\widetilde{H}_K = \inf \{ n\geq 1 ; X_n \in K \}$. This allows us to define the Green function $g_{K} (\cdot, \cdot)$ killed outside $K$ as
\begin{equation}\label{GreenFunctionSubK}
g_{K} (x,y) = \sum_{n \geq 0} P_x [X_n = y, \ n < H_{K^c}],  \text{ for } x,y \in \mathbb{Z}^d,
\end{equation}
which is symmetric and vanishes if $x \notin K$ or $y \notin K$. The relation between $g$ and $g_{K}$ for any $K \subset \mathbb{Z}^d$ is given by the following formula, the proof of which is a mere application of the strong Markov property (at time $H_{K^c}$),
\begin{equation} \label{G-GsubK}
g(x,y) = g_{K} (x, y) + E_x [ H_{K^c} < \infty , \, g(X_{H_{K^c}}, y) ], \text{ for } x,y \in \mathbb{Z}^d.
\end{equation}
We now turn to a few aspects of potential theory associated to simple random walk. For any finite subset $K$ of $\mathbb{Z}^d$, we write 
\begin{equation} \label{1.10}
e_{K} (x) = P_x [\widetilde{H}_K = \infty], \text{ $x \in K$},
\end{equation}
for the equilibrium measure (or escape probability) of $K$, and
\begin{equation}\label{1.11}
\text{cap}(K) = \sum_{x \in K} e_{K} (x)
\end{equation}
for its capacity. It immediately follows from the definitions \eqref{1.10} and \eqref{1.11} that the capacity is subadditive, i.e. 
\begin{equation} \label{1.12}
\text{cap}(K\cup K') \leq \text{cap}(K) + \text{cap}(K'), \text{ for all } K,K' \subset \subset \mathbb{Z}^d,
\end{equation}
and one also easily infers that it is monotonous, i.e. that
\begin{equation}\label{1.12bis}
\text{cap}(K) \leq \text{cap}(K'), \text{ for all } K \subseteq K' \subset \subset \mathbb{Z}^d.
\end{equation}
The latter follows e.g. by observing that $\text{cap} (K)= \lim_{L \to \infty}\sum_{y\in S_2(0,L)}P_y[\widetilde{H}_K < \widetilde{H}_{S_2(0,L)}]$, which follows from \eqref{1.11} by a straightforward reversibility argument (see for example \cite{L}, Proposition 2.2.1 (a) for more details). Moreover, the entrance probability in $K$ may be expressed in terms of $e_{K}(\cdot)$ as
\begin{equation} \label{1.13}
P_x [H_K < \infty] = \sum_{y \in K} g(x,y) \cdot e_{K}(y),
\end{equation}
which is a mere consequence of the simple Markov property (see for example \cite{Sp}, Theorem 25.1, p. 300). We collect some useful estimates on these quantities in high dimension, which will be used repeatedly in the sequel. We remind the Reader of our convention regarding constants at the end of the previous section.

\begin{lemma} $(d \geq 3)$
\begin{align} 
&g(0) = 1 + \frac{1}{2d} + o(d^{-1}), \text{ as $d\to \infty$}. \label{g(0)} \\
&g(x) \leq \bigg( \frac{c_0d}{|x|_1} \bigg)^{d/2-2}, \text{ for $d\geq 5$ and $x \in \mathbb{Z}^d \setminus \{ 0 \}$}. \label{c_0} \\ 
&\sup_{|x|_1=k} P_x[\widetilde{H}_{B_1(0,k)} < \infty] \leq \frac{c(k)}{d}, \text{ for $k \geq 0$}. \label{return_probas1}\\
&g(x) \leq \bigg( \frac{c \sqrt{d}}{ |x|_2} \bigg)^{d-2}, \text{ for } |x|_2 \geq d. \label{g1} \\ 
&\text{\textnormal{cap}}(B_2(0,L)) \leq \bigg(\frac{cL}{ \sqrt{d}} \bigg)^{d-2}, \text{ for } L \geq d. \label{g2} 
\end{align}
\end{lemma}
\begin{proof}
For \eqref{g(0)}, see \cite{Mo}, pp. 246--247; for \eqref{c_0} and \eqref{return_probas1}, see \cite{SLB}, Lemma 1.2; for \eqref{g1} and \eqref{g2}, see \cite{SUB}, Lemma 1.1 and $(1.22)$, respectively. 
\end{proof}

We now turn to the Gaussian free field on $\mathbb{Z}^d$, $d\geq 3$, as defined in \eqref{phi}, and introduce certain crossing events involving paths of high level. On the space $\{0,1 \}^{\mathbb{Z}^d}$ endowed with its canonical $\sigma$-algebra $\mathcal{A}$, let $\{ K \longleftrightarrow K' \}$ $(\in \mathcal{A})$, for $K,K' \subset \mathbb{Z}^d$, denote the event that there exists a nearest-neighbor path connecting $K$ and $K'$ along which the configuration has value \nolinebreak$1$. Letting $\Phi_h:\mathbb{R}^{\mathbb{Z}^d} \rightarrow \{0,1 \}^{\mathbb{Z}^d}$, $\varphi \mapsto (1\{\varphi_x  \geq h\})_{x \in \mathbb{Z}^d}$, for $h \in \mathbb{R}$, we introduce
\begin{equation}\label{general_crossing_events}
\{ K \stackrel{\geq h}{\longleftrightarrow} K' \} = \Phi_h^{-1} (\{ K \longleftrightarrow K' \}), \text{ for $K,K'\subset \mathbb{Z}^d$}
\end{equation}
(part of $\mathbb{R}^{\mathbb{Z}^d}$). In words, this is the event that $K$ and $K'$ are connected by a nearest-neighbor path of vertices in the level set $E_\varphi^{\geq h}$, cf. \eqref{Ephi}. Note that this event is increasing upon introducing on $\mathbb{R}^{\mathbb{Z}^d}$ the usual partial order (i.e. $f \leq f'$ if and only if $f_x \leq f'_x$ for all $x \in \mathbb{Z}^d$). Moreover, the probability $\mathbb{P}[K \stackrel{\geq h}{\longleftrightarrow} K']$ is a decreasing function of $h \in \mathbb{R}$. In general, we will use the notation
\begin{equation} \label{EQ:event_A_Ah}
A^h \stackrel{\text{def.}}{=} \Phi_h^{-1}(A), \text{ for all $A \in \mathcal{A}$ and $h \in \mathbb{R}$.}
\end{equation}
We will also need the notion of \textit{flipping} events. Let $Y_x$, $x \in \mathbb{Z}^d$, denote the canonical coordinates on $\{0,1 \}^{\mathbb{Z}^d}$. One defines the inversion map $\iota : \{0,1 \}^{\mathbb{Z}^d} \rightarrow \{0,1 \}^{\mathbb{Z}^d}$ such that $Y_x \circ \iota  = 1- Y_x$, for all $x \in \mathbb{Z}^d$. 
Given some event $A \subset \{0,1 \}^{\mathbb{Z}^d}$, let $\overline{A} = \iota^{-1}(A)= \iota(A)$, which will be referred to as \textit{flipped} event. Note that $\overline{A}$ is decreasing whenever $A$ is increasing. Moreover, observing that $(-\varphi_x)_{x \in \mathbb{Z}^d}$ has the same law as $(\varphi_x)_{x \in \mathbb{Z}^d}$ under $\mathbb{P}$, we obtain,  for all $A \in \mathcal{A}$ and $h \in \mathbb{R}$,
 \begin{equation} \label{di7}
\begin{split}
\mathbb{P} [A^h ] &= \mathbb{P} [ ( 1\{ \varphi_x \geq h \} )_{x \in \mathbb{Z}^d} \in A] = \mathbb{P} [ ( 1\{ - \varphi_x \geq h \} )_{x \in \mathbb{Z}^d} \in A] \\
&=  \mathbb{P} [ ( 1\{ \varphi_x \geq -h \} )_{x \in \mathbb{Z}^d} \in \overline{A}] = \mathbb{P} [\overline{A}^{-h} ]. 
\end{split}
\end{equation}

Next, we collect some classical results concerning the maximum of the Gaussian free field in a finite set. Let $\emptyset \neq K \subset \subset \mathbb{Z}^d$. First, note that for all $d\geq 3$,
\begin{equation} \label{EQ:max_phi_exp}
\mathbb{E}\big[  \text{max}_{_{ x\in K}} \varphi_x \big] \leq \sqrt{ 2 g(0)  \log|K|},
\end{equation}
see for example \cite{Ta}, Proposition 1.1.3 (see also Theorem 1.3.3 in \cite{AT} for a much more general result). Moreover, setting 
\begin{equation} \label{EQ:c_1}
c_1 = ( \sup_{d \geq 3} 2 g(0))^{1/2},
\end{equation}
which is finite due to \eqref{g(0)}, the Borell-Tsirelson-Ibragimov-Sudakov (BTIS) inequality (see for example \cite{AT}, Theorem 2.1.1), together with \eqref{EQ:max_phi_exp}, yields the tail estimate, for all $d\geq 3$,
\begin{equation} \label{EQ:BTIS}
\mathbb{P} \big[  \text{max}_{_{x \in K}} \varphi_x >   c_1 \alpha  \big] \leq  e^{-( \alpha - \sqrt{  \log|K|} )^2},  \text{ if }  \alpha > \sqrt{  \log|K|}.
\end{equation}
We also recall the following elementary tail estimate of the normal distribution. For $\xi \sim \mathcal{N}(0, \sigma^2)$ (see for example \cite{AT}, Ch. 2, p. 49),
\begin{equation}\label{GRV_basic_estimate}
\Big( \frac{1}{h} - \frac{1}{h^3}\Big) e^{- h^2 /2 } \leq \sqrt{2\pi} \cdot P[\xi> \sigma h] \leq \frac{1}{h} e^{- h^2 /2 }, \text{ for all $h>0$}.
\end{equation}

We proceed with a classical fact concerning conditional distributions for the Gaussian free field on $\mathbb{Z}^d$, the proof of which can be found in  \cite{RS} (see Lemma 1.2 therein).

\begin{lemma} \label{L:conds_exps} $( d \geq 3,$  $\emptyset \neq K \subset \subset \mathbb{Z}^d)$ 

\medskip
\noindent Let
\begin{equation} \label{EQ:p_x,y^K}
p_{x , y}^K = P_x [H_{K} < \infty , X_{H_{K}} = y], \text{ for $x \in \mathbb{Z}^d$, $y \in K$},
\end{equation}
and
\begin{equation} \label{mu}
\mu_x^K = E_x [H_{K} < \infty , \varphi_{X_{H_{K}}}] = \sum_{y \in K } p_{x, y}^K \cdot \varphi_y,  \text{ for $x \in \mathbb{Z}^d$,}
\end{equation}
which is $\sigma(\varphi_x ; \; x \in K)$-measurable, and define $(\widetilde{\varphi}_x^K)_{x \in \mathbb{Z}^d}$ by
\begin{equation} \label{mudecomp}
\varphi_x = \widetilde{\varphi}_x^K + \mu_x^K, \text{ for $x \in \mathbb{Z}^d.$}
\end{equation}
Then, under $\mathbb{P}$,
\begin{equation} \label{ind+cond_exps}
\begin{split}
&(\widetilde{\varphi}^K_x)_{x \in \mathbb{Z}^d}  \text{ is a centered Gaussian field, independent from } \\
&\text{$\sigma(\varphi_x ; \; x \in K)$, with covariances $\mathbb{E}[\widetilde{\varphi}_x^K \widetilde{\varphi}_y^K]  = g_{K^c}(x,y)$.}
\end{split}
\end{equation}
(In particular, $\widetilde{\varphi}_x^K = 0$, $\mathbb{P}$-a.s. whenever $x \in K$).
\end{lemma}

Lemma \ref{L:conds_exps} yields the following choice of regular conditional distributions for $(\varphi_x)_{x \in \mathbb{Z}^d}$ conditioned on the variables $(\varphi_x)_{x \in K}$, which will prove very useful in several instances below. Namely, $\mathbb{P}$-a.s.,
\begin{equation} \label{phi_cond_exps}
\mathbb{P} \big[ (\varphi_x)_{x \in \mathbb{Z}^d} \in \cdot \  \big\vert (\varphi_x)_{x \in K} \big] = \widetilde{\mathbb{P}} \big[ (\widetilde{\varphi}_x^K + \mu_x^K)_{x \in \mathbb{Z}^d} \in \cdot \ \big], 
\end{equation}
where $\mu_x^K$, $x \in \mathbb{Z}^d$, is given by \eqref{mu}, and $(\widetilde{\varphi}_x^K)_{x\in \mathbb{Z}^d}$ is a centered Gaussian field under $\widetilde{\mathbb{P}}$ with covariance structure $g_{K^c}(\cdot,\cdot)$, independent of $\varphi_x$, $x\in K$.

The following is an immediate corollary of Lemma \ref{L:conds_exps}, which provides a way to construct $(\varphi_x)_{x \in K}$ inductively from a family of independent Gaussian variables. 

\begin{lemma} \label{L:conds_exps2} $(d \geq 3, \;  \emptyset \neq K \subset \subset \mathbb{Z}^d)$

\medskip
\noindent Let $x_1,x_2,\dots,x_{|K|}$ be an enumeration of the elements of $K$, and $K_n = \{x_i; \, 1 \leq i <n\}$, for $1 \leq  n \leq |K|$ (in particular, $K_1 = \emptyset$). Let $\psi_x$, $x \in K$, be an independent family of Gaussian random variables, with
\begin{equation}\label{L:conds_exps2_psi}
\psi_{x_n} \sim \mathcal{N} (0,g_{K_{n}^c}(x_{n},x_{n})), \text{ for $ 1\leq n \leq |K|$}, 
\end{equation}
and define recursively (with a slight abuse of notation)
\begin{equation} \label{L:conds_exps2_phi}
\begin{split}
&\varphi_{x_{1}}= \psi_{x_1} \\
&\varphi_{x_{n}} = \psi_{x_{n}} + \mu_{x_{n}}^{K_{n}}, \text{ for $ 1 < n \leq |K|$.} 
\end{split}
\end{equation}
Then $(\varphi_x)_{x\in K}$ has the law of Gaussian free field restricted to $K$.
\end{lemma}

\begin{proof}
We proceed by induction over $n$. For $n=1$, \eqref{L:conds_exps2_psi} implies that $\varphi_{x_1}$ is indeed a centered Gaussian variable with $\mathbb{E}[\varphi_{x_{1}}^2] = \mathbb{E}[\psi_{x_{1}}^2] = g(0)$. Suppose now that $(\varphi_{x_1}, \dots, \varphi_{x_{n-1}} ) = (\varphi_x)_{x \in K_n}$ has the law of Gaussian free field restricted to $K_n$, for some $1 < n < |K|$. In particular, by \eqref{L:conds_exps2_phi}, $(\varphi_x)_{x \in K_n}$ is a (linear) map of $(\psi_x)_{x \in K_n}$, it is therefore independent of $\psi_{x_{n}}$. Thus, Lemma \ref{L:conds_exps} (with $K=K_{n}$) readily yields that $((\varphi_x)_{x \in K_n}, \varphi_{x_{n}})$, with $\varphi_{x_{n}}$ as defined in \eqref{L:conds_exps2_phi}, has the desired law.
\end{proof}

\section{Upper bound} \label{S:UPPER_BOUND}

In this section, we show Theorem \ref{T:ASYMPTOTICS_UB}. As explained in the Introduction, this includes setting up an appropriate renormalization scheme, which is done in Subsection \ref{S:UB_RS}, and deducing suitable recursive bounds for the probability of the relevant crossing events at different scales. Proposition \ref{P:DEC_INEQ} entails the induction step, which is then propagated inductively in Proposition \ref{P:DEC_INEQ_PROPAGATED} along any increasing sequence of levels $(h_n)_{n\geq 0}$ satisfying a mild growth condition (cf. Remark \ref{R:UB_FINAL}), provided a suitable bound for the probability of crossing events at the lowest scale holds. Obtaining the desired seed estimate requires a substantial amount of work (see also Remark \ref{R:renorm}, 2) below, which details this difficulty more quantitatively) and is the object of Subsection \ref{S:UB_LCB}. The main result is Theorem \ref{T:LOCAL_CB} therein. Finally, Subsection \ref{S:DENOUEMENT} brings together the two ingredients to complete the proof of the upper bound \eqref{EQ:ASYMPTOTICS_UB}. 

\subsection{Renormalization scheme} \label{S:UB_RS}
We start by developing a renormalization scheme, which, in its broad lines, is adapted from the one described in Section 2 of \cite{RS}. We start by introducing an integer parameter $N \geq 1$ and a geometrically increasing sequence of length scales
\begin{equation} \label{L_n}
L_{n} = l_0^n L_0, \text{ for all $n \geq 0$, satisfying $L_0 \geq d$, $l_0 \geq 20(\sqrt{d} + N)$},
\end{equation}
and corresponding renormalized lattices
\begin{equation} \label{LL_n}
\mathbb{L}_n = L_n \mathbb{Z}^d, \text{ so that $\mathbb{L}_{n+1} \subset \mathbb{L}_n \subset \dots \subset \mathbb{L}_0 \; (\subset \mathbb{Z}^d)$, for all $n \geq 0$.}
\end{equation}
We are interested in the crossing events (cf. \eqref{general_crossing_events} for notation)
\begin{equation} \label{eventA}
A_{n,x}^h= \{ B_{\infty}(x, L_n) \stackrel{\geq h}{\longleftrightarrow} \partial_{\text{int}} B_{n,x} \}, \text{ for $n \geq 0$, $x \in \mathbb{L}_n$ and $h \in \mathbb{R}$,} 
\end{equation}
where
\begin{equation} \label{EQ:B_n,x}
B_{n,x} = 
\begin{cases}
B_{\infty}(x, 3L_n), & \text{ if $n \geq 1$} \\
B_{\infty}(x, L_0) + B_1(0,NL_0), &\text{ if $n = 0$}
\end{cases}
, 
\end{equation}
for all $x \in \mathbb{L}_n$. The case $n=0$ requires special treatment due to the following competing interests. On the one hand, the seed estimate we will need to establish forces $B_{0,x}$, $x \in \mathbb{L}_0$, to be not too small (this is the reason for introducing the additional parameter $N$). On the other hand, $B_{0,x}$ should remain sufficiently ``invisible'' to a simple random walk started at $\ell^\infty$-distance of order $L_n$, for some $n \geq 1$, from $x$ (see the proof of Proposition \ref{P:DEC_INEQ} below).

Similarly to the approach taken in Section 2 of \cite{RS}, instead of looking directly at the whole crossing path above level $h$ needed for $A_{n,x}^h$ to occur, we will consider its projection onto $2^n$ small and ``well-separated'' sets at scale $n=0$ (translates of $B_{0,0}$). These sets will be indexed by the leaves of a binary tree of depth $n$. We now describe this ``graphical cantor set'' construction more precisely. We write $T^{(k)} = \{1,2 \}^k$ for all $k \geq 0$ (with the convention $\{1,2 \}^0 = \emptyset$), and $T_n = \bigcup_{0 \leq k \leq n}T^{(k)}$ for the canonical dyadic tree of depth $n$. For arbitrary $n \geq 0$, we call a map $\mathcal{T}: T_n \rightarrow \mathbb{Z}^d$ a \textit{proper embedding of $T_n$ in $\mathbb{Z}^d$ with root at $x \in \mathbb{L}_n$} if
\begin{equation} \label{embedding}
\begin{array}{ll}
\text{i)} & \mathcal{T}(\emptyset) =x, \\
\text{ii)} & \text{for all $ 0 \leq k < n$: if $m_1,m_2 \in T^{(k+1)}$ are the two descendants} \\
& \text{of $m \in T^{(k)}$, then $\mathcal{T}(m_1)  \in \mathbb{L}_{n-k-1} \cap S_{\infty}(\mathcal{T}(m),L_{n-k})$, and}  \\
& \text{$\mathcal{T}(m_2)  \in \mathbb{L}_{n-k-1} \cap S_\infty(\mathcal{T}(m), 2L_{n-k})$.}
\end{array}
\end{equation}
We denote by $\Lambda_{n,x}$ the set of proper embeddings of $T_n$ in $\mathbb{Z}^d$ with root at $x \in \mathbb{L}_n$, for $n \geq 0$. One easily infers that
\begin{equation} \label{Lambda_n,x}
|\Lambda_{n,x}| \leq ((c_2l_0)^{d-1})^2 \cdot ((c_2l_0)^{d-1})^{2^2} \cdots ((c_2l_0)^{d-1})^{2^n} \leq (c_2l_0)^{2(d-1) 2^n},
\end{equation}
for some constant $c_2 \geq 1$. Moreover, on account of \eqref{L_n}, \eqref{EQ:B_n,x} and \eqref{embedding}, if $\mathcal{T} \in \Lambda_{n,x}$ for some $n \geq 1$ and $x \in \mathbb{L}_n$, and if $m \in T^{(k)}$ for some $0 \leq k < n$ with descendants $m_1,m_2 \in T^{(k+1)}$, then
\begin{equation} \label{EQ:embedding_geometry}
B_{n - k-1, \mathcal{T}(m_i)} \subset B_{n- k, \mathcal{T}(m)}, \text{ for $i=1,2$}.
\end{equation}
By elementary geometric considerations, and using \eqref{EQ:embedding_geometry}, one then deduces that for all $n \geq 1$, $x \in \mathbb{L}_n$ and $h \in \mathbb{R}$,
\begin{equation*}
A_{n,x}^h \subseteq \bigcup_{\mathcal{T}\in \Lambda_{n,x}} A_{n-1,\mathcal{T}(1)}^h \cap A_{n-1,\mathcal{T}(2)}^h,
\end{equation*}
see Figure \ref{F:crossing_high_dim} below, and inductively that
\begin{equation} \label{cascading_events}
A_{n,x}^h \subseteq \bigcup_{\mathcal{T}\in \Lambda_{n,x}} A_{\mathcal{T}}^h, \text{ where } A_{\mathcal{T}}^h = \bigcap_{m \in T^{(n)}} A_{0,\mathcal{T}(m)}^h, \text{ for $n \geq 0$, $x \in \mathbb{L}_n$ and $h \in \mathbb{R}$}.
\end{equation}
Accordingly, we introduce the quantity
\begin{equation} \label{p_n}
p_n(h) = \sup_{\mathcal{T} \in \Lambda_{n,x}} \mathbb{P} 
[A_{\mathcal{T}}^h], \text{ for $n \geq 0$, $h \in \mathbb{R}$},
\end{equation}
which does not depend on $x \in \mathbb{L}_n$ due to translation invariance, and is a decreasing function of $h\in \mathbb{R}$. 
\begin{figure} [h!] 
\centering
\psfragscanon
\includegraphics[width=18cm]{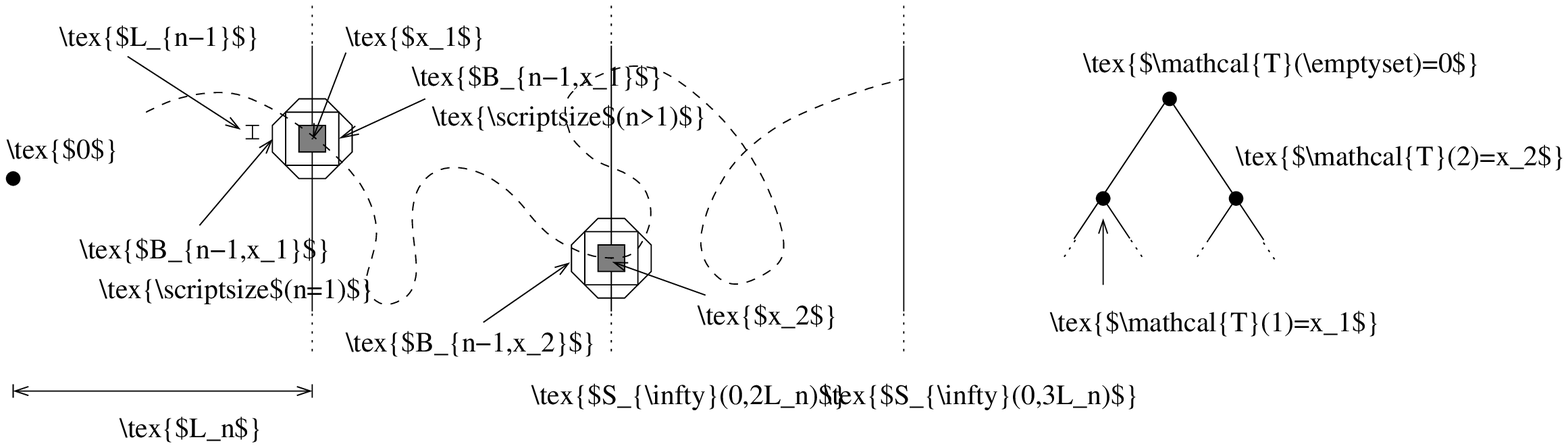}
\caption{Any nearest-neighbor path in $E_\varphi^{\geq h}$ (represented by the dashed line) connecting $B_{\infty}(0,L_n)$ to $S_\infty(0,3L_n)$, for some $n \geq 1$, must cross a box $B_{n-1,x_1}$, for some $x_1 \in S_\infty(0,L_n) \cap \mathbb{L}_{n-1}$, as well as $B_{n-1,x_2}$, for some $x_2 \in S_\infty(0,2L_n) \cap \mathbb{L}_{n-1}$, thus giving rise to the simultaneous occurrence of $A_{n-1,x_1}^h$ and $A_{n-1,x_2}^h$.}
\label{F:crossing_high_dim}
\end{figure}
We will later apply a union bound in \eqref{cascading_events} in order to bound $\mathbb{P}[A_{n,x}^h]$. Thus, estimates for $p_n(h)$ will have to be strong enough to overcome the combinatorial complexity $|\Lambda_{n,x}|$ coming from the number of trees one can choose. To begin with, the following proposition provides recursive bounds for $p_n(h_n)$, $n\geq 0$, along a suitable non-decreasing sequence $(h_n)_{n\geq 0}$. 

\begin{proposition} \label{P:DEC_INEQ} $(d\geq3$, $N \geq 1$, $L_0 \geq d$, $l_0 \geq 20 (\sqrt{d}+N))$

\medskip
\noindent There exists a constant $c_3$ such that, defining 
\begin{equation} \label{M}
m_n(d,L_0,N)=  \sqrt{ \log (2^n ((N+1) 3L_0)^d) },
\end{equation}
given any positive sequence $(\alpha_n)_{n \geq 0}$ satisfying
\begin{equation} \label{alpha_n_cond}
\alpha_n >  m_n(d,L_0,N),  \text{ for all } n \geq 0,
\end{equation}
and any increasing, real-valued sequence $(h_n)_{n \geq 0}$ satisfying
\begin{equation} \label{h_n_cond}
h_{n+1} \geq h_n +  \alpha_n \big(c_3 (\sqrt{d} + N) \big)^{d-2} \big(2 l_0^{-(d-2)} \big)^{n+1},  \text{ for all } n \geq 0,
\end{equation}
one has
\begin{equation} \label{dec_ineq_incr}
p_{n+1}(h_{n+1}) \leq p_n(h_n)^2 +   e^{-(\alpha_n -m_n(d,L_0,N))^2}, \ \text{for all } n \geq 0.
\end{equation}
\end{proposition}
The proof of Proposition \ref{P:DEC_INEQ} is similar to that of Proposition $2.2$ in \cite{RS} (see also Lemma 2.6 in \cite{DRS} and Proposition 4.1 in \cite{R}), with certain modifications. In particular, note that the constants appearing in these references all implicitly depend on $d$. In the present situation however, it is imperative to have a good control over the ``sprinkling'' condition for the sequence $(h_n)_{n\geq 0}$, cf. \eqref{h_n_cond}, in terms of $d$ as $d$ grows to infinity.

\begin{proof}
Let $\mathcal{T}$ be a proper embedding in $\Lambda_{n+1,0}$, for some $n \geq 0$. For $i=1,2$, we denote by $\mathcal{T}_i \in \Lambda_{n, \mathcal{T}(i)}$ the proper embedding of depth $n$ with root at $\mathcal{T}(i)$ ``induced'' by $\mathcal{T}$, i.e. determined by $\mathcal{T}_i(m) = \mathcal{T}(im)$, for all $m \in T_n$ (where $im \in T_{n+1}$ stands for the concatenation of $\{ i \} \in T^{(1)}$ and $m$), and define the sets
\begin{equation} \label{K_i_dec}
K_i= \bigcup_{m \in T^{(n)} } B_{0, \mathcal{T}_i(m)}, \text{ for $i=1,2$}.
\end{equation}
On account of \eqref{eventA} and \eqref{cascading_events}, we see that $A_{\mathcal{T}_i}^h \in \sigma(\varphi_x \, ; \, x\in K_i)$, for $i=1,2$. We introduce a parameter $\alpha >0$, set $\alpha' = c_1 \alpha$ (cf. \eqref{EQ:c_1} for the definition of $c_1$), and write, abbreviating $\text{max}_{_{K_1}} \varphi = \text{max}_{_{y\in K_1}} \varphi_y$,
\begin{equation} \label{P2.2.1}
\begin{split}
\mathbb{P}[A_{\mathcal{T}}^h]
&= \mathbb{P}[A_{\mathcal{T}_1}^h \cap A_{\mathcal{T}_2}^h] \\
&\leq \mathbb{P} [ A_{\mathcal{T}_1}^h \cap A_{\mathcal{T}_2}^h \cap \{ \text{max}_{_{K_1}} \varphi \leq \alpha' \} ]  +  \mathbb{P} [ \text{max}_{_{K_1}} \varphi > \alpha' ]  \\
& = \mathbb{E} [ 1_{A_{\mathcal{T}_1}^h} \cdot 1_{ \{ \text{max}_{_{K_1}} \varphi  \ \leq   \alpha' \} } \cdot \mathbb{P} [A_{\mathcal{T}_2}^h \, | \, (\varphi_{x})_{x\in K_1}] ] \ + \ \mathbb{P} [  \text{max}_{_{K_1}} \varphi > \alpha'  ]. 
\end{split}
\end{equation} 
By \eqref{phi_cond_exps}, the conditional probability appearing in the last line of \eqref{P2.2.1} can be rewritten as 
\begin{equation}\label{EQ:dec_ineq_cond_exp}
\mathbb{P} [A_{\mathcal{T}_2}^h \, | \, (\varphi_{x})_{x\in K_1}] = \widetilde{ \mathbb{P} } [ (1\{ \widetilde{\varphi}_x^{K_1} + \mu_x^{K_1} \geq h \})_{x\in \mathbb{Z}^d} \in A_{\mathcal{T}_2}], \text{ $\mathbb{P}$-a.s.},
\end{equation}
with $\mu_x^{K_1} = E_x [H_{K_1} < \infty , \varphi_{X_{H_{K_1}}}]$ and where $A_{\mathcal{T}_2} \in \mathcal{A}$, the canonical $\sigma$-algebra on $\{0,1 \}^{\mathbb{Z}^d}$, is such that $A_{\mathcal{T}_2}^h = \Phi_h^{-1}(A_{\mathcal{T}_2})$ (recall \eqref{EQ:event_A_Ah}). In particular, $A_{\mathcal{T}_2}$ is measurable with respect to the canonical coordinates in $K_2$. We will now estimate the random shift $\mu_x^{K_1} $, for $x \in K_2$, which will involve the bounds \eqref{g1}, \eqref{g2}. To this end, we first note that, by construction,
\begin{equation} \label{distK1_K2}
d_2(K_1,K_2) \geq   L_{n+1} / 2 \ (> d). 
\end{equation}
Indeed,
\begin{equation*}
\begin{array}{rcl}
\underset{y_i \in K_i, i=1,2}{\inf} |y_1-y_2|_2 \hspace{-1ex} &
\stackrel{ \eqref{embedding}, \eqref{EQ:embedding_geometry}}{\geq} & \hspace{-1ex} |\mathcal{T}(1)-\mathcal{T}(2)|_2 - \displaystyle\sum_{i=1,2} \displaystyle\sup_{y \in B_{n,\mathcal{T}(i)}} |y - \mathcal{T}(i)|_2 \\
&\stackrel{\eqref{norm_equivalence},\eqref{EQ:B_n,x}}{\geq} & \hspace{-1ex} L_{n+1} - 2( (\sqrt{d} +N)L_0 \vee \sqrt{d} \, 3L_n) \\
&\stackrel{\eqref{L_n}}{\geq} & \hspace{-1ex} (l_0 -  2(\sqrt{d}+N))L_n \wedge (l_0 - 6\sqrt{d}))L_n
\end{array}
\end{equation*}	
(in particular, we have used that $B_{0,x} \subset B_2(x, (\sqrt{d} + N)L_0)$ for all $x \in \mathbb{L}_0$ in the second line),
which, together with the constraints $L_0 \geq d$ and $l_0 \geq 20 (\sqrt{d}+N)$, immediately yields \eqref{distK1_K2}. Thus, for arbitrary $x \in K_2$ and $\alpha>0$, on the event $\{ \text{max}_{_{K_1}} \varphi \leq \alpha' \}$,
\begin{equation} \label{P2.2.4}
\begin{array}{rcl}
\mu_x^{K_1} \hspace{-1ex} &  \leq &  \hspace{-1ex} \alpha' \cdot P_x [H_{K_1} < \infty] \\
 \hspace{-1ex} & \stackrel{\eqref{1.13}}{\leq} & \hspace{-1ex}  \alpha' \cdot \text{cap}(K_1) \cdot \displaystyle\sup_{x \in K_2, \; y \in K_1} g(x,y)\\
 &\stackrel{\eqref{1.12}, \eqref{1.12bis}}{ \leq} & \hspace{-1ex} \alpha' 2^n \cdot \text{cap}(B_2(0,(\sqrt{d} + N)L_0))\cdot \displaystyle \sup_{x \in K_2, \; y \in K_1} g(x,y) \\
 & \stackrel{\eqref{g1}, \eqref{g2}}{\leq} & \hspace{-1ex} \alpha' 2^n (c(\sqrt{d} + N)L_0 / \sqrt{d})^{d-2}\cdot \big(c' \sqrt{d} /  l_0^{n+1} L_0 \big)^{d-2} \\
 &\leq &\hspace{-1ex} \alpha 2^n (c_3(\sqrt{d}+N))^{d-2} \cdot \big( l_0^{-(d-2)} \big)^{n+1} \stackrel{\text{def.}}{=} \gamma /2,
\end{array}
\end{equation}
where \eqref{g1} applies due to \eqref{distK1_K2} and \eqref{g2} because $(\sqrt{d} + N)L_0 > d$ by assumption, and where the last line in \eqref{P2.2.4} defines the constant $c_3$ appearing in the statement of Proposition \ref{P:DEC_INEQ} above. In particular, on the event $\{ \text{max}_{_{K_1}} \varphi \leq \alpha' \}$ and for any $x \in K_2$, the inequality $\widetilde{\varphi}_x^{K_1} + \mu_x^{K_1} \geq h$ implies 
\vspace{-0.8ex}
\begin{equation*}
\widetilde{\varphi}_x^{K_1} - \mu_x^{K_1} \geq h - 2 \mu_x^{K_1} \stackrel{\eqref{P2.2.4}}{\geq} h - \gamma.
\end{equation*}
Hence, on the event $\{ \text{max}_{_{K_1}} \varphi \leq \alpha' \}$, since $A_{\mathcal{T}_2}$ is increasing, \eqref{EQ:dec_ineq_cond_exp} yields
\begin{equation*} 
\begin{split}
\begin{array}{rcl}
\mathbb{P} [A_{\mathcal{T}_2}^h \, | \, (\varphi_{x})_{x\in K_1}]
& \leq & \hspace{-1ex} \widetilde{ \mathbb{P} } [ ( 1\{ \widetilde{\varphi}_x^{K_1} - \mu_x^{K_1} \geq h - \gamma \})_{x \in \mathbb{Z}^d} \in A_{\mathcal{T}_2} ] \\
& = & \hspace{-1ex} \widetilde{ \mathbb{P} } [ ( 1\{ -\widetilde{\varphi}_x^{K_1} - \mu_x^{K_1} \geq h - \gamma \})_{x \in \mathbb{Z}^d} \in A_{\mathcal{T}_2} ] \\
& = & \hspace{-1ex} \widetilde{ \mathbb{P} } [ ( 1\{ \widetilde{\varphi}_x^{K_1} + \mu_x^{K_1} <  \gamma - h\})_{x \in \mathbb{Z}^d} \in A_{\mathcal{T}_2} ] \\
& = & \hspace{-1ex} \widetilde{ \mathbb{P} } [ ( 1\{ \widetilde{\varphi}_x^{K_1} + \mu_x^{K_1} \geq \gamma - h \})_{x \in \mathbb{Z}^d} \in \overline{A}_{\mathcal{T}_2} ] \\
& \stackrel{\eqref{phi_cond_exps}}{=}  & \hspace{-1ex} \mathbb{P} [\overline{A}_{\mathcal{T}_2}^{ \gamma - h} \, | \, (\varphi_{x})_{x\in K_1} ],
\end{array}
\end{split}
\end{equation*}
where $\overline{A}_{\mathcal{T}_2}$ denotes the flipped event (recall the notation from above \eqref{di7}), and where we have used in the second line that $\widetilde{\varphi}^{K_1}$ and $-\widetilde{\varphi}^{K_1}$ have the same law under $\widetilde{\mathbb{P}}$, cf. \eqref{ind+cond_exps}, and that $\widetilde{\mathbb{P}}$ does not act on $\mu^{K_1}$. Inserting this bound into \eqref{P2.2.1}, we obtain
\begin{equation} \label{P2.2.7}
\begin{split}
\mathbb{P}[A_{\mathcal{T}}^h] &\leq \mathbb{P} [ A_{\mathcal{T}_1}^h \cap \overline{A}_{\mathcal{T}_2}^{ \gamma - h} ] +  \mathbb{P} [  \text{max}_{_{K_1}} \varphi > \alpha'] \\
&\leq \mathbb{P} [ A_{\mathcal{T}_1}^h] \cdot \mathbb{P} [A_{\mathcal{T}_2}^{ h - \gamma} ] +  \mathbb{P} [  \text{max}_{_{K_1}} \varphi > \alpha'],
\end{split}
\end{equation}
and the second step is due to \eqref{di7} and the FKG-inequality for the free field (see for example \cite{GHM}, Ch. 4), which applies and yields an upper bound, because the event $ A_{\mathcal{T}_1}^h$ is increasing and $ \overline{A}_{\mathcal{T}_2}^{ \gamma - h}$ is decreasing, cf. \eqref{eventA}, \eqref{cascading_events}.

It remains to bound the error term. Since $B_{0,0} \subset B_\infty(0,(N+1)L_0)$, we have by \eqref{K_i_dec} that $|K_1| \leq 2^n |B_\infty(0,(N+1)L_0)|< 2^n(3(N+1)L_0)^d$, and \eqref{EQ:BTIS}, \eqref{M} imply that (recall that $\alpha' = c_1 \alpha$)
\begin{equation*}
\mathbb{P} [  \text{max}_{_{K_1}} \varphi > \alpha'] \leq e^{-(\alpha- m_n(d,L_0,N))^2}, \text{ for all $\alpha$ satisfying \eqref{alpha_n_cond}}.
\end{equation*}
Substituting this into \eqref{P2.2.7} yields, for all $\alpha_n \stackrel{\text{def.}}{=} \alpha$ fulfilling \eqref{alpha_n_cond} and all $h' \geq h$,
\begin{equation*} 
\mathbb{P} [A_{\mathcal{T}}^{h'}] \ \leq \ \mathbb{P}[A_{\mathcal{T}}^h]  \leq \mathbb{P} [ A_{\mathcal{T}_1}^{h- \gamma} ] \cdot \mathbb{P} [ A_{\mathcal{T}_2}^{h - \gamma} ]  +  e^{-(\alpha_n - m_n(d,L_0,N))^2}.
\end{equation*}
The claim \eqref{dec_ineq_incr} now readily follows upon taking suprema over all $\mathcal{T} \in \Lambda_{n+1,0}$ on both sides, letting $h_n \stackrel{\text{def.}}{=} h - \gamma \in \mathbb{R}$ ($h$ was arbitrary), $h_{n+1}\stackrel{\text{def.}}{=}h'$, so that requiring $h_{n+1} = h' \geq h = h_n + \gamma$, by virtue of \eqref{P2.2.4}, is precisely the condition \eqref{h_n_cond}. This concludes the proof of Proposition \nolinebreak \ref{P:DEC_INEQ}.
\end{proof}

We will now propagate the bounds \eqref{dec_ineq_incr} inductively along a suitable sequence $(h_n)_{n \geq 0}$, and select to this end
\begin{equation} \label{alpha_n}
\alpha_n =  m_n(d,L_0,N) + 2^{(n+1)/2} \big(n^{1/2} + k_0^{1/2}\big), \text{ for } n \geq 0,
\end{equation}
for some parameter $k_0 > 0$ to be specified later, and with $m_n(d,L_0,N)$ as defined in \eqref{M}. In particular, note that \eqref{alpha_n} satisfies the condition \eqref{alpha_n_cond} for every choice of $k_0 >0$.

\begin{proposition} \label{P:DEC_INEQ_PROPAGATED} $(d\geq3$, $N \geq 1$, $L_0 \geq d$, $l_0 \geq 20 (\sqrt{d}+N))$

\medskip
\noindent  Assume $h_0 \in \mathbb{R}$ and $k_0 \geq (1-e^{-1})^{-1} \stackrel{\text{\textnormal{def.}}}{=}b$ are such that
\begin{equation} \label{p_0_cond}
p_0(h_0) \leq e^{-k_0}, 
\end{equation}
and let the sequence $(h_n)_{n \geq 0}$ satisfy \eqref{h_n_cond} with $(\alpha_n)_{n\geq0}$ as defined in \eqref{alpha_n}. Then,
\begin{equation} \label{p_n_bounds}
p_n(h_n) \leq e^{-(k_0 - b)2^n},  \text{ for all } n \geq 0. 
\end{equation}
\end{proposition}
On account of Proposition \ref{P:DEC_INEQ} and the choice of $\alpha_n$, $n \geq 0$, in \eqref{alpha_n}, (compare this to Proposition 2.2 and (2.51) in \cite{RS}, respectively), the proof of Proposition \ref{P:DEC_INEQ_PROPAGATED} is completely analogous to that of Proposition 2.4 in \cite{RS}. We therefore omit it.

\begin{remark} \label{R:renorm}

 1) Even though we will not need this below, note that the conclusions of Propositions \ref{P:DEC_INEQ} and \ref{P:DEC_INEQ_PROPAGATED} continue to hold for arbitrary increasing (seed) events $A_{0,x} \subset \{0,1 \}^{\mathbb{Z}^d}$, for $x \in \mathbb{L}_0$, which are measurable with respect to the canonical coordinates in $B_{0,x}$ as defined in \eqref{EQ:B_n,x}, upon letting $A_{0,x}^h = \Phi_h^{-1}(A_{0,x})$, for all $x \in \mathbb{L}_0$ and $h \in \mathbb{R}$, defining events $A_{\mathcal{T}}^h$, with $\mathcal{T}  \in \Lambda_{n,x}$ for arbitrary $n \geq 0$ and $ x \in \mathbb{L}_n$, as in \eqref{cascading_events}, and taking an additional supremum over $x \in \mathbb{L}_n$ in the definition \eqref{p_n} of $p_n(h)$ (no translation invariance required). Moreover, by symmetry, if instead all events $A_{0,x}$, $x \in \mathbb{L}_0$, are decreasing, under the assumptions of Proposition \ref{P:DEC_INEQ}, the conclusion \eqref{dec_ineq_incr} holds for the sequence $(-h_n)_{n \geq 0}$ in place of $(h_n)_{n \geq 0}$.

\bigskip

\noindent 2) The condition \eqref{p_0_cond} is quite strong. Indeed, by \eqref{cascading_events} and \eqref{p_n}, we have, for all $h \in \mathbb{R}$, $n \geq 0$ and $x\in \mathbb{L}_n$,
\begin{equation*}
\mathbb{P}[A_{n,x}^h] = \mathbb{P}[B_{\infty}(x,L_n) \stackrel{\geq h}{\longleftrightarrow} \partial_{\text{int}}B_{n,x}] \leq |\Lambda_{n,x}| \cdot p_n(h).
\end{equation*}
If the bound \eqref{p_n_bounds} were to hold for some value of $k_0 \geq b$, it would yield, together with the estimate \eqref{Lambda_n,x} on the combinatorial complexity $|\Lambda_{n,x}|$, that $\mathbb{P}[A_{n,x}^h] \leq (c_2 l_0^{2(d-1)} / e^{k_0 - b})^{2^n}$. Thus $k_0$ will have to satisfy the requirement $e^{k_0} \geq c l_0^{2(d-1)}$ for a certain constant $c>0$ in order for this bound to be of any use. Substituting this into \eqref{p_0_cond} yields the condition $p_0(h_0) = \mathbb{P}[A_{0,0}^{h_0}] \leq c' l_0^{-2(d-1)}$. For the purpose of proving Theorem \ref{T:ASYMPTOTICS_UB}, $h_0$ will be at most $h_{\text{as}}(1 + \varepsilon)$, for some $\varepsilon >0$. We will therefore essentially have to prove that the crossing event $A_{0,0}^{h_{\text{as}}(1 + \varepsilon)}$ decays roughly like $\exp (-c d \log d)$ for all $d \geq c'(\varepsilon)$ (recall that $l_0$ is required to grow at least like $\sqrt{d}$, cf. \eqref{L_n}). 
\hfill $\square$
\end{remark}

\subsection{The local picture} \label{S:UB_LCB}

In this subsection, we provide an estimate regarding the local connectivity properties of the level set $E_\varphi^{\geq h_{\text{as}}(1+ \varepsilon)}$, for arbitrary $\varepsilon > 0$, in high dimension. Here, ``local'' means that we only investigate connectivity within an $\ell^1$-ball having a radius of order at most $d$ around a given point on the lattice. Specifically, we will consider the probability of a nearest-neighbor path in $E_\varphi^{\geq h_{\text{as}}(1+ \varepsilon)}$ connecting the center of such a ball to its boundary. The resulting bound, cf. Theorem \ref{T:LOCAL_CB}, is of independent interest, but it will be crucial in establishing the seed estimate \eqref{p_0_cond}, thus enabling us to launch the renormalization. We begin with the following simple lemma, which will be useful in several instances below.

\begin{lemma} \label{L:zeta_i_bound} $( d\geq 3)$ 

\medskip
\noindent There exists a constant $c>0$ such that, for all $\ell > 0$, $x \in \mathbb{Z}^d$, all sets $K, U$ satisfying $ U \subseteq K \subset \subset \mathbb{Z}^d \setminus \{ x \}$ and $|U| \leq \ell d$, and all $\varepsilon > 0$, 
\begin{equation} \label{zeta_i_bound}
\mathbb{P} \Big[ \sum_{y\in U} p_{x,y}^K \varphi_y > \varepsilon h_{\textnormal{as}} \Big] \leq \exp \Big\{- \frac{c \, \varepsilon^2}{\ell(\ell+1)} d \log d \Big\},
\end{equation}
with $p_{x,y}^K$ as defined in \eqref{EQ:p_x,y^K}.
\end{lemma}

\begin{proof}
Let $\ell > 0$, $x \in \mathbb{Z}^d$, $\varepsilon > 0$, and the sets $K,U$ be fixed as to satisfy the above assumptions. We abbreviate $\zeta_x =  \sum_{y\in U} p_{x,y}^K \varphi_y$, which is a centered Gaussian variable. We compute, for arbitrary $\lambda > 0$,
\begin{equation} \label{var_bound1}
\begin{split}
\text{Var}( \lambda \zeta_x) = \mathbb{E}[(\lambda \zeta_x)^2] 
&=  \lambda^2 \sum_{y,z \in U} p_{x,y}^K \, p_{x,z}^K \, \mathbb{E}[\varphi_y \varphi_z] \\
&=  \lambda^2 \Big\{ g(0) \sum_{y \in U} (p_{x,y}^K)^2 + \sum_{\substack{ y,z \in U \\ y\neq z}} p_{x,y}^K \, p_{x,z}^K \, g(y - z) \Big\}.
\end{split}
\end{equation}
Observe that for all $y \in U$ and $d\geq 3$,
\begin{equation} \label{var_bound2}
p_{x,y}^K \stackrel{\eqref{EQ:p_x,y^K}}{=} P_{x}[H_{K} < \infty, X_{H_{K}}= y] \leq  P_{x}[H_y < \infty]   \leq c/d,
\end{equation}
where the last step follows from an elementary application of the strong Markov property at time $H_{S_1(y,1)}$ (recall that $x \notin U$) and \eqref{return_probas1}. Moreover, for all $z \in \mathbb{Z}^d \setminus \{ 0\}$ and $d \geq 3$, by the strong Markov property at time $H_z$,
\begin{equation}\label{var_bound3}
g(z)= P_0[H_z < \infty]\cdot g(0) \stackrel{\eqref{g(0)}, \eqref{var_bound2}} \leq c/d.
\end{equation}
Inserting the bounds \eqref{var_bound2}, \eqref{var_bound3} into \eqref{var_bound1} and using again that $g(0) = O(1)$ as $d\to \infty$ (cf. \eqref{g(0)}) gives
\begin{equation*}
\begin{split}
\text{Var}(\lambda \zeta_x) \leq c \lambda^2  ( d^{-2}|U| + d^{-3}|U|^2) \leq c \lambda^2 d^{-1} (\ell + \ell^2),
\end{split}
\end{equation*}
where the second step follows because $|U| \leq \ell d$ by assumption. Thus, by Markov's inequality, we obtain, for all $\lambda >0$ and $d\geq 3$,
\begin{equation*}
\mathbb{P} [ \zeta_x > \varepsilon h_{\text{as}}] \leq  e^{- \lambda \varepsilon h_{\text{as}}} \cdot e^{\text{Var}( \lambda \zeta_x)/2} \leq \exp \Big\{-  \frac{1}{2c \ell (\ell +1)} d (\varepsilon h_{\text{as}})^2\Big\},
\end{equation*}
where the last step follows by optimizing over $\lambda$, which occurs for the choice $\lambda =   \varepsilon d  h_{\text{as}} / c \ell (\ell + 1)$ (with $c$ the constant appearing in the bound for $\text{Var}(\lambda \zeta_x)$ above). On account of \eqref{EQ:h_as}, and since $g(0) \geq 1$, this implies \eqref{zeta_i_bound}, and thus completes the proof of Lemma \ref{L:zeta_i_bound}. 
\end{proof}

We proceed to the main result of this section. For convenience, we introduce the shorthand
\begin{equation} \label{EQ:h_epsilon}
h_{\text{as}}^{( \varepsilon)} = h_{\text{as}}(1+ \varepsilon), \text{ for $\varepsilon > 0$}.
\end{equation}

\begin{theorem} \label{T:LOCAL_CB} $(\varepsilon > 0)$

\medskip
\noindent There exist constants $c_4(\varepsilon) \geq 1$ and $c_5>0$ such that for all $N \geq c_4(\varepsilon)$ and $d \geq c(\varepsilon, N)$,
\begin{equation} \label{T:LOCAL_CB1}
\mathbb{P}[0 \stackrel{\geq h_{\textnormal{as}}^{(\varepsilon)}}{\longleftrightarrow} S_1(0, Nd)] \leq \exp \{- c_5 \, f(\varepsilon, N) \cdot d \log d \},
\end{equation}
where $f(\varepsilon, N)= \varepsilon^3 \sqrt{ N /(1+ \varepsilon)}$.
\end{theorem}

We briefly outline the proof of Theorem \ref{T:LOCAL_CB}, which essentially comprises three steps. First, instead of considering the whole path connecting the origin to $S_1(0, Nd)$ directly, we look at the ``local traces'' it leaves after first visiting each of $N_0 = c N$ (for suitable $c \in (0,1)$) concentric $\ell^1$-annuli around the origin having a width of order $d$ (similarly to what was done in \cite{SUB} in the context of interlacement percolation). We are led to consider the probability that $N_0$ ``well-separated'' paths in $E_\varphi^{\geq h_{\textnormal{as}}^{(\varepsilon)}}$ of length $ \lfloor \ell d \rfloor $ each, for some $0 < \ell \leq 1$, all occur simultaneously (which competes against a suitable combinatorial complexity). The parameter $\ell$ will have to be carefully chosen a posteriori. In a second step, we discover the field along a fixed collection of paths ``dynamically,'' using Lemma \ref{L:conds_exps2}, as alluded to in the Introduction. Finally, Lemma \ref{L:zeta_i_bound} will provide good controls on the error terms that arise.  

\begin{proof}
We define
\begin{equation} \label{bar_c_6}
c_6 = 2(\lceil c_0 \rceil + 1), \quad \text{(see \eqref{c_0} for the definition of $c_0$)}.
\end{equation}
Let $\varepsilon > 0$  be fixed and assume $N \geq c_6$ (stronger conditions on $N$ will follow). Instead of working with $N$ directly, it will be convenient to use
\begin{equation} \label{Nhat}
N_0 = \lfloor  N / c_6 \rfloor \ (\geq 1).
\end{equation}
We also introduce a parameter
\begin{equation} \label{eta}
0 < \ell = \ell(\varepsilon,N_0) \leq 1
\end{equation}
to be specified later. Given a vertex $x \in \mathbb{Z}^d$, we denote by $\Pi_\ell (x)$ the set of all self-avoiding nearest-neighbor paths of length $\lfloor \ell d \rfloor $ starting in $x$ (by convention, if $ \lfloor \ell d \rfloor = 0$, the set consists only of the vertex $x$ itself). By construction, any path connecting the origin to $S_1(0, Nd)$ intersects all spheres $S_1(0, n\, c_6 d)$, for $0 \leq n \leq N_0$. Moreover, any self-avoiding path $\pi \in \Pi_\ell(x)$, with $x \in S_1(0, n \,c_6d)$ for some $0 \leq n < N_0$ and $0<\ell \leq 1$, satisfies  $\text{range}(\pi) \subset B_1(0, Nd)$. Thus,
\begin{equation} \label{Peierls1}
\begin{split}
&\mathbb{P}[0 \stackrel{\geq h_{\text{as}}^{(\varepsilon)}}{\longleftrightarrow} S_1(0, Nd)] \\
&\qquad \leq \mathbb{P}\Big[\bigcup_{\substack{x_n \in S_1(0, n c_6 d) \\ 0 \leq n < N_0 }} \, \bigcup_{\substack{\pi_n \in \Pi_\ell(x_n) \\ 0 \leq n < N_0 }} \big\{ E_\varphi^{ \geq h_{\text{as}}^{(\varepsilon)}} \supseteq \text{range}(\pi_n), 0 \leq n < N_0 \big\}\Big] \\
&\qquad \leq \sum_{\substack{x_n \in S_1(0, nc_6 d) \\ 0 \leq n < N_0 }} \, \sum_{\substack{\pi_n \in \Pi_\ell(x_n) \\ 0 \leq n < N_0 }} \mathbb{P} \big[ E_\varphi^{ \geq h_{\text{as}}^{(\varepsilon)}} \supseteq \text{range}(\pi_n), 0 \leq n < N_0 \big],
\end{split}
\end{equation}
for all $N_0 \geq 1$ (i.e. $N\geq c_6$) and $0<\ell \leq 1$ (see also Figure \ref{F:local_trace} below). For later reference, we note that given two paths $\pi_n \in \Pi_\ell(x_n)$ and $\pi_m \in \Pi_\ell(x_m)$ for some $0< \ell \leq 1$, with $0 \leq n < m < N_0$, $x_{n} \in S_1(0, {n}c_6d)$ and $x_{m} \in S_1(0, {m}c_6d)$,
\begin{equation} \label{dist_pi_n}
d_1(\pi_n, \pi_m) \geq |x_{m}-x_{n}|_1 - 2 \ell d \geq c_6 d - 2d \stackrel{\eqref{bar_c_6}}{\geq} 2 c_0 d. 
\end{equation}
\begin{figure} [h!] 
\centering
\psfragscanon
\includegraphics[width=15cm]{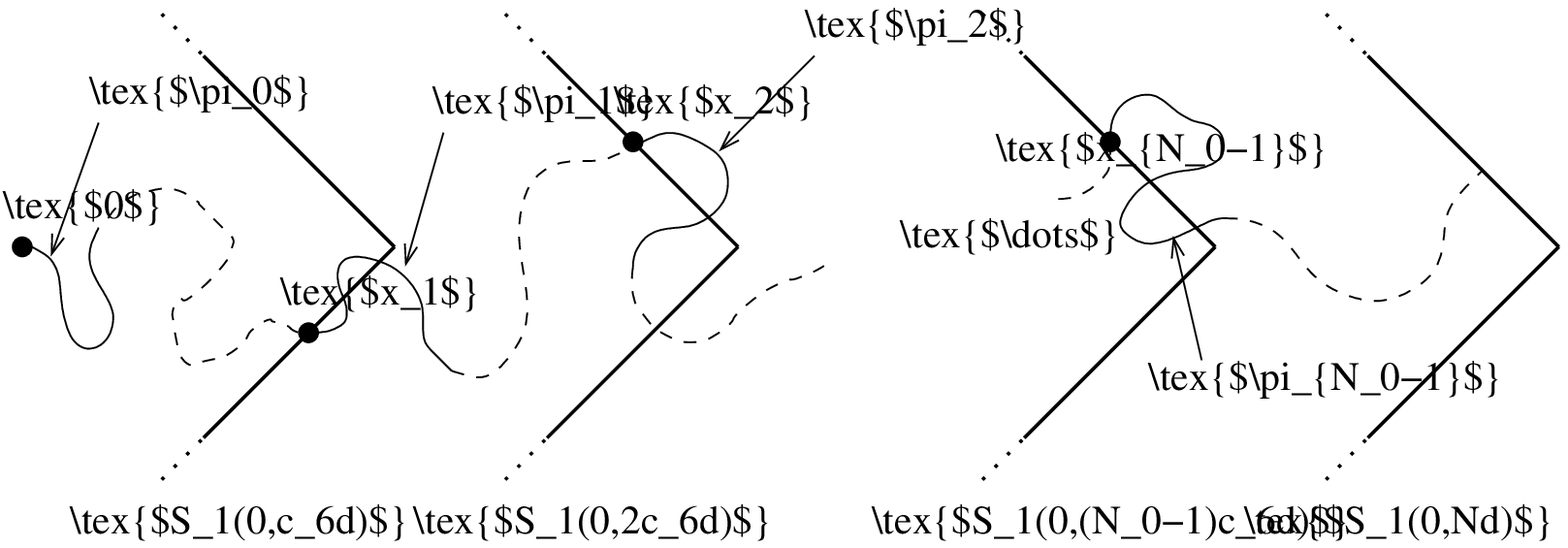}
\caption{The collection $(\pi_n)_{0\leq n <N_0}$ of ``local traces'' left by a self-avoiding path in $E_{\varphi}^{\geq h_{\text{as}}^{(\varepsilon)}}$ connecting the origin to $S_1(0,Nd)$, obtained by considering the first $\lfloor \ell d \rfloor$ steps after first hitting each of the spheres  $S_1(0, {n}c_6d)$, for $0\leq n < N_0$ (when viewed as starting from $0$). The paths $\pi_n$, $0\leq n < N_0$, are in fact ``well-separated,'' as captured by \eqref{dist_pi_n}.}
\label{F:local_trace}
\end{figure}
Since $|\Pi_\ell(x)| \leq (2d)^{\lfloor \ell d \rfloor}$ for all $d \geq 3$, $0< \ell \leq 1$ and  $x \in \mathbb{Z}^d$, and using the estimate \eqref{S_1_bound} on the cardinality of $\ell^1$-spheres, the number of terms appearing in the double sum of \eqref{Peierls1} is bounded by
\begin{equation*}
\begin{split}
\prod_{n=0}^{N_0-1}|S_1(0, nc_6d)|\cdot (2d)^{\lfloor \ell d \rfloor} &\leq \exp \Big\{  N_0 \lfloor \ell d \rfloor \log(2d) + d \sum_{n=0}^{N_0-1}(n c_6  +2)\Big\} \\
&= \exp \Big\{ N_0  \Big( \lfloor \ell d \rfloor \log (2d) + \frac{c_6 (N_0-1)d}{2} + 2d \Big) \Big\} \\ 
&\leq \exp \big\{ (1+ \varepsilon)  N_0 \lfloor \ell d \rfloor  \log d \big\} \\
&\stackrel{\text{def.}}{=} \mathbf{C}_d(\varepsilon, N_0 ,\ell)  
\end{split}
\end{equation*}
for all $N_0 \geq 1$, $0<  \ell \leq 1$ and $d \geq c( \varepsilon, N_0, \ell)$. Returning to \eqref{Peierls1}, this yields
\begin{equation}\label{Peierls2}
\begin{split}
&\mathbb{P}[0 \stackrel{\geq h_{\text{as}}^{(\varepsilon)}}{\longleftrightarrow} S_1(0, Nd)] \\
& \qquad \leq \mathbf{C}_d(\varepsilon, N_0 ,\ell) \cdot \sup_{\substack{x_n \in S_1(0, nc_6d) \\ \pi_n \in \Pi_\ell (x_n) \\ 0 \leq n < N_0}} \mathbb{P} [ E_\varphi^{ \geq h_{\text{as}}^{(\varepsilon)}} \supseteq \text{range}(\pi_n), 0 \leq n <N_0 ],
\end{split}
\end{equation}
for all $N_0 \geq 1$, $0<  \ell \leq 1$ and $d \geq c( \varepsilon, N_0, \ell)$. We now focus on the probability appearing on the right-hand side of \eqref{Peierls2}. Thus, we fix $N_0$ self-avoiding paths
\begin{equation*}
\begin{split}
&\pi_n = (x_{n,k})_{0\leq k \leq \lfloor \ell d \rfloor} \in \Pi_\ell (x_{n,0}), \text{ with $x_{n,0} \in S_1(0, nc_6d)$ for all $0\leq n < N_0$,} 
\end{split}
\end{equation*}
and define
\begin{equation} \label{EQ:local_conn_K}
K = \bigcup_{n=0}^{N_0-1}\text{range}(\pi_n), \text{ so that $|K| = N_0 \lceil \ell d \rceil $}.
\end{equation}
For notational convenience, we introduce the set of labels 
\begin{equation*}
I = \{0,\dots, N_0-1 \} \times \{ 0,\dots, \lfloor \ell d \rfloor \} \ni i = (i_1, i_2)
\end{equation*}
(where $i_k$, $k=1,2$, denote the coordinates of $i$, with values in $\{0,\dots, N_0-1 \}$ and $\{ 0,\dots, \lfloor \ell d \rfloor \}$, respectively), so that $\text{range} (\pi_n) = \{ x_i ; \, i \in I, i_1= n  \}$ for all $0\leq n < N_0$ and $K= \{ x_i ; \, i \in I \}$. We further denote by $\prec$ the lexicographic order on $I$. This induces an ordering of the points in $K$. Finally, we also set
\begin{equation}\label{K_i}
K_{i} = \{ x_j ; \, j \prec i \} \  (\subset K), \text{ for $i \in I$}.
\end{equation}
The event
\begin{equation} \label{event_START}
\{ E_\varphi^{ \geq h_{\text{as}}^{(\varepsilon)}} \supseteq \text{range}(\pi_n), 0 \leq n < N_0 \} = \bigcap_{i \in I} \{\varphi_{x_i} \geq h_{\text{as}}^{(\varepsilon)} \}
\end{equation}
is measurable with respect to the $\sigma$-algebra generated by $\varphi_x$, $x \in K$, and we will now construct the field $(\varphi_x)_{x\in K}$ using Lemma \ref{L:conds_exps2}, by adding one variable at a time, according to the ordering \nolinebreak $\prec$ (thus, $K_{i}$ as defined in \eqref{K_i} denotes the set of points at which the field has been discovered ``before time $i$''). Specifically, we proceed as follows. Let $\psi_i$, $i \in I$, be an independent family of Gaussian random variables, with
\begin{equation}\label{psi_i}
\psi_{i} \sim \mathcal{N} (0,g_{K_{ i} ^c}(x_i,x_i)), \text{ for all $i \in I$}
\end{equation}
(see \eqref{GreenFunctionSubK} for the definition of the killed Green function), and recall from \eqref{EQ:p_x,y^K} that $p_{x, y}^{K_i} = P_{x}[H_{K_{ i} } < \infty, X_{H_{K_{  i} }}= y]$, for all $x,y \in \mathbb{Z}^d$. We define recursively
\begin{equation} \label{phi_K}
\begin{split}
&\varphi_{x_i}= \psi_{i}, \text{ for $i = (0,0)$}, \\
&\varphi_{x_{i}} = \psi_i + \sum_{y \in K_i} p_{x_i, y}^{K_i} \cdot \varphi_{y},  \text{ for $i \in I \setminus \{ (0,0) \} $}.
\end{split}
\end{equation}
By Lemma \ref{L:conds_exps2}, the law of $(\varphi_{x_i})_{i \in I}$ as defined above is precisely that of the Gaussian free field restricted to $K$. By \eqref{dist_pi_n}, the main contribution to the sum on the right-hand side of \eqref{phi_K} will come from the points which belong to the same path as $x_i$ (i.e. to $\pi_{i_1}$). Accordingly, we introduce the sets
\begin{equation} \label{overlineK_i}
U_i \stackrel{\text{def.}}{=} \{ x_j ; \, j \prec i \text{ and } j_1 = i_1 \} \ (\subset K_i), \text{ for $i \in I$},
\end{equation}
and obtain the decomposition,
\begin{equation} \label{phi_K_2}
\varphi_{x_i}=  \psi_i \,  + \,  \zeta_i  \, +  \, \sum_{y \in K_i \setminus U_i} p_{x_i, y}^{K_i}  \varphi_{y}, \text{ with $\, \zeta_i =  
\sum_{y \in U_i} p_{x_i, y}^{K_i}  \varphi_{y}$, for all $i \in I$,}
\end{equation}
where the above sums are understood to vanish identically whenever the summation is over the empty set. Now, observe that for all $i \in I$, on the event $\{ \varphi_{x_i} \geq h_{\text{as}}^{(\varepsilon)} \} \cap \{\zeta_i \leq \varepsilon  h_{\text{as}} /4 \}$,
\begin{equation} \label{event_split_1}
\begin{split}
h_{\text{as}}(1+ \varepsilon) = h_{\text{as}}^{(\varepsilon)}
&\stackrel{\eqref{phi_K_2}}{\leq } \psi_i + \zeta_i +  \sum_{y \in K_i \setminus U_i} p_{x_i, y}^{K_i}  \varphi_{y} \\
&\stackrel{\eqref{EQ:p_x,y^K}}{\leq} \psi_i + \frac{\varepsilon}{4} h_{\text{as}} + 1_{\{i_1 > 0\}} P_{x_i} [H_{K_i \setminus U_i} < \infty] \cdot \sup_{y \in K_i \setminus U_i} \varphi_y.
\end{split}
\end{equation}
In order to bound the hitting probability $P_{x_i} [H_{K_i \setminus U_i} < \infty]$, we note that, by construction, cf. \eqref{dist_pi_n} and \eqref{overlineK_i},
\begin{equation} \label{dist_pi_n2}
d_1(  x_i , K_i \setminus U_i) \geq 2c_0 d, \text{ for all $i \in I$ with $i_1 > 0$}.
\end{equation}
Thus, we can find $c > 0$ such that
\begin{equation} \label{easy_hittingprobas}
\begin{split}
P_{x_i} [H_{K_i \setminus U_i} < \infty] 
&\stackrel{\eqref{1.13}}{\leq} \text{cap}(K_i \setminus U_i) \cdot \sup_{y \in K_i \setminus U_i} g(y-x_i) \\
&\stackrel{\eqref{c_0}}{\leq} |K_i \setminus U_i | \cdot 2^{-d/2 -2} \leq e^{-cd},
\end{split}
\end{equation}
for all $i \in I$ with $i_1 >0$ and all $d \geq c'(\varepsilon, N_0, \ell)$, where \eqref{c_0} applies because of \eqref{dist_pi_n2}, and we have bounded $|K_i \setminus U_i | \leq |K|  = N_0 \lceil \ell d \rceil$ using \eqref{1.10}, \eqref{1.11} and \eqref{EQ:local_conn_K}. Inserting \eqref{easy_hittingprobas} into \eqref{event_split_1} yields that there exists a suitable constant $c_7 >0$ such that for all $i \in I$ and $d \geq c(\varepsilon, N_0, \ell)$, on the event $\{ \varphi_{x_i} \geq h_{\text{as}}^{(\varepsilon)} \} \cap \{\zeta_i \leq \varepsilon  h_{\text{as}} /4 \}$,
\begin{equation*} 
\begin{split}
h_{\text{as}}(1+ \varepsilon) = h_{\text{as}}^{(\varepsilon)} \leq \varphi_{x_i} 
&\leq \psi_i + \frac{\varepsilon}{4} h_{\text{as}} + \frac{\varepsilon}{4} h_{\text{as}} \cdot e^{-c_7 d} \cdot \sup_{y \in K_i \setminus U_i} \varphi_y
\end{split}
\end{equation*}
(with the convention $\sup \emptyset = 0$, which occurs if $K_i \setminus U_i = \emptyset$, i.e. $i_1 =0$). Thus, for all $i \in I$ and $d \geq c(\varepsilon, N_0, \ell)$,
\begin{equation*}
\Big( \{ \varphi_{x_i} \geq h_{\text{as}}^{(\varepsilon)} \} \cap \{\zeta_i \leq \varepsilon  h_{\text{as}} /4 \} \cap \Big\{ \sup_{y \in K_i \setminus U_i} \varphi_y \leq  e^{c_7 d} \Big\} \Big) \ \subseteq \{ \psi_i \geq h_{\text{as}}^{(\varepsilon/2)} \},
\end{equation*}
and therefore
\begin{align*}
\{ \varphi_{x_i} \geq h_{\text{as}}^{(\varepsilon)} \} &\subseteq \Big(  \{ \psi_i \geq h_{\text{as}}^{(\varepsilon/2)} \} \cup \Big\{ \sup_{y \in K_i \setminus U_i} \varphi_y >  e^{c_7 d} \Big\} \cup \{\zeta_i > \varepsilon  h_{\text{as}} /4 \}  \Big) \\
&\subseteq \Big(\{ \psi_i \geq h_{\text{as}}^{(\varepsilon/2)} \} \cup \Big\{ \sup_{y \in K} \varphi_y >  e^{c_7 d} \Big\} \cup \bigcup_{j \in I }\{\zeta_j > \varepsilon  h_{\text{as}} /4 \}  \Big).
\end{align*}
Going back to our initial event in \eqref{event_START}, by a union bound, we obtain, for all $N_0 \geq 1$, $0< \ell \leq 1$ and $d \geq c(\varepsilon, N_0, \ell)$,
\begin{equation}\label{event_SPLITTING}
\begin{split}
&\mathbb{P} \big[ E_\varphi^{ \geq h_{\text{as}}^{(\varepsilon)}} \supseteq \text{range}(\pi_n), 0 \leq n < N_0 \big]  \\
&\qquad \leq  \prod_{i \in I } \mathbb{P}[\psi_i \geq h_{\text{as}}^{(\varepsilon/2)}] + \mathbb{P}\Big[ \sup_{y \in K} \varphi_y >  e^{c_7 d}\Big] + |I| \cdot \sup_{i \in I} \mathbb{P}[\zeta_i > \varepsilon  h_{\text{as}} /4] ,
\end{split}
\end{equation}
where we also used independence of the variables $\psi_i$, $i \in I$, see above \eqref{psi_i}. We consider each of the three terms in \eqref{event_SPLITTING} separately. Recall the definition of the combinatorial complexity $\mathbf{C}_d$ above \eqref{Peierls2}.
\begin{lemma} \label{L:MASTER_FORMULA_BD}
\begin{align}
&\prod_{i \in I } \mathbb{P}[\psi_i \geq h_{\text{as}}^{(\varepsilon/2)}] \leq  \mathbf{C}_d^{-1} (\varepsilon, N_0, \ell) \cdot e^{- c \varepsilon^2 N_0 \lceil \ell d \rceil \log d}, \text{ for $N_0 \geq 1$, $0< \ell \leq 1$, $d \geq c'$}.\label{EQ:local_conn_final3} \\
&\mathbb{P}\Big[ \sup_{y \in K} \varphi_y >  e^{d}\Big] \leq \mathbf{C}_d^{-2}(\varepsilon, N_0, \ell), \text{ for  $N_0 \geq 1$, $0 < \ell \leq 1$, $d \geq c(\varepsilon, N_0, \ell)$}. \label{EQ:local_conn_final2}
\end{align}
Moreover, for the choice of
\begin{equation} \label{EQ:local_conn_l}
\ell =  \ell(\varepsilon, N_0) = 1  \wedge \sqrt{\frac{c_8 \varepsilon^2}{4(1 + \varepsilon)N_0}}, 
\end{equation}
one has
\begin{equation} \label{EQ:local_conn_final1}
|I| \cdot  \sup_{i \in I} \mathbb{P}[\zeta_i > \varepsilon  h_{\text{as}} /4] \leq \mathbf{C}_d^{-2}(\varepsilon, N_0), \text{ for all $N_0 \geq 1$ and $d \geq c(\varepsilon, N_0)$}.
\end{equation}
\end{lemma}
\begin{proof9}
First, observe that
\begin{equation*}
\mathbb{E}[\psi_i^2] \stackrel{\eqref{psi_i}}{=} g_{K_i^c}(x_i,x_i) \stackrel{\eqref{G-GsubK}}{\leq} g(0) \stackrel{\eqref{EQ:c_1}}{<} c_1^2, \text{ for all $i \in I$ and $d \geq 3$}. 
\end{equation*}
Hence, the elementary tail estimate \eqref{GRV_basic_estimate} yields (note that $|I|= |K| = N_0 \lceil \ell d \rceil$, cf. \eqref{EQ:local_conn_K}),
\begin{equation*}
 \prod_{i \in I } \mathbb{P}[\psi_i \geq h_{\text{as}}(1 + \varepsilon/2 )] \leq \big(c_1 e^{- h_{\text{as}}^2(1+ \varepsilon/2)^2/2g(0) }\big)^{N_0 \lceil \ell d \rceil}, 
\end{equation*}
for all $N_0 \geq 1$, $0< \ell \leq 1$ and $d \geq c$, which, upon expanding $(1+ \varepsilon/2)^2$, immediately implies \eqref{EQ:local_conn_final3}. The estimate \eqref{EQ:local_conn_final2} is also a direct consequence of \eqref{GRV_basic_estimate} and a union bound (the ease with which this is obtained merely reflects the fact that for any $x_i$, $i = (i_1,i_2)\in I$, the shift produced by points lying on paths discovered ``before'' $\pi_{i_1}$ (if any), has little influence on $\varphi_{x_i}$, cf. \eqref{phi_K_2} and \eqref{easy_hittingprobas}).

It remains to show \eqref{EQ:local_conn_final1}. Observe that, by definition of $\zeta_i$, see \eqref{phi_K_2}, Lemma \ref{L:zeta_i_bound} with $K=K_i$ and $U=U_i$ (which satisfies $|U_i| \leq \ell d$, cf. \eqref{overlineK_i}), applies and yields (for a suitable $c_8>0$; this defines the constant appearing in \eqref{EQ:local_conn_l}),
\begin{equation*}
|I| \cdot \mathbb{P}[\zeta_i > \varepsilon  h_{\text{as}} /4] \leq \exp\{ -c_8 \varepsilon^2 d \log d / \ell (\ell +1)\}, 
\end{equation*}
for all $i \in I$, $N_0 \geq 1$, $0 < \ell \leq 1$ and $d \geq c (\varepsilon, N_0, \ell)$. Selecting $\ell $ as defined in \eqref{EQ:local_conn_l}, which, in particular, satisfies the requirements of \eqref{eta}, ensures that
\begin{equation} \label{EQ:crucial_l}
\frac{c_8 \varepsilon^2 d \log d}{\ell(\ell+1)} \geq 2 (1+ \varepsilon)\lfloor \ell d \rfloor N_0 \log d
\end{equation}
(the last term should be read as twice the exponent appearing in the combinatorial complexity $\mathbf{C}_d$), for all $N_0 \geq 1$ and $d\geq 3$, and \eqref{EQ:local_conn_final1} follows. \hfill $\square$
\end{proof9}

Substituting the bounds \eqref{EQ:local_conn_final3}, \eqref{EQ:local_conn_final2} and \eqref{EQ:local_conn_final1} into \eqref{event_SPLITTING}, we see that for the value of $\ell$ in \eqref{EQ:local_conn_l}, all $N_0 \geq 1$ and $d \geq c(\varepsilon, N_0)$,
\begin{equation*}
\mathbb{P} \big[ E_\varphi^{ \geq h_{\text{as}}^{(\varepsilon)}} \supseteq \text{range}(\pi_n), 0 \leq n < N_0 \big] \leq \mathbf{C}_d^{-1}(\varepsilon, N_0) \cdot e^{- c' \varepsilon^2 N_0  \ell d  \log d}
\end{equation*}
(the contribution from \eqref{EQ:local_conn_final3} is dominating). Finally, inserting this into \eqref{Peierls2}, noting that $\ell(\varepsilon, N_0) < 1$ whenever $N_0 \geq c_9 (\varepsilon)$ for some constant $c_9(\varepsilon) \geq 1$, cf. \eqref{EQ:local_conn_l}, and on account of \eqref{Nhat}, we  readily obtain \eqref{T:LOCAL_CB1}, with $c_4(\varepsilon) \stackrel{\text{def.}}{=} c_6 \cdot c_9(\varepsilon)$, and for a suitable value of $c_5>0$. This completes the proof of Theorem \ref{T:LOCAL_CB}.
\end{proof}

\begin{remark} The only feature of the function $f(\varepsilon, N)$ which will be of importance below is that $\lim_{N\to \infty} f(\varepsilon, N) = \infty$ for every $\varepsilon > 0$. For the sole purpose of proving Theorem \nolinebreak\ref{T:ASYMPTOTICS_UB}, any other function $f(\varepsilon, N)$ with this property would have sufficed. Moreover, while the term $\ell(\ell+1)$ appearing in the bound \eqref{zeta_i_bound} looks innocent when $\ell \in (0,1]$, it is crucially used in \eqref{EQ:crucial_l} to ensure that the resulting function $f(\varepsilon, \cdot)$ indeed grows at least like a positive power of $N$ (in fact one could even have obtained $f (\varepsilon, \cdot)= \Theta(N^{2/3})$ from \eqref{EQ:crucial_l}). \hfill $\square$
\end{remark}

\subsection{D\'enouement} \label{S:DENOUEMENT}

Finally, we complete the proof of Theorem \ref{T:ASYMPTOTICS_UB}. This will involve the renormalization scheme introduced in Section \ref{S:UB_RS}. As noted in Remark \ref{R:renorm}, 2), the seed estimate \eqref{p_0_cond} needed to initiate the renormalization is rather strong, and Theorem \ref{T:LOCAL_CB} will be of crucial use to establish it.
  
\begin{proof2}
Let $\varepsilon > 0$ be fixed. In order to establish \eqref{EQ:ASYMPTOTICS_UB}, we have to show that
\begin{equation} \label{T:UB_GOAL}
h_{**}(d) \leq h_{\text{as}}(d)(1 + \varepsilon), \text{ for all $d \geq c(\varepsilon)$}.
\end{equation}
To begin with, we select the parameters $L_0$, $l_0$ and $N$ in the renormalization scheme from Section \ref{S:UB_RS}, cf. \eqref{L_n}. We let
\begin{equation} \label{T:UB_parameters}
L_0 = l_0 = d, \quad N = N(\varepsilon)=  \bigg \lceil c_4(\varepsilon)  \vee  \frac{2^5(2+ \varepsilon)6^2}{(c_5 \varepsilon^3)^2} \bigg \rceil  
\end{equation}
(recall the statement of Theorem \ref{T:LOCAL_CB} for the definition of $c_4(\varepsilon)$, $c_5$), so that the constraint $l_0 \geq20( \sqrt{d} + N(\varepsilon))$ in \eqref{L_n} holds for all $d \geq c(\varepsilon)$, and choose $k_0$ appearing in \eqref{alpha_n} as
\begin{equation} \label{T:UB_parameters2}
k_0= b+ \log(2(c_2 l_0)^{2(d-1)})\stackrel{\eqref{T:UB_parameters}}{=} b+ \log(2(c_2 d)^{2(d-1)})
\end{equation}
(see \eqref{Lambda_n,x} for the definition of $c_2$). Lastly, we define the increasing sequence $(h_n)_{n\geq 0}$ recursively as 
\begin{equation} \label{T:UB_h}
\begin{split}
& h_0 = h_{\text{as}}(1 + \varepsilon/2) \ (= h_{\text{as}}^{(\varepsilon/2)})
, \\
& h_{n+1} = h_n +  \alpha_n \big(c_3 (\sqrt{d}+N)\big)^{d-2} \big(2 l_0^{-(d-2)} \big)^{n+1},  \text{ for all $n \geq 0$},
\end{split}
\end{equation}
with $\alpha_n$ given by \eqref{alpha_n}. In particular, $(h_n)_{n\geq0}$ satisfies the ``sprinkling'' condition \eqref{h_n_cond}. To see that the choices in \eqref{T:UB_parameters}, \eqref{T:UB_parameters2} and \eqref{T:UB_h} are judicious, we first check that
\begin{equation} \label{T:UB_h_infty}
h_\infty = \lim_{n \to \infty} h_n \leq h_{\text{as}}(1 + \varepsilon), \text{ for all $d \geq c(\varepsilon)$}.
\end{equation}
Indeed, it follows from \eqref{alpha_n}, using \eqref{M}, \eqref{T:UB_parameters} and \eqref{T:UB_parameters2}, that for all $d \geq c(\varepsilon)$ and $n\geq 0$,
\begin{equation} \label{T:UB_alpha_n}
\begin{split}
\alpha_n &  \, = \,  \big( \log(2^n((N+1)3d)^d) \big)^{1/2} +  2^{(n+1)/2} \big(n^{1/2} + \big( b+ \log(2(c_2 d)^{2(d-1)})\big)^{1/2} \big) \\
& \, \leq \, c'(\varepsilon) (d\log d)^{1/2} \,  2^{n+1}
\end{split}
\end{equation}
(the dependence on $\varepsilon$ is due to $N$). Hence,
\begin{equation} \label{EQ:h_infty_bound}
\begin{array}{rcl}
h_\infty & \hspace{-1ex} \stackrel{\eqref{T:UB_h_infty}, \eqref{T:UB_h}}{=} & \hspace{-1ex}  \displaystyle h_0 + \sum_{n\geq 0}  \alpha_n \big(c_3 (\sqrt{d}+N)\big)^{d-2} \big(2 l_0^{-(d-2)} \big)^{n+1} \\
& \hspace{-1ex}  \stackrel{ \eqref{T:UB_parameters}, \eqref{T:UB_alpha_n}}{\leq} & \hspace{-1ex}  \displaystyle h_0 + c(\varepsilon) \, (d\log d)^{1/2} \cdot  (c'(\varepsilon) d^{1/2})^{-(d-2)} \sum_{n\geq 0} \big(4d^{-(d-2)} \big)^{n} \\
& \hspace{-1ex}  \leq &\hspace{-1ex}  h_0 + \varepsilon/2,  
\end{array}
\end{equation}
for all $d \geq c''(\varepsilon)$. This is more than enough to deduce \eqref{T:UB_h_infty} (we comment on this in Remark \ref{R:UB_FINAL} below).

In order to complete the proof of Theorem \ref{T:ASYMPTOTICS_UB}, we apply Proposition \ref{P:DEC_INEQ_PROPAGATED}. Except for the condition \eqref{p_0_cond}, all its requirements  are clearly satisfied by our choices of parameters in \eqref{T:UB_parameters}, \eqref{T:UB_parameters2} and \eqref{T:UB_h}, whenever $d \geq c(\varepsilon)$. To deduce the necessary seed estimate \eqref{p_0_cond}, we use the connectivity bound from Theorem \ref{T:LOCAL_CB}. Observing that the choice of $N=N(\varepsilon)$ in \eqref{T:UB_parameters} guarantees that $c_5 \, f(\varepsilon/2, N) \geq 6$, for $f$ and $c_5$ as appearing in \eqref{T:LOCAL_CB1}, it follows that
\begin{equation*} 
\begin{array}{rcl}
p_0(h_0) & \stackrel{\eqref{p_n}, \, \eqref{eventA}}{=} & \mathbb{P} [ B_\infty(0,L_0)  \stackrel{\geq h_0}{\longleftrightarrow} \partial_{\text{int}} B_{0,0}] \\
& \stackrel{\eqref{EQ:B_n,x}, \, \eqref{T:UB_h}}{\leq}& 2d (2L_0+1)^{d-1} \cdot \mathbb{P} [ 0 \stackrel{\geq h_{\text{as}}^{(\varepsilon/2)}}{\longleftrightarrow} S_1(0, N L_0)] \\
& \stackrel{\eqref{T:LOCAL_CB1}, \, \eqref{T:UB_parameters}}{\leq} & e^{d\log3d - 6 d\log d} \leq d^{-4d},
\end{array}
\end{equation*}
for all $d \geq c'(\varepsilon)$. 
On the other hand, \eqref{T:UB_parameters2} impies that $e^{-k_0} \geq d^{-3d}$ for all $d \geq c$, hence the condition $p_0(h_0) \leq e^{-k_0}$ in \eqref{p_0_cond} holds for all$d \geq c'(\varepsilon)$. We may thus apply Proposition \nolinebreak \ref{P:DEC_INEQ_PROPAGATED}, thereby obtaining
\begin{equation} \label{p_n_bound_final}
p_n(h_n) \stackrel{\eqref{p_n_bounds}}{\leq} e^{-(k_0-b)2^n} \stackrel{\eqref{T:UB_parameters2}}{\leq} \big(2(c_2 d)^{2(d-1)}\big)^{-2^n}, \text{for all $n \geq 0$ and $d \geq c(\varepsilon)$},
\end{equation}
with $(h_n)_{n \geq 0}$ as defined in \eqref{T:UB_h}. Hence, for all $n \geq 1$ and $d \geq c(\varepsilon)$, (N.B.: the first of the following chain of inequalities is obtained by covering $B_\infty(0,2L_n)$ with essentially disjoint $\ell^\infty$-boxes of radius $L_n$, and performing a union bound)
\begin{equation*}
\begin{split}
&\mathbb{P}[B_\infty(0,2L_n)  \stackrel{\geq h_{\text{as}}^{(\varepsilon)}}{\longleftrightarrow} S_\infty(0,4L_n)] \leq 2^d \, \mathbb{P}[B_\infty(0,L_n)  \stackrel{\geq h_{\text{as}}^{(\varepsilon)}}{\longleftrightarrow} S_\infty(0,3L_n)] \\
&\begin{array}{rcl}
\qquad & \stackrel{\eqref{eventA}, \eqref{EQ:B_n,x}}{=} & \hspace{-1ex} 2^d \, \mathbb{P}[A_{n,0}^{h_{\text{as}}^{(\varepsilon)}}] \stackrel{\eqref{T:UB_h_infty}}{\leq}  2^d \, \mathbb{P}[A_{n,0}^{h_n}] \stackrel{\eqref{cascading_events}, \eqref{p_n}}{\leq}  2^d |\Lambda_{n,0}| \cdot p_n(h_n) \\
\qquad & \stackrel{\eqref{Lambda_n,x},\eqref{p_n_bound_final}}{\leq} & \hspace{-1ex}  2^d  \big((c_2 d)^{2(d-1)}\big)^{2^n} \cdot \big((2 c_2 d)^{2(d-1)}\big)^{-2^n} \leq 2^{-2(d-1)(2^n-1)} \xrightarrow{n \to \infty} 0.
\end{array}
\end{split}
\end{equation*}
In particular, this implies that $\liminf_{L \to \infty} \mathbb{P}[B_\infty(0,L)  \stackrel{\geq h_{\text{as}}^{(\varepsilon)}}{\longleftrightarrow} S_\infty(0,2L)]=0$, for all $d \geq c(\varepsilon)$. Recalling the definition \eqref{h_**} of the critical level $h_{**}$, this yields \eqref{T:UB_GOAL}, and thus completes the proof of Theorem \ref{T:ASYMPTOTICS_UB}. \hfill $\square$
\end{proof2}

\begin{remark} \label{R:UB_FINAL}
Our renormalization scheme is somewhat asymmetrical. On the one hand, the sprinkling condition \eqref{h_n_cond} for the sequence $(h_n)_{n\geq 0}$ turns out to be very mild, as it only costs an ``additive'' $\varepsilon$, cf. \eqref{EQ:h_infty_bound} (obtaining $h_\infty \leq h_0(1+ \varepsilon)$, for all $d\geq c(\varepsilon)$ would have sufficed for the purpose of proving Theorem \ref{T:ASYMPTOTICS_UB}). On the other hand, the scheme relies on the strong seed estimate \eqref{p_0_cond}, and establishing it is what prevents us from obtaining a more precise result than \eqref{EQ:ASYMPTOTICS_UB}. \hfill $\square$
\end{remark}

\section{Lower bound} \label{S:LOWER_BOUND}

We proceed to show the lower bound, Theorem \ref{T:ASYMPTOTICS_LB}. As described in the Introduction, the proof comprises two main steps. The first one, which is the subject of Subsection  \ref{S:FINITE_SIZE_CRITERION}, reduces the problem of constructing an infinite cluster to a local statement. The assertion is roughly the following (see Theorem \ref{T:LOCAL_GLOBAL} below): given $\delta > 0$, if the level $h$ is such that with high probability, the set $E_\varphi^{\geq h + \delta}$ possesses a ubiquitous component in each of the $d$-dimensional hypercubes $2x + \{0,1 \}^d$, for $x\in \mathbb{Z}^2$ ($\subset \mathbb{Z}^d$) and $|x|_1 \leq 1$, which are all connected, then $E_\varphi^{\geq h}$ percolates whenever $d$ is sufficiently large (depending on $\delta$). 

In order to prove the lower bound \eqref{EQ:ASYMPTOTICS_LB}, it then suffices to verify that this criterion holds when $h = h_{\text{as}}(1 - \varepsilon)$, for arbitrary $\varepsilon > 0$. This step is split again into two parts, to which Subsections \ref{S:SUBST_COMP} and \ref{S:CONNECTING_SUBST_COMPS} are respectively devoted. First, we construct a wealth of substantial components in the hypercube, with cardinality growing polynomially in $d$, and show that most vertices in $\{0,1 \}^d$ are either contained in or neighboring such a component. The main result is entailed in Corollary \ref{C:number_subst_comps} below. The second part consists of patching together these substantial components to form a ``giant'' one, and to then connect the latter to the ubiquitous components contained in the neighboring translates of $\{ 0,1 \}^{\mathbb{Z}^d}$. This is achieved in Theorem \nolinebreak \ref{T:giant_comp}. All ingredients are put together at the end of Subsection \ref{S:CONNECTING_SUBST_COMPS} to complete the proof of \eqref{EQ:ASYMPTOTICS_LB}.

\subsection{Local ubiquity and connectivity are sufficient} \label{S:FINITE_SIZE_CRITERION}

We begin by establishing the finite-size criterion that guarantees percolation, Theorem \nolinebreak \ref{T:LOCAL_GLOBAL} below. To cope with the long-range interactions, we use a renormalization argument on $\mathbb{Z}^2$, which bears some resemblance to the one developed in Section \ref{S:UB_RS} above, together with a standard duality argument.

We view $\mathbb{Z}^2$ as a subset of $\mathbb{Z}^d$ by identifying $x=(x_1,x_2) \in \mathbb{Z}^2$ with $(x_1,x_2,0,\dots, 0) \in \mathbb{Z}^d$, and write $B^{(2)}_p(x,r) = B_p(x,r) \cap \mathbb{Z}^2$, with $x \in \mathbb{Z}^2$, $r \geq 0$, $p=1,2,\infty$, for two-dimensional balls. Similarly, we let $S^{(2)}_p(x,r) = \partial_{\text{int}} B^{(2)}_p(x,r)$. Moreover, we denote the $d$-dimensional hypercube and its translates in $\mathbb{Z}^2$ by
\begin{equation}\label{D:hypercube}
\mathbb{H}_x = 2x + \{0,1 \}^d, \text{ for all $x \in \mathbb{Z}^2$,}  
\end{equation}
and abbreviate $\mathbb{H} \stackrel{\text{def.}}{=} \mathbb{H}_0$, so that $\mathbb{H}_x$, $x \in \mathbb{Z}^2$, forms a partition of $\mathbb{H} + \mathbb{Z}^2$. The following definition is essentially borrowed from \cite{SUB}, Section 2 (this will in particular enable us to reinject certain isoperimetric controls which follow from the results in \cite{BL} (see Lemma 4.3 in \cite{SUB}) directly at a later stage). For arbitrary $h \in \mathbb{R}$ and $x \in \mathbb{Z}^2$, we introduce the (local) event
\begin{align} \label{EQ:goodEventDef}
\begin{split}
 {G}_x^h
= \big\{ &\omega \in \Omega ; 
\text{ for all $x' \in B^{(2)}_1(x,1)$, the set $\mathbb{H}_{x'} \cap E_\varphi^{\geq h}(\omega)$ contains a connected}\\
&\text{component $C_{x'}$ with the property that $\vert \overline{C}_{x'}^{\mathbb{H}_{x'}} \vert \geq (1-d^{-2})\vert \mathbb{H}_{x'}\vert$, and the}\\
&\text{sets $C_{x'}$, $x' \in B^{(2)}_1(x,1)$, are connected within $ E_{\varphi}^{\geq h} \cap \textstyle \bigcup_{x' \in B^{(2)}_1(x,1)}  \mathbb{H}_{x'}$} \big\}
\end{split}
\end{align}
(recall that $\overline{C}_{x'}^{\mathbb{H}_{x'}}$ denotes the $\ell^1$-closure of $C_{x'}$ in $\mathbb{H}_{x'}$). A vertex $x \in \mathbb{Z}^2$ will be called $h$-good if $ G_x^h$ occurs, for arbitrary $h \in \mathbb{R}$, and $h$-bad otherwise. Note that for all $d \geq 3$, each of the sets $C_{x'}$ in \eqref{EQ:goodEventDef} is necessarily unique in $\mathbb{H}_{x'}$: indeed, if $C \subset \mathbb{H}$ satisfies $\vert\overline{C}^{\mathbb{H}} \vert \geq (1-d^{-2})\vert \mathbb{H}\vert$, then in fact $|C|\geq |\overline{C}^{\mathbb{H}}| - |\partial C \cap \mathbb{H}| \geq (1-d^{-2})\vert \mathbb{H}\vert - d|C|$ (each vertex has $d$ neighbors in $\mathbb{H}$), and solving for $|C|$ yields $|C|\geq (d-1) d^{-2}\vert \mathbb{H}\vert \geq 2 d^{-2}\vert \mathbb{H}\vert$. But any other set $D\subset \mathbb{H}$ not connected to $C$ must be contained in $\mathbb{H}\setminus \overline{C}^{\mathbb{H}}$ and thus satisfy $|D|\leq d^{-2}|\mathbb{H}|$, hence $C$ is unique.
Accordingly, we will henceforth refer to any set $C_{x'}$ appearing in \eqref{EQ:goodEventDef} as \textit{the} giant component of $\mathbb{H}_{x'} \cap E_\varphi^{\geq h}$. It is then plain from \eqref{EQ:goodEventDef} that for all $d \geq 3$ and $h \in \mathbb{R}$,
\begin{equation}\label{EQ:good_sites}
\{ 0 \leftrightarrow \infty \text{ in } \{y \in \mathbb{Z}^2; \; y \text{ is $h$-good} \} \} \subseteq \{ 0 \leftrightarrow \infty \text{ in  } E_\varphi^{\geq h} \cap (\mathbb{H}+ \mathbb{Z}^2) \},
\end{equation}
i.e. percolation of $h$-good sites in $\mathbb{Z}^2$ implies percolation of $E_\varphi^{\geq h}$ (in $\mathbb{H}+ \mathbb{Z}^2$). To see this, observe that if two neighboring vertices $x,x+e \in \mathbb{Z}^2$ (with $e$ a unit vector in $\mathbb{Z}^2$) are both $h$-good, then the corresponding giant clusters $C_x \subset \mathbb{H}_x \cap E_\varphi^{\geq h}$ and $C_{x+2e} \subset \mathbb{H}_{x+2e} \cap E_\varphi^{\geq h}$ (cf. \eqref{D:hypercube}) are connected to the \textit{same} (by uniqueness) giant cluster in $\mathbb{H}_{x+e} \cap E_\varphi^{\geq h}$, thus $C_x$ and $C_{x+2e}$ belong to the same cluster of $E_\varphi^{\geq h}$. It then follows inductively that an infinite nearest-neighbor path of $h$-good vertices in $\mathbb{Z}^2$ implies the existence of an infinite cluster in $E_\varphi^{\geq h} \cap (\mathbb{H}+ \mathbb{Z}^2)$.

We are now ready to proceed to the main result of this section, which has a similar flavor as Theorem \nolinebreak 2.2 of \cite{SLB}.

\begin{theorem} \label{T:LOCAL_GLOBAL} $(\varepsilon > 0 , \, a = 1/10)$

\medskip
\noindent Given any sequence $(h(d))_{d \ge 3}$ such that
\begin{equation} \label{EQ:local_global}
\limsup_{d \to \infty} d^{2 + 3a}  \, \mathbb{P} [( G_0^{h(d) + \varepsilon} )^c ] = 0,
\end{equation}
one has
\begin{equation} \label{EQ:local_global2}
h_*(d) \ge h(d), \text{ for all $d\geq c(\varepsilon)$.}
\end{equation}
\end{theorem}

\begin{proof}
By \eqref{EQ:good_sites} and the definition of $h_*$ in \eqref{h_*}, in order to prove \eqref{EQ:local_global2}, it suffices to show that
\begin{equation}  \label{EQ:local_global3}
\mathbb{P}[0 \leftrightarrow \infty \text{ in } \{y \in \mathbb{Z}^2; \; y \text{ is $h(d)$-good} \}] > 0, \text{ for all $d\geq c(\varepsilon)$,}
\end{equation}
for this implies $\mathbb{P}[0 \stackrel{\geq h(d)}{\longleftrightarrow} \infty]>0$ and hence $h(d) \leq h_*(d)$, for all $d\geq c(\varepsilon)$. The proof of \eqref{EQ:local_global3} involves a renormalization scheme, which we now describe. We introduce the integer parameter
\begin{equation} \label{EQ:local_global_scales1}
 L_0 \geq 2,
\end{equation}
and define an increasing sequence of length scales $(L_n)_{n \geq 0}$ recursively as
\begin{equation} \label{EQ:local_global_scales2}
L_{n+1} = l_n L_n, \text{ $l_n = 20 \lceil L_n^a \rceil$, for $n \geq 0$} 
\end{equation}
(this should not be confused with the geometric sequence of length scales appearing in Section \ref{S:UB_RS}), together with corresponding renormalized lattices
\begin{equation} \label{EQ:local_global_lattice}
\mathbb{L}_n^{(2)} = L_n \mathbb{Z}^2 \ (\subset \mathbb{Z}^d), \text{ for $n \geq 0$}.
\end{equation}
We also introduce, for arbitrary $n \geq 0$ and $x \in \mathbb{L}_n^{(2)}$, the (bad) events
\begin{equation} \label{EQ:local_global_events}
\begin{split}
D_{n,x}^h = \{
&\text{$B^{(2)}_\infty (x,L_n)$ is connected to $S^{(2)}_\infty (x,3L_n)$} \\
&\text{by a $*$-path of $h$-bad vertices (in $\mathbb{Z}^2$)}\}
\end{split}
\end{equation}
(recall that $y\in \mathbb{Z}^2$ is called $h$-bad if $( G_y^h )^c$ occurs). By \eqref{EQ:goodEventDef}, the event $D_{n,x}^h$ is decreasing (in $\varphi$). Moreover, by translation invariance, the function
\begin{equation}\label{EQ:local_global_q}
q_n(h) = \mathbb{P}[D_{n,x}^h], \text{ for $n \geq 0$, $x\in \mathbb{L}^{(2)}_n$, $h \in \mathbb{R}$},
\end{equation}
is well-defined (i.e. independent of $x$) and non-decreasing in $h$, for every $n \geq 0$. The key to establishing \eqref{EQ:local_global3} will be to show that, if $L_0$ is a suitable (increasing) function of $d$, and \eqref{EQ:local_global} holds for some $\varepsilon > 0$, the probability $q_n(h(d))$ decays sufficiently rapidly to $0$ as $n \to \infty$, for all $d\geq c(\varepsilon)$, see \eqref{EQ:local_global_q_bound_final} below. Together with a straightforward (planar) duality argument, which will be detailed below \eqref{EQ:local_global_q_bound_final}, this will then yield \eqref{EQ:local_global3}.

To obtain good estimates for $q_n(h(d))$, $n \geq 0$, we first develop ``recursive bounds'' relating the functions $q_{n+1}(\cdot)$ and $q_{n}(\cdot)$, for arbitrary $n \geq 0$ (similar in spirit to what was done in Proposition \ref{P:DEC_INEQ} above, but simpler). To this end, we let $n \geq 0$ and $x \in \mathbb{L}^{(2)}_{n+1}$ be fixed, and introduce the sets
\begin{equation}\label{EQ:local_global_K_i}
\begin{split}
&S_i = \mathbb{L}^{(2)}_n \cap S_\infty^{(2)}(x, i \, L_{n+1}), \text{ for $i=1,2$}.
\end{split}
\end{equation}
By geometric considerations similar to those leading to \eqref{cascading_events}, see also Figure \ref{F:crossing_high_dim}, and on account of \eqref{EQ:local_global_scales2}, we deduce that
\begin{equation} \label{EQ:local_global_est1}
q_{n+1}(h) \leq \mathbb{P} \Big[ \bigcup_{\substack{x_i \in S_i \\ i=1,2}} D_{n,x_1}^h \cap D_{n,x_2}^h \Big] \leq c_0' l_n^2 \sup_{\substack{x_i \in S_i \\ i=1,2}} \mathbb{P}[D_{n,x_1}^h \cap D_{n,x_2}^h],
\end{equation}
for a suitable constant $c_0' \geq 1$ and all $h \in \mathbb{R}$. We consider the probability on the right-hand side of \eqref{EQ:local_global_est1}, for fixed $x_i \in S_i$, $i=1,2$ and $h \in \mathbb{R}$. Abbreviating 
\begin{equation} \label{EQ:local_global_K_i}
K_i = \bigcup_{ y \in B_\infty^{(2)}(x_i,3L_n)}  \bigcup_{z \in B_1^{(2)}(y,1)}\mathbb{H}_z
\end{equation}
and observing that $D_{n,x_i}^h \in \sigma(\varphi_y ; \, y \in K_i)$, for $i=1,2$, due to \eqref{EQ:goodEventDef} and \eqref{EQ:local_global_events}, we write
\begin{equation} \label{EQ:local_global_est2}
\begin{split}
\mathbb{P}[D_{n,x_1}^h \cap D_{n,x_2}^h] \leq 
&\ \mathbb{P}[D_{n,x_1}^h, \, \inf_{y \in K_1} \varphi_y \geq - \beta_n , \, \mathbb{P}[D_{n,x_2}^h \, | \, (\varphi_y)_{y \in K_1}]] \\
&+ \mathbb{P} [ \inf_{y \in K_1} \varphi_y < - \beta_n]
\end{split}
\end{equation}
for some cut-off value $\beta_n > 0$ to be selected below. By \eqref{phi_cond_exps}, 
\begin{equation*}
\mathbb{P}[D_{n,x_2}^h \, | \, (\varphi_y)_{y \in K_1}]= \widetilde{\mathbb{P}}[(1\{ \widetilde{\varphi}_x^{K_1} + \mu_x^{K_1} \geq h \})_{x \in \mathbb{Z}^d} \in D_{n,x_2}], \text{ $\mathbb{P}$-a.s.} 
\end{equation*}
(recall the notation from \eqref{EQ:event_A_Ah}), and we focus on bounding the random shift $ \mu_x^{K_1}$, for $x \in K_2$. By construction, cf. \eqref{EQ:local_global_scales1}, \eqref{EQ:local_global_scales2} and \eqref{EQ:local_global_K_i},
\begin{equation*}
d_1(K_1,K_2) \geq d_\infty(K_1,K_2) \geq  L_{n+1} - 6L_n - 6 \geq c_1' L_{n+1},
\end{equation*}
for a suitable constant $c_1' > 0$. Hence, setting
\begin{equation} \label{EQ:local_global_L_0}
L_0 =  \lceil (c_0 / c_1') \, d \,  \rceil \quad \text{(see \eqref{c_0} for the definition of $c_0$)},
\end{equation}
we obtain, on the event $\{\inf_{y \in K_1} \varphi_y \geq - \beta_n \}$, for all $x \in K_2$ and $d \geq 5$,
\begin{equation} \label{EQ:local_global_delta_def}
\begin{split}
- \mu_x 
&\stackrel{\eqref{mu}}{\leq} \beta_n P_x[H_{K_1} < \infty] \stackrel{\eqref{1.13}}{\leq} \beta_n \, \text{cap}(K_1) \sup_{y \in K_1} g(y-x) \\
&\stackrel{\eqref{c_0}}{\leq} \beta_n |K_1| \Big(\frac{c_0 d}{ c_1' L_{n+1}}\Big)^{d/2-2} \stackrel{\eqref{EQ:local_global_L_0}}{\leq} c \beta_n L_n^2 2^d l_n^{-(d/2-2)} \stackrel{\text{def.}}{=} \frac{\delta_n}{2},
\end{split}
\end{equation}
where in the second line, we have used the crude estimate $\text{cap}(K_1) \leq |K_1|$ (see \eqref{1.10}, \eqref{1.11}), the bound $|K_1| \leq c L_n^2 2^d$, which follows immediately from \eqref{EQ:local_global_K_i}, and the fact that $L_{n+1} \geq l_n L_0  \geq  l_n c_0 d / c_1'$ for all $n \geq 0$, due to \eqref{EQ:local_global_scales2} and \eqref{EQ:local_global_L_0}. We henceforth tacitly assume that $ d \geq 5$. By \eqref{EQ:local_global_delta_def}, on $\{\inf_{y \in K_1} \varphi_y \geq - \beta_n \}$, the inequality $\widetilde{\varphi}_x^{K_1} + \mu_x^{K_1} \leq h$ implies that $\widetilde{\varphi}_x^{K_1} - \mu_x^{K_1} \leq h + \delta_n$, for arbitrary $x \in K_2$, and therefore
\begin{equation}\label{EQ:local_global_est3}
\begin{split}
\mathbb{P}[D_{n,x_2}^h \,  | \,  (\varphi_x)_{x \in K_1}] 
&\leq \widetilde{\mathbb{P}}[(1\{ \widetilde{\varphi}_x^{K_1} - \mu_x^{K_1} \geq h + \delta_n \})_{x \in \mathbb{Z}^d} \in D_{n,x_2}] \\
&= \widetilde{\mathbb{P}}[(1\{ - \widetilde{\varphi}_x^{K_1} - \mu_x^{K_1} \geq h + \delta_n \})_{x \in \mathbb{Z}^d} \in D_{n,x_2}]\\
&= \mathbb{P}[\overline{D}_{n,x_2}^{-(h + \delta_n)} \,  | \,  (\varphi_x)_{x \in K_1}],
\end{split}
\end{equation}
on the event $\{\inf_{y \in K_1} \varphi_y \geq - \beta_n \}$, where we have used that $D_{n,x_2}$ is decreasing in the first line, and the symmetry of $ \widetilde{\varphi}^{K_1}$ in the second line (recall also \eqref{di7}). Inserting \eqref{EQ:local_global_est3} into \eqref{EQ:local_global_est2}, applying the FKG-inequality and \eqref{di7}, we obtain
\begin{equation*}
\mathbb{P}[D_{n,x_1}^h \cap D_{n,x_2}^h] \leq \mathbb{P}[D_{n,x_1}^{h} ] \cdot \mathbb{P} [D_{n,x_2}^{h+ \delta_n}] +   \mathbb{P}[ \inf_{x \in K_1} \varphi_x < - \beta_n].
\end{equation*}
Finally, substituting this into \eqref{EQ:local_global_est1}, taking suprema over $x_i \in S_i$, for $i=1,2$, and using symmetry yields, in view of \eqref{EQ:local_global_q},
\begin{equation}\label{EQ:local_global_est4}
q_{n+1}(h) \leq c_0'l_n^2 \big( q_n(h + \delta_n)^2 + \epsilon_n \big), \text{ with } \epsilon_n=   \mathbb{P}[ \sup_{x \in K_1} \varphi_x > \beta_n],
\end{equation}
for all $n \geq 0$, $h \in \mathbb{R}$ and $\beta_n > 0$. This yields the desired recursive bounds. In order to propagate them inductively, we select 
\begin{equation}\label{EQ:local_global_beta_n}
\beta_n =  c_1 \big( \sqrt{\log (2c_0' l_n^2l_{n+1}^3)} + \sqrt{\log|K_1|}\big), \text{ for all $n \geq 0$},
\end{equation}
with $c_1$ as defined in \eqref{EQ:BTIS} and $c_0'$ in \eqref{EQ:local_global_est1}. The key estimate comes in the following result.
\begin{lemma} \label{L:LOCAL_GLOBAL_q} $(h \in \mathbb{R}, \, \varepsilon >0, \text{ and } L_0, \, (l_n)_{n\geq 0}, \, (\beta_n)_{n\geq 0} \text{ as in \eqref{EQ:local_global_L_0},  \eqref{EQ:local_global_scales2}, \eqref{EQ:local_global_beta_n}, respectively})$

\medskip
\noindent There exists a constant $c_2'(\varepsilon) \geq 3$ such that, if 
\begin{equation}\label{EQ:local_global_qbound_cond}
q_0(h + \varepsilon) \leq l_0^{-3}, \text{ for all $d\geq c_2'(\varepsilon)$}
\end{equation}
holds, then
\begin{equation} \label{EQ:local_global_qbound}
q_n(h) \leq l_n^{-3}, \text{ for all $d\geq c_2'(\varepsilon)$ and $n \geq 0$}.
\end{equation}
\end{lemma}
\begin{proof3}
By \eqref{EQ:local_global_scales1} and \eqref{EQ:local_global_scales2}, we have $l_n^2 l_{n+1}^3 \leq cl_n^2 L_{n+1}^{3a} = cl_n^{2+3a}L_n^{3a} \leq c'L_n^{a(5+3a)}$, for all $n \geq 0$. Substituting this into \eqref{EQ:local_global_beta_n} and using that $|K_1| \leq cL_n^2 2^d$, we obtain the bound $\beta_n \leq c((\log L_n)^{1/2} + d^{1/2} )$, for all $d \geq 3$ and $n \geq 0$. Inserting this into the definition of $\delta_n$, see \eqref{EQ:local_global_delta_def}, yields
\begin{equation*}
\begin{split}
\sum_{n \geq 0} \delta_n 
&\leq c 2^d \sum_{n\geq 0} L_n^2 l_n^{-(d/2-2)} ((\log L_n)^{1/2} + d^{1/2} ) \\
&\leq c' 3^d \sum_{n\geq 0} L_n^{-a(d/2 -2 - 3/a)} \\
&\leq c' \frac{3^d}{L_0^{c''d}} \sum_{n\geq 0} (L_0^{-ac'' d})^n
\end{split}
\end{equation*}
for all $d \geq c$, where we have used $L_n \stackrel{\eqref{EQ:local_global_scales2}}{\geq} L_0^{(1+a)^n} \geq L_0^{1+ na}$ in the last line. Due to the choice of $L_0$ in \eqref{EQ:local_global_L_0}, given $\varepsilon >0,$ it follows that
\begin{equation}\label{EQ:local_global_delta_bound}
\sum_{n \geq 0} \delta_n \leq \varepsilon,
\end{equation}
for all $d\geq c(\varepsilon)$. Moreover (for reasons that will become clear shortly) we observe that
\begin{equation} \label{EQ:local_global_c_2_choose}
 2c_0' l_n^{-4}l_{n+1}^3 \leq c' l_n^{-4} (l_n L_n)^{3a} \leq c'' L_n^{-a + 3a^2} \leq 1, \text{ for all $n \geq 0$},
\end{equation}
whenever $ d\geq c$, due to the choice of $L_0$ in \eqref{EQ:local_global_L_0} and $a$ in Theorem \ref{T:LOCAL_GLOBAL} (this is why $a$ should not be chosen too large). Given $\varepsilon > 0$, we define $c_2'(\varepsilon)$ appearing in the statement of Lemma \ref{L:LOCAL_GLOBAL_q} in a way that \eqref{EQ:local_global_delta_bound} and \eqref{EQ:local_global_c_2_choose} simultaneously hold whenever $d \geq c_2'(\varepsilon)$.

We now prove \eqref{EQ:local_global_qbound} by induction over $n$. Let $h \in \mathbb{R}$, $\varepsilon > 0$ and $d \geq c_2'(\varepsilon)$ be fixed. In view of \eqref{EQ:local_global_delta_bound}, and because the function $q_{n}(\cdot)$ is non-decreasing for all $n \geq 0$, it suffices to show that
\begin{equation}\label{EQ:local_global_q_est1}
q_{n}\Big(h + \varepsilon - \sum_{i=0}^{n-1} \delta_i \Big) \leq l_{n}^{-3}, \text{ for all $n \geq0$}
\end{equation}
(with the convention that the sum equals $0$ when $n=0$). By assumption, cf. \eqref{EQ:local_global_qbound_cond}, we have that \eqref{EQ:local_global_q_est1} holds for $n=0$. Assume now it holds for some $n \geq 0$. By \eqref{EQ:local_global_est4}, we have
\begin{equation}\label{EQ:local_global_q_est2}
q_{n+1}\Big(h + \varepsilon - \sum_{i=0}^{n} \delta_i \Big) \leq c_0' l_n^2 \Big(q_n \Big(h + \varepsilon - \sum_{i=0}^{n-1} \delta_i \Big)^2 + \epsilon_n \Big).
\end{equation}
We bound each of the two terms appearing on the right-hand side separately.  By the BTIS-inequality \eqref{EQ:BTIS} and the choice of $\beta_n$ in \eqref{EQ:local_global_beta_n}, 
\begin{equation} \label{EQ:local_global_q_est3}
c_0' l_n^2 \epsilon_n \leq (2l_{n+1}^3)^{-1}.
\end{equation}
Moreover,
\begin{equation}\label{EQ:local_global_q_est4}
c_0' l_n^2 q_n \Big(h + \varepsilon - \sum_{i=0}^{n-1} \delta_i \Big)^2 \stackrel{\substack{\text{induction} \\ \text{hypothesis}}}{ \leq} c_0' l_n^{-4} \stackrel{\eqref{EQ:local_global_c_2_choose}}{\leq} (2l_{n+1}^3)^{-1}.
\end{equation}
Substituting \eqref{EQ:local_global_q_est3} and \eqref{EQ:local_global_q_est4} into \eqref{EQ:local_global_q_est2} yields $q_{n+1}(h + \varepsilon - \sum_{i=0}^{n} \delta_i) \leq l_{n+1}^{-3}$, as desired. This completes the proof of \eqref{EQ:local_global_q_est1}, and thus of Lemma \ref{L:LOCAL_GLOBAL_q}. \hfill $\square$
\end{proof3}

\bigskip
We now complete the proof of Theorem \ref{T:LOCAL_GLOBAL}. Let $\varepsilon > 0$ and $(h(d))_{d\geq 3}$ be a sequence of levels such that \eqref{EQ:local_global} holds. The latter implies that
\begin{equation*}
\begin{split}
l_0^3 \, q_0(h(d)+ \varepsilon) 
&\stackrel{\eqref{EQ:local_global_q}}{=} l_0^3 \, \mathbb{P} [D_{0,0}^{h(d)+ \varepsilon}] \\
&\stackrel{\eqref{EQ:local_global_events}}{\leq} l_0^3 \, \mathbb{P} \Big[ \bigcup_{y \in S_\infty^{(2)}  (0,L_0)} \{ \text{$y$ is $(h(d)+ \varepsilon)$-bad} \}\Big] \\
&\stackrel{\eqref{EQ:local_global_scales2}}{\leq} c L_0^{2+3a} \mathbb{P}\big[\big(G_0^{h(d)+ \varepsilon}\big)^c \big] \leq 1,
\end{split}
\end{equation*}
for all $d \geq c(\varepsilon)$, where the last step follows from the choice of $L_0$ in \eqref{EQ:local_global_L_0} and the assumption \eqref{EQ:local_global}. By Lemma \ref{L:LOCAL_GLOBAL_q}, we thus obtain
\begin{equation}\label{EQ:local_global_q_bound_final}
q_n(h(d)) \leq l_n^{-3}, \text{ for all $n \geq 0$ and $d \geq c(\varepsilon)$}.
\end{equation}
With this estimate at hand, we now prove the assertion \eqref{EQ:local_global3} by a standard planar duality argument. Let us call a $*$-\textit{circuit around $0$} any closed $*$-path $\gamma$ in $\mathbb{Z}^2$ such that the origin is contained in a finite connected component of $\mathbb{Z}^d \setminus \text{range}(\gamma)$. Denoting by $e_1$ the unit vector in the first coordinate, we have, recalling the definition of the events $D_{n,x}^{h}$ in \eqref{EQ:local_global_events},
\begin{equation*}
\begin{split}
&\big\{ \text{$0$ lies in a finite cluster of $\{y \in \mathbb{Z}^2; \, y \text{ is $h(d)$-good} \}$} \big\}  \\
&\subseteq \big\{ \text{$0$ is $h(d)$-bad}\big\} \cup  \big\{ \text{$\exists$ $*$-circuit of $h(d)$-bad vertices around $0$ intersecting $B_{\infty}^{(2)}(0,3L_0)$} \big\} \\
&\quad \;  \cup \bigcup_{n \geq 0} \big\{ \text{$\exists$ $*$-circuit of $h(d)$-bad vertices around $0$ intersecting $(3L_n, 3L_{n+1}]e_1 \cap \mathbb{Z}$} \big\} \\
&\subseteq \Big[ \bigcup_{x \in B_\infty^{(2)}(0,3L_0)} ( G_y^{h(d)} )^c \Big] \cup \bigcup_{n\geq 0} \  \bigcup_{y \in \mathbb{L}_n^{(2)} \cap (3L_n, 3L_{n+1}]e_1} D_{n,y}^{h(d)}.
\end{split}
\end{equation*}
By the choice of scales in \eqref{EQ:local_global_scales2} and \eqref{EQ:local_global_L_0}, this yields, for all $d\geq c(\varepsilon)$,
\begin{equation*}
\begin{split}
&\mathbb{P}[\text{$0$ lies in a finite cluster of $\{y \in \mathbb{Z}^2; \, y \text{ is $h(d)$-good} \}$}]  \\
&\begin{array}{rcl}
\qquad & \hspace{-1ex} \leq  & \hspace{-1ex} \displaystyle cL_0^2 \mathbb{P}[( G_0^{h(d)} )^c] + \sum_{n\geq 0} 3l_n \cdot q_n(h(d)) \\
\qquad & \hspace{-1ex} \stackrel{\eqref{EQ:local_global_q_bound_final}}{\leq} & \hspace{-1ex}  \displaystyle c' d^2 \mathbb{P}[( G_0^{h(d)} )^c]  + 3 \sum_{n\geq 0} l_n^{-2} \\
\qquad  & \hspace{-1ex} \leq  & \hspace{-1ex} \displaystyle c' d^2 \mathbb{P}[( G_0^{h(d) + \varepsilon} )^c] + \sum_{n\geq 0} (c''d)^{-2a(1+a)^n} \\
\qquad & \hspace{-1ex} \stackrel{\eqref{EQ:local_global}}{<} & \hspace{-1ex}1,
\end{array}
\end{split}
\end{equation*}
where we have also used in the penultimate step that $G_0^h$ is decreasing (in $\varphi$) and that $l_n > L_n^a \geq L_0^{a(1+a)^n}$ for all $n \geq 0$. This completes the proof of \eqref{EQ:local_global3}, hence of Theorem \nolinebreak \ref{T:LOCAL_GLOBAL}.
\end{proof}

\begin{remark}
By a more careful analysis, the condition \eqref{EQ:local_global} in Theorem \ref{T:LOCAL_GLOBAL} can be somewhat relaxed. Indeed, \eqref{EQ:local_global2} continues to hold under the weaker assumption that
\begin{equation*}
 \limsup_{d \to \infty} d^{2 + \delta} \mathbb{P} [( G_0^{h(d) + \varepsilon} )^c ] = 0, \text{ for some $\delta > 0$}
\end{equation*} 
(using a choice of $a$ in \eqref{EQ:local_global_scales1} depending on $\delta$). However, this will not be of importance, as our proof will show that for the relevant choice of $h(d) = h_{\text{as}}(d) (1 - 8\varepsilon)$, with $\varepsilon > 0$ arbitrary (the factor of $8$ is just for convenience), the above probability decays to $0$ as $d\to \infty$ faster than any polynomial, see Theorem \ref{T:giant_comp} below. \hfill $\square$
\end{remark}

\subsection{Constructing substantial components} \label{S:SUBST_COMP}

With Theorem \ref{T:LOCAL_GLOBAL} at hand, in order to prove the lower bound \eqref{EQ:ASYMPTOTICS_LB}, we will check that condition \eqref{EQ:local_global} holds at level $h(d) = h_{\text{as}}(d) (1 - 8\varepsilon)$, for arbitrary $\varepsilon > 0$ (the factor $8$ is immaterial, and merely reflects the fact that a few more sprinkling operations will be performed on the way). This will involve showing that $\mathbb{H}$ contains a giant component above this level with sufficiently high probability, cf. \eqref{EQ:goodEventDef}. We will construct this component by gluing together smaller building blocks, so-called substantial components, which, by definition, have cardinality growing like a polynomial in $d$ of sufficiently high degree. In this subsection, we show that most vertices in $\mathbb{H}$ are either neighboring or contained in such a substantial component with high probability, see Corollary \ref{C:number_subst_comps} below for the precise statement. The substantial component neighboring a given point in $\mathbb{H}$ will be built as a connected subset of a (large) deterministic tree embedded in $\mathbb{H}$ and rooted at this point. In particular, the ``perturbative'' representation of Lemma \ref{L:conds_exps2} will enable us to show that, conditionally on an event of high probability, the law of the restriction of $E_{\varphi}^{\geq  h_{\text{as}}(1- 2\varepsilon)}$ to this tree dominates a Galton-Watson process (with suitable binomial offspring distribution) on the same tree, cf. Lemma \ref{L:subt_comp_dom} below. 

Let $\varepsilon > 0$. For the sake of clarity, in the remainder of this article, the dependence of constants on $\varepsilon$ will be kept implicit. We introduce a parameter
\begin{equation}\label{EQ:subst_comp_b}
b=b(\varepsilon)= 1+ \frac{11}{\varepsilon}  \ (\geq 1).
\end{equation}
Given $K \subset \mathbb{Z}^d$, we call $C$ a \textit{substantial component of $K$} if $C$ is a connected subset of $K$ containing at least $\lfloor d^\varepsilon / b \rfloor^{b-1}$ points. 
\begin{theorem} \label{P:SUBST_COMP} $(d \geq 3 , \,  0< \varepsilon < 1/3, \, x \in \mathbb{H})$
\begin{equation} \label{P:SUBST_COMP1} 
\begin{split}
\mathbb{P}[&\text{a neighbor of $x$ is contained in a }\\
&\text{substantial component of $\mathbb{H} \cap E_\varphi^{\geq h_{\textnormal{as}}(1- 2\varepsilon)}$}] \geq 1 - c e^{-c' d^\varepsilon}.
\end{split}
\end{equation}
\end{theorem}

\begin{proof}
By symmetry, it suffices to consider the case $x=0$. We begin with some notation. Let $\varepsilon \in (0,1/3)$. We assume without loss of generality that $d \geq  b(\varepsilon) = b$, with $b$ as defined in \eqref{EQ:subst_comp_b} (it suffices to show Theorem \ref{P:SUBST_COMP} for $d \geq c'$, since the remaining cases can be taken care of by adapting the constant $c$ appearing in \eqref{P:SUBST_COMP1}), and introduce $b$ consecutive subsets
\begin{equation*}
I_k = \{(k-1) \lfloor d/b \rfloor + 1, \ldots, k \lfloor d/b \rfloor  \}, \text{ $1 \leq k \leq b$},
\end{equation*}
of $\{1, 2, \ldots, d\}$. We will interpret a part of $\mathbb{H}$ as a tree by considering
\begin{equation} \label{EQ:subst_comp_tree}
\mathbb{T} =
 \Big\{  \sum_{k=1}^j e_{i_k} \in  \mathbb{H} \;  ;  \;  1 \leq j \leq b \text{ and $i_k \in I_k$ for all $1\leq k \leq j$}  \Big\},
\end{equation}
where $e_i$ denotes the canonical unit vector in the $i$-th direction, for $1\leq i \leq d$. Given $x \in \mathbb{T}$, we refer to $j$ in the (unique) decomposition of $x =  \sum_{k=1}^j e_{i_k}$, with $i_k \in I_k$, $1\leq k \leq j$, as the generation of $x$, and to the set $\{ x + e_{i} ; \; i \in I_{j+1} \}$ (understood as $\emptyset$ if $j=b$) as the children of $x$. Thus, every vertex in $\mathbb{T} $ in generation smaller than $b$ has precisely $\lfloor d/b \rfloor$ children, and $b$ corresponds to the depth of $\mathbb{T} $. Furthermore, it will be convenient to set $ \mathbb{T}_0 = \mathbb{T}  \cup \{0\}$. We now show that
\begin{equation}\label{EQ:subst_comp1}
\begin{split}
\mathbb{P}[&\text{$0$ has a neighbor in $\mathbb{T}$ contained in a }\\
&\text{substantial component of $\mathbb{H} \cap E_\varphi^{\geq h_{\textnormal{as}}(1- 2\varepsilon)}$}] \geq 1 - c e^{-c' d^\varepsilon},
\end{split}
\end{equation}
for all $d \geq c$, which implies \eqref{P:SUBST_COMP1}. To this end, we enumerate the elements of $\mathbb{T}$ as a sequence $x_n$, $1 \leq n \leq |\mathbb{T}|$ in a hierarchical way, i.e. such that $|x_{n}|_1 \leq |x_{n+1}|_1$ for all $1 \leq n < |\mathbb{T}|$, and set $K_n = \{ x_1,\dots , x_{n-1}\}$, for $1 \leq n \leq |\mathbb{T}|$. We construct the field $(\varphi_x)_{x \in \mathbb{T}}$ using Lemma \ref{L:conds_exps2} according to this ordering. Thus, introducing a family $(\psi_x)_{x \in \mathbb{T}}$ of independent random variables with
\begin{equation*}
\psi_{x_n} \sim \mathcal{N}(0, g_{K_n^c}(x_n,x_n)), \text{ for all $1 \leq n \leq |\mathbb{T}|$,}
\end{equation*}
we define the field $(\varphi_x)_{x \in \mathbb{T}}$ in terms of $(\psi_x)_{x \in \mathbb{T}}$ as in \eqref{L:conds_exps2_phi}. Moreover, letting 
\begin{equation} \label{EQ:subst_comp_const}
c_3' = \max_{0 \leq k \leq b} c(k), \text{ with $c(k)$, $k \geq 0$, as appearing in \eqref{return_probas1}},
\end{equation} 
we introduce the independent events $\widehat{M}_n = \{ \psi_{x_n} \geq - \frac{\varepsilon }{ 4 c_3'} d h_{\text{as}}(d) \}$, for $1 \leq n \leq |\mathbb{T}|$, and define 
\begin{equation} \label{EQ:subst_comp_M}
M_n = \bigcap_{i=1}^n \widehat{M}_i = \Big\{ \min_{x \in K_{n+1}} \psi_{x} \geq - \frac{\varepsilon }{ 4 c_3'} d h_{\text{as}}(d) \Big\}, \text{ for $0 \leq n \leq  |\mathbb{T}|$, and }  M \stackrel{\text{def.}}{=}  M_{|\mathbb{T}|} 
\end{equation}
(with the convention $M_0 = \mathbb{R}^{\mathbb{Z}^d}$). Thus, the events $M_n$ decrease towards $ M$. Observing that $|\mathbb{T}| \leq c \lfloor d/ b \rfloor^b$, cf. \eqref{EQ:subst_comp_tree}, and that $\text{Var}(\psi_x) \geq 1$ for all $x \in \mathbb{T}$, \eqref{GRV_basic_estimate} yields that
\begin{equation*}
\mathbb{P}[M^c] \leq |\mathbb{T}| \cdot \sup_{x \in \mathbb{T}} \mathbb{P}[\psi_{x} > (4c_3')^{-1} \varepsilon d h_{\text{as}}(d)] \leq e^{-c d^2 \log d},
\end{equation*}
for all $d \geq 3$. Hence, by looking separately at the intersection of the event 
\begin{equation*}
\{\text{no neighbor of $0$ in $\mathbb{T}$ is contained in a substantial component of $\mathbb{H} \cap E_\varphi^{\geq h_{\textnormal{as}}(1- 2\varepsilon)}\}$}
\end{equation*}
with $M$ and its complement, respectively, we deduce that in order to prove \eqref{EQ:subst_comp1}, it suffices to show that
\begin{equation} \label{EQ:subst_comp3}
\begin{split}
\mathbb{P}[&\text{$0$ has a neighbor in $\mathbb{T} $  which is contained in a }\\
&\text{substantial component of $\mathbb{H} \cap E_\varphi^{\geq h_{\textnormal{as}}(1- 2\varepsilon)}$} | M] \geq 1 -  c e^{-c'd^\varepsilon}.
\end{split}
\end{equation}
holds for $d \geq c$. We now show that, conditionally on $M$, the collection of variables 
$$
Y_x \stackrel{\text{def.}}{=} 1\{\varphi_x \geq h_{\textnormal{as}}(1-2\varepsilon)\}, \quad x \in \mathbb{T},
$$
stochastically dominates a Galton-Watson chain on $\mathbb{T}$ with suitable (binomial) offspring distribution. More precisely, we prove the following.

\begin{lemma} \label{L:subt_comp_dom} $(0< \varepsilon < 1/3)$

\medskip
\noindent Let $\mathbf{b}_x$, $x \in \mathbb{T}$, denote a family of independent Bernoulli variables under some auxiliary probability measure $\mathbf{P}$ such that $\mathbf{P}[\mathbf{b}_x = 1]= 1- \mathbf{P}[\mathbf{b}_x = 0]= d^{-(1- \frac{3}{2}\varepsilon)}$ for $x\in \mathbb{T}$. Then, for all $d\geq c$,
\begin{equation} \label{L:subt_comp_dom1}
(Y_x)_{x \in \mathbb{T}} \circ \mathbb{P}[\, \cdot \, | M] \; \geq_{\textnormal{st.}} \; (\mathbf{b}_x)_{x \in \mathbb{T}} \circ \mathbf{P}. 
\end{equation}
\end{lemma}

\begin{proof4}
We begin with the following remark. From \eqref{L:conds_exps2_phi}, one immediately infers that $\varphi_{x_{n}} = \sum_{k=1}^{n} \alpha_{n,k} \psi_{x_k}$, for $1 \leq n \leq |\mathbb{T}|$, where
\begin{equation*} \label{L:subst_comp_alpha}
\alpha_{n,n}=1, \quad \alpha_{n,k} = \sum_{l=k}^{n-1} p_{x_n,x_l}^{K_n} \alpha_{l,k}, \text{ for $1 \leq n \leq |\mathbb{T}|$ and $1 \leq k < n$, }
\end{equation*}
with $p_{x_n,x_l}^{K_n}$, $1 \leq l < n$, as defined in \eqref{EQ:p_x,y^K}. We claim that
\begin{equation} \label{L:subst_comp_alpha_bound}
\sum_{k=1}^n  \alpha_{n,k} \leq 2, \text{ for all $ 1 \leq n \leq |\mathbb{T}|$ and $d \geq c$}.
\end{equation}
Indeed,
\begin{equation} \label{L:subst_comp_alpha}
\sum_{k=1}^n  \alpha_{n,k} = 1 + \sum_{l=1}^{n-1} p_{x_n,x_l}^{K_n}  \sum_{k=1}^l \alpha_{l,k} \leq 1+ P_{x_n}[H_{K_n} < \infty]  \cdot \sup_{1 \leq l < n } \sum_{k=1}^l \alpha_{l,k},
\end{equation}
for all $1 \leq n \leq |\mathbb{T}|$. Since the elements of the tree $\mathbb{T}$ are enumerated in a hierarchical fashion, $K_n \subset B_1(0, |x_{n-1}|_1)$ for all $ 1 < n \leq |\mathbb{T}|$, and therefore
\begin{equation} \label{L:subst_comp_hit_probas}
P_{x_{n}}[H_{K_n} < \infty] \leq P_{x_{n}}[\widetilde{H}_{B_1(0, |x_{n-1}|_1)} < \infty] \stackrel{\eqref{return_probas1}, \eqref{EQ:subst_comp_const}}{\leq} c_3'/d
\end{equation}
(by construction, $|z|_1 \leq b$, for $z \in \mathbb{T}$, and the depth $b$ does not depend on $d$, cf. \eqref{EQ:subst_comp_b}). In particular, this probability is less than $1/2$, uniformly in $n$ for all $ 1 < n \leq |\mathbb{T}|$, whenever $d \geq c$. Inserting this into 
\eqref{L:subst_comp_alpha} and a trivial inductive argument yield \eqref{L:subst_comp_alpha_bound}. 

We proceed with the proof of \eqref{L:subt_comp_dom1}. Let
$$
Z_x = 1\{ \psi_x \geq h_{\text{as}}(1-\frac{3}{2}\varepsilon) \}, \quad \text{ for $x \in \mathbb{T}$}.
$$
We first aim at showing that for all $d \geq c$,
\begin{equation}\label{L:subst_comp_dom11}
\begin{split}
&\text{the law of $(Y_x)_{x\in \mathbb{T}}$ under $\mathbb{P}[\, \cdot \, |M]$, stochastically} \\
&\text{dominates the law of $(Z_x)_{x\in \mathbb{T}}$ under $\mathbb{P}[\, \cdot \, |M]$}.
\end{split}
\end{equation}
To this end, we claim that for all $d \geq c$ and $ 1 \leq n \leq |\mathbb{T}|$,
\begin{equation} \label{L:subst_comp_dom1}
\{ \varphi_{x_{n}} \geq h_{\text{as}}(1-2\varepsilon) \} \supseteq M_{n-1} \cap \{ \psi_{x_n} \geq h_{\text{as}}(1-\frac{3}{2}\varepsilon) \}.
\end{equation}
Indeed, this is trivial for $n=1$ (recall that $\varphi_{x_1} = \psi_{x_1}$, cf. \eqref{L:conds_exps2_phi} and $M_0 = \mathbb{R}^{\mathbb{Z}^d})$. By construction, for all $ 1 < n \leq |\mathbb{T}|$, on the event $M_{n-1} \cap \{ \psi_{x_{n}} \geq h_{\text{as}}(1-\frac{3}{2}\varepsilon) \}$,
\begin{equation*}
\begin{array}{rcl}
\varphi_{x_{n}} \hspace{-1ex} & \stackrel{\eqref{L:conds_exps2_phi}}{\geq} & \hspace{-1ex} \psi_{x_{n}} +  P_{x_{n}}[H_{K_{n}} < \infty] \cdot \displaystyle \min_{1\leq i < n} \varphi_{x_i} \\
& \geq & \hspace{-1ex} \psi_{x_{n}} +  P_{x_{n}}[H_{K_{n}} < \infty] \cdot \displaystyle \min_{1\leq i < n} \big[ \sum_{k=1}^i \alpha_{i,k} \cdot \min_{1 \leq l < n} \psi_{x_l} \big] \\
&\stackrel{\eqref{EQ:subst_comp_M}}{\geq} & \hspace{-1ex} \displaystyle h_{\text{as}}(1-\frac{3}{2}\varepsilon) -  \frac{\varepsilon}{4 c_3'} d h_{\text{as}} \cdot P_{x_{n}}[H_{K_{n}} < \infty]  \cdot \max_{1\leq i < n}  \sum_{k=1}^i \alpha_{i,k}.
\end{array}
\end{equation*}
Inserting the bounds \eqref{L:subst_comp_alpha_bound}, \eqref{L:subst_comp_hit_probas} into the last line immediately yields that $\varphi_{x_n} \geq h_{\text{as}}(1-2\varepsilon)$, for $d\geq c$, and \eqref{L:subst_comp_dom1} follows. But since the sets $M_n$ decrease towards $M$, \eqref{L:subst_comp_dom1} actually implies that, 
$$
(M \cap\{ \varphi_{x_{n}} \geq h_{\text{as}}(1-2\varepsilon) \}) \, \supseteq \, (M \cap \{ \psi_{x_n} \geq h_{\text{as}}(1-\frac{3}{2}\varepsilon) \}),
$$
for all $1\leq n \leq |\mathbb{T}|$ and $d\geq c$, or, equivalently, that
$$
\mathbb{P}[Y_x \geq Z_x, \, \text{for all } x\in \mathbb{T}\,|\, M]=1, \text{ whenever $d\geq c$}.
$$
By a classical theorem of Strassen, see \cite{Stra}, the existence of this monotone coupling is equivalent to the asserted stochastic domination in \eqref{L:subst_comp_dom11}.

Finally, we explain how \eqref{L:subt_comp_dom1} follows from \eqref{L:subst_comp_dom11}. First, notice that by definition of $M$, see \eqref{EQ:subst_comp_M}, the variables $Z_x$, $x\in \mathbb{T}$ are still (conditionally) independent under $\mathbb{P}[\, \cdot \, |M]$. Next, since $1 \leq \mathbb{E}[\psi_{x_n}^2] \leq g(0)$ for all $1 \leq n \leq |\mathbb{T}|$, we can arrange that $\tilde{g}_n \stackrel{\text{def.}}{=} g(0)/ \mathbb{E}[\psi_{x_n}^2] \leq (1-2\varepsilon)/(1-\frac{3}{2}\varepsilon)^2$ (note that this last quantity is always greater than $1$ for $\varepsilon \in (0,1/3)$). Moreover, recalling the definition of $\widehat{M}_n$ above \eqref{EQ:subst_comp_M}, we have that
\begin{equation*}
\begin{array}{rcl}
\mathbb{P}[\psi_{x_n} \geq  h_{\text{as}}(1-\frac{3}{2}\varepsilon)| M] & = & \hspace{-1ex}\mathbb{P}[\psi_{x_n} \geq  h_{\text{as}}(1-\frac{3}{2}\varepsilon)| \widehat{M}_n] \\
 & \geq & \hspace{-1ex} \mathbb{P}[\psi_{x_n} \geq  h_{\text{as}}(1-\frac{3}{2}\varepsilon)] \\
\hspace{-1ex} & \stackrel{\eqref{GRV_basic_estimate}}{\geq} & \hspace{-1ex} \displaystyle \frac{c}{h_{\text{as}}(1-\frac{3}{2}\varepsilon)} \cdot d^{-\tilde{g}_n(1-\frac{3}{2}\varepsilon)^2}  \geq d^{-(1-\frac{3}{2}\varepsilon)},
\end{array}
\end{equation*} 
for all $d\geq c'$. A standard coupling then yields that the law of $(Z_x)_{x\in \mathbb{T}}$ under $\mathbb{P}[\, \cdot \, | M]$ dominates the law of $(\mathbf{b}_x)_{x\in \mathbb{T}}$ under $\mathbf{P}$, for all sufficiently large $d$. Together with \eqref{L:subst_comp_dom11}, this implies \eqref{L:subt_comp_dom1}, and thus completes the proof of Lemma \nolinebreak \ref{L:subt_comp_dom}.
\hfill $\square$
\end{proof4}
\bigskip
\noindent We continue with the proof of \eqref{EQ:subst_comp3}. We introduce, for $x \in \mathbb{T}_0 \cap B_1(0,b-1)$, the variable
\begin{equation*} 
\mathbf{N}(x) = \sum_{i \in I_{|x|_1 +1} } \mathbf{b}_{x + e_i},
\end{equation*}
which can be interpreted as the number of ``existing'' children of $x$ in the Galton-Watson chain $(\mathbf{b}_x)_{x\in \mathbb{T}}$. We will need the following estimate.
\begin{lemma} \label{L:subst_comp_gwbound} $( \varepsilon \in (0,1/3), \, d\geq c)$
\begin{equation}\label{EQ:subst_comp_gwbound}
\mathbf{P}[\mathbf{N}(x) < d^\varepsilon / b, \text{ for some $x \in \mathbb{T}_0 \cap B_1(0,b-1)$}] \leq  e^{-c d^\varepsilon}.
\end{equation}
\end{lemma}

\begin{proof5}
Fix some $x \in \mathbb{T}_0 \cap B_1(0,b-1)$ and $\varepsilon \in (0,1)$. It suffices to show that 
\begin{equation} \label{EQ:chern1}
\mathbf{P}[\mathbf{N}(x) < d^\varepsilon / b] \leq  e^{-c d^\varepsilon}, \text{ when $d \geq c$,}
\end{equation}
for \eqref{EQ:subst_comp_gwbound} then follows with a simple union bound, observing that $|\mathbb{T}_0| \leq c \lfloor d/b \rfloor^b$. Since $\mathbf{N}(x)$ is a sum of independent $\{0,1\}$-valued random variables, a classical Chernov estimate (see for example \cite{Kle}, Ch. 5.3, p.111) gives
$$
\mathbf{P}[\mathbf{N}(x) < (1-\delta)\mathbf{E}[\mathbf{N}(x)]\, ]\leq e^{-\frac{\delta^2 \mathbf{E}[\mathbf{N}(x)]}{2}}.
$$
Observing that $\mathbf{E}[\mathbf{N}(x)]\geq \lfloor \frac{d}{b}  \rfloor d^{-(1-\frac{3}{2}\varepsilon)} \geq c d^{\frac{3}{2}\varepsilon}$, for all $x\in \mathbb{T}$, this bound (with, say, $\delta =\frac{1}{2}$) is more than enough to deduce \eqref{EQ:chern1}. \hfill $\square$
\bigskip
\end{proof5}

\noindent With Lemmas \ref{L:subt_comp_dom} and \ref{L:subst_comp_gwbound} at hand, the claim \eqref{EQ:subst_comp3} follows readily. First, observe that if the event  $\{ \mathbf{N}(x) \geq d^\varepsilon / b, \text{ for all $x \in \mathbb{T}_0 \cap B_1(0,b-1)$} \}$ occurs, then the origin has a neighbor in $\mathbb{T}_0$ (in fact even $ \lfloor d^\varepsilon / b \rfloor$ such) which belongs to a connected component of $\{ x\in \mathbb{T}; \, \mathbf{b}_x =1 \}$ containing at least
\begin{equation*}
1 + \sum_{k=2}^b \Big\lfloor \frac{d^\varepsilon}{ b} \Big\rfloor^{k-1} \geq \;  \Big\lfloor \frac{d^\varepsilon}{ b} \Big\rfloor^{b-1}
\end{equation*}
points, i.e., a substantial component of $\mathbb{T}$ (recall the definition below \eqref{EQ:subst_comp_b}). Hence, we obtain, for all $d\geq c$, applying Lemmas \ref{L:subt_comp_dom} and \ref{L:subst_comp_gwbound},
\begin{equation*}
\begin{array}{cl}
 & \hspace{-1ex}\mathbb{P}[\text{$0$ has a neighbor in $\mathbb{T}_0 $ which is contained in a substantial component} \\
 & \hspace{-1ex }\text{of $\mathbb{H} \cap E_\varphi^{\geq h_{\textnormal{as}}(1- 2\varepsilon)}$} | M] \\
\stackrel{\eqref{L:subt_comp_dom1}}{\geq} 
&\mathbf{P}[\text{$0$ has a child in $\mathbb{T}_0 $  which is contained in a substantial component}\\
 & \hspace{-1ex} \text{of $\{x\in \mathbb{T} ; \, \mathbf{b}_x = 1 \}$}] \\
 \geq & \hspace{-1ex} \mathbf{P}[\mathbf{N}(x) \geq d^\varepsilon / b, \text{ for all $x \in \mathbb{T}_0 \cap B_1(0,b-1)$}] \\
\stackrel{\eqref{EQ:subst_comp_gwbound}}{\geq} & \hspace{-1ex} 1 -  e^{-cd^\varepsilon},
\end{array}
\end{equation*}
which is \eqref{EQ:subst_comp3}. The proof of Theorem \ref{P:SUBST_COMP} is now complete.
\end{proof}

\noindent For future reference, we introduce  the random set
\begin{equation}\label{EQ:event_E}
\begin{split}
B_{\varepsilon}=
&\{ x\in \mathbb{H}; \; \text{$x$ has no neighbor in $\mathbb{H}$ contained} \\
&\text{in a substantial component of $\mathbb{H} \cap E_{\varphi}^{\geq h_{\text{as}}(1-2\varepsilon)}$}\}, 
\end{split}
\end{equation}
for arbitrary $\varepsilon \in (0,1/3)$, and think of the points in $B_{\varepsilon}$ as \textit{bad} points in $\mathbb{H}$. Theorem \ref{P:SUBST_COMP} has the following immediate
\begin{corollary}\label{C:number_subst_comps}
$(d\geq 3$, $0< \varepsilon<1/3)$
\begin{equation}\label{EQ:number_subst_comps}
\mathbb{P}[|B_{\varepsilon}| > |\mathbb{H}|e^{-c_4'd^\varepsilon}] \leq ce^{-c'd^\varepsilon}.
\end{equation}
\end{corollary}
\begin{proof}
On account of \eqref{P:SUBST_COMP1}, Chebyshev's inequality in the form 
\begin{align*}
\mathbb P [ |B_{\varepsilon}| > \lambda] \leq \lambda^{-1} \sum_{x\in \mathbb{H}} \mathbb{P}[&x \text{ has no neighbor in $\mathbb{H}$ contained in a} \\
&\text{substantial component of $\mathbb{H} \cap E_{\varphi}^{\geq h_{\text{as}}(1-2\varepsilon)}$]},
\end{align*}
with $\lambda =  |\mathbb{H}|e^{-c_4'd^\varepsilon}$ and suitable $c_4'>0$, readily yields \eqref{EQ:number_subst_comps}.
\end{proof}

\subsection{Connecting substantial components} \label{S:CONNECTING_SUBST_COMPS}

In this subsection, we complete the proof of Theorem \ref{T:ASYMPTOTICS_LB}. In what follows, and in accordance with our definition in the paragraph following \eqref{EQ:goodEventDef}, given a (random) set $K \subset \mathbb{H}_x$, we call any connected subset $C$ of $K$ satisfying $|\overline{C}^{\mathbb{H}_x}|\geq (1-d^{-2})2^d$ a \textit{giant} component of $K$.

We will prove Theorem \ref{T:ASYMPTOTICS_LB} by verifying the finite-size criterion \eqref{EQ:local_global} when $h(d)= h_{\text{as}}(d)(1-8\varepsilon)$, for all sufficiently small $\varepsilon > 0$. In order to deduce that the event $G_0^{h_{\text{as}}(1-8\varepsilon)}$ appearing in this context (recall \eqref{EQ:goodEventDef}) occurs with sufficiently high probability, we will use isoperimetry considerations to patch together the substantial components of $E_{\varphi}^{\geq h_{\text{as}}(1-2\varepsilon)}$ we have just constructed, see Corollary \ref{C:number_subst_comps} above, first to form a giant component in $E_{\varphi}^{\geq h_{\text{as}}(1-5\varepsilon)}\cap \mathbb{H}$, and then to connect the latter to neighboring giant components within $E_{\varphi}^{\geq h_{\text{as}}(1-8\varepsilon)}$. This is the object of Theorem \ref{T:giant_comp}. The lower bound \eqref{EQ:ASYMPTOTICS_LB} then follows readily, by virtue of Theorem \ref{T:LOCAL_GLOBAL}.

\begin{theorem} \label{T:giant_comp} 
$(d\geq 3,$ $0< \varepsilon<1/3)$
\begin{equation} \label{EQ:giant_comps_conn}
\mathbb{P} [ G_0^{h_{\textnormal{as}}(1-8\varepsilon)}] \geq 1 - c e^{-c'd^{\varepsilon}}.
\end{equation}
\end{theorem}

\begin{proof}
Fix $\varepsilon \in (0, 1/3)$. We start by showing
\begin{equation} \label{EQ:giant_comp}
\mathbb{P}[E_\varphi^{\geq h_{\text{as}}(1-5\varepsilon)} \cap \mathbb{H}_x \text{ contains a giant component}] \geq  1- c e^{-c'd^{\varepsilon}},
\end{equation}
for all $d\geq 3$ and $x\in \mathbb{Z}^2$. By translation invariance, it suffices to consider the case $x=0$. We denote by $S= (S_1,\dots, S_{N_{\text{subst}}})$ (with $S = \emptyset$ if $N_{\text{subst}}=0$) the collection of substantial components of $\mathbb{H} \cap E_{\varphi}^{\geq h_{\text{as}}(1-2\varepsilon)}$ (we assume for sake of definiteness that $S_1,\dots, S_{N_{\text{subst}}}$ are enumerated according to a specified procedure, e.g. using the lexicographic order induced by the points closest to the origin in each component). With a slight abuse of notation, $S$ will also be used to denote the set $\bigcup_{i=1}^{N_{\text{subst}}} S_i$, but the meaning will always be clear from the context. By definition, see \eqref{EQ:subst_comp_b}, the random sets $S_i$ satisfy $|S_i| \geq d^{10}$, for all $1\leq i \leq N_{\text{subst}}$, whenever $d \geq c$, and $\overline{S}_i^{ \mathbb{H}} \cap S_j = \emptyset$ for all $i \neq j$. We define the following set of partitions of $S$,
\begin{equation*}
\begin{split}
\mathbf{P}(S) = \
&\Big\{ \{K,K'\} ; \; K = \bigcup_{i \in I} S_i, \; K' = \bigcup_{i \in \{ 1,\dots,N_{\text{subst}}\} \setminus I}
 S_i \text{, for} \\ 
 &\text{some $I \subset \{ 1,\dots, N_{\text{subst}} \}$, and } |K| \wedge |K'| \geq d^{-4}|\mathbb{H}| \Big\}.
 \end{split} 
\end{equation*}
Since the number $N_{\text{subst}}$ of substantial components of $\mathbb{H}\cap E_{\varphi}^{\geq h_{\text{as}}(1-2\varepsilon)}$ is bounded by $d^{-10}|\mathbb{H}|$, when $d \geq c$, the (random) partition set  $\mathbf{P}(S)$ satisfies
\begin{equation} \label{EQ:giant_partitions}
|\mathbf{P}(S)|  \leq 2^{d^{-10}|\mathbb{H}|}, \text{ for all $d \geq c$}. 
\end{equation}
The main step towards proving \eqref{EQ:giant_comp} will consist of showing that the event 
\begin{equation}\label{EQ:giant_event_H}
\mathcal{H} = \bigcap_{\{ K,K'\} \in \mathbf{P}(S)} \{ K \stackrel{ \geq h_\text{as}(1-5\varepsilon)}{\longleftrightarrow} K' \}
\end{equation}
(with the convention that $\mathcal{H} $ is the whole space $\mathbb{R}^{\mathbb{Z}^d}$ whenever  $\mathbf{P}(S) = \emptyset$) occurs with sufficiently high probability. In words, $\mathcal{H}$ is the event that for any partition of the substantial components of $\mathbb{H} \cap E_{\varphi}^{\geq h_{\text{as}}(1-2\varepsilon)}$ into two ``sizeable'' classes (in the sense that the cardinality of the respective aggregate unions of substantial components is at least $d^{-4}|\mathbb{H}|$), one can find a substantial component in each class such that the two are connected in $E_\varphi^{\geq h_\text{as}(1-5\varepsilon)}$.

\begin{lemma} \label{L:giant_event_H} $(d\geq 3)$
\begin{equation} \label{EQ:giant_event_H_bound}
\mathbb{P}[\mathcal{H}] \geq 1- c e^{-c'd^{\varepsilon}}.
\end{equation}
\end{lemma}
\begin{proof6}
Let $\Pi \subset 2^\mathbb{H}$ denote the set of singletons and nearest-neighbor edges in $\mathbb{H}$, i.e. if $U \in \Pi$, then either $U=\{ x \}$ for some $x \in \mathbb{H}$ or $U = \{ x,y \}$ with $x,y \in \mathbb{H}$ and $x \sim y$. For  $U \in \Pi$, we define 
\begin{equation} \label{EQ:giant_E_U}
\mathcal{G}_U = \big\{ \sup_{x \in U} \big|\mu_x^{U^c} \big| < \varepsilon h_{\text{as}}  \big\}, 
\end{equation}
with $\mu_x^{U^c} =   \sum_{z \in U^c} P_x [H_{U^c} < \infty, X_{H_{U^c}} = z] \varphi_z$, for $x \in U$, as in \eqref{mu}. Since a non-vanishing contribution to the previous sum arises only from the points in $\partial_{\text{out}}U$, and since $|U| \leq 2$, Lemma \ref{L:zeta_i_bound} applies (with $\ell=4$ and $K,U$ appearing therein both equal to $\partial_{\text{out}}U$ here), thus yielding
\begin{equation}\label{EQ:giant_E_U_bound}
\mathbb{P}[\mathcal{G}_U^c] \leq e^{-c d \log d}, \text{ for all $U \in \Pi$, $d \geq 3$}.
\end{equation}
We also introduce the (good) event
\begin{equation} \label{EQ:giant_good_event}
\mathcal{G} = \{ |B_{\varepsilon}| \leq |\mathbb{H}|e^{-c_4'd^\varepsilon} \} \cap \bigcap_{U \in \Pi} \mathcal{G}_U, 
\end{equation}
(recall \eqref{EQ:event_E} for the definition of $B_{\varepsilon}$). On account of \eqref{EQ:number_subst_comps} and \eqref{EQ:giant_E_U_bound}, a union bound yields, for all $d \geq 3$,
\begin{equation} \label{EQ:giant_good_event_bound}
\mathbb{P}[\mathcal{G}^c] \leq ce^{-c' d^{\varepsilon}} + (1+ d)|\mathbb{H}| \cdot e^{-c'' d \log d} \leq c'''e^{-c' d^{\varepsilon}}.
\end{equation}
It will be convenient to specify configurations of the level set above $h_{\text{as}}(1-2\varepsilon)$ in the hypercube. Thus, given $K_+ \subset \mathbb{H}$, we abbreviate $\mathcal{C}(K_+) = \{ E_\varphi^{\geq h_{\text{as}}(1-2\varepsilon)} \cap \mathbb{H} = K_+ \}$, and write
\begin{equation}\label{EQ:giant_event_H_bound1}
\begin{array}{rcl}
\mathbb{P}[\mathcal{H}^c] \hspace{-1ex} & \leq & \hspace{-1ex} \displaystyle \mathbb{P}[\mathcal{G}^c] + \sum_{K_+ \subset \mathbb{H}} \mathbb{P}[\mathcal{G}, \; \mathcal{H}^c, \; \mathcal{C}(K_+)]   \\
\hspace{-1ex} & \stackrel{\eqref{EQ:giant_event_H}}{\leq} & \hspace{-1ex} \mathbb{P}[\mathcal{G}^c] + \displaystyle \sum_{K_+ \subset \mathbb{H}} |\mathbf{P}(S_{K_+})| \sup_{\{K,K'\} \in \mathbf{P}(S_{K_+})} \mathbb{P}[\mathcal{G}, K \nleftrightarrow K' \text{ in $E_\varphi^{\geq h_\text{as}(1-5\varepsilon)}$}, \, \mathcal{C}(K_+)] \\ 
\hspace{-1ex} & \stackrel{\eqref{EQ:giant_partitions}}{\leq} & \hspace{-1ex} \mathbb{P}[\mathcal{G}^c] + \displaystyle 2^{d^{-10}|\mathbb{H}|} \sup_{K_+ \subset \mathbb{H}} \; \sup_{\{K,K'\} \in \mathbf{P}(S_{K_+})} \mathbb{P}[\mathcal{G}, K \nleftrightarrow K' \text{ in $E_\varphi^{\geq h_\text{as}(1-5\varepsilon)}$} \, | \, \mathcal{C}(K_+)],
\end{array}
\end{equation}
for all $d \geq c$, where the set $S_{K_+}$ in the second and third line refers to the (deterministic) family of substantial components associated to the configuration $\mathcal{C}(K_+)$. In order to  bound the conditional probability appearing on the right-hand side of \eqref{EQ:giant_event_H_bound1}, we rely on isoperimetry considerations for subsets of the hypercube by Bollob\'as and Leader \cite{BL}. Thus, let $K_+ \subset \mathbb{H}$ be such that $\mathbf{P}(S_{K_+})$ is non-empty and $\{K,K'\} \in \mathbf{P}(S_{K_+})$ be fixed. First, observe that $B_\varepsilon$ is a deterministic set under $\mathbb{P}[\, \cdot \, |\; \mathcal{C}(K_+)]$. By construction, 
\begin{equation} \label{EQ:isoperimetry_H_cond1}
{B}_\varepsilon \cup \overline{K}^{\mathbb{H}} \cup \overline{K'}^{\mathbb{H}} = \mathbb{H} \text{ and } K \cap \overline{K'}^{\mathbb{H}} = K' \cap \overline{K}^{\mathbb{H}} = \emptyset
\end{equation}
(recall that $ \overline{K}^{\mathbb{H}}$ denotes the $\ell^1$-closure of $K$ in $\mathbb{H}$).
Moreover, by definition of $\mathbf{P}(S_{K_+})$, and on the event $\mathcal{G}$, cf. \eqref{EQ:giant_good_event},
\begin{equation} \label{EQ:isoperimetry_H_cond2}
|K| \wedge |K'| \geq d^{-4}|\mathbb{H}| \text{ and } |B_\varepsilon| \leq e^{-c_4' d^{\varepsilon}}|\mathbb{H}|.
\end{equation}
On account of \eqref{EQ:isoperimetry_H_cond1} and \eqref{EQ:isoperimetry_H_cond2}, Lemma 4.3 in \cite{SLB} (itself a consequence of Corollary 4 in \cite{BL}) yields that for all $d \geq c$, 
\begin{equation}\label{EQ:isoperimetry_H_1}
\begin{split}
&\text{there exist disjoint subsets $U'_1,\dots, U'_{m'}$ in $ \Pi \cap (\mathbb{H}\setminus B_\varepsilon)$ with $m' \geq c d^{-6}|\mathbb{H}|$} \\
&\text{such that for all $1 \leq k \leq m'$, either $U'_k = \{y_k\}$ with $y_k \in \partial_{\text{out}}K \cap \partial_{\text{out}}K'$,} \\
&\text{or $U'_k = \{y_k, z_k \}$ with $y_k \in \partial_{\text{out}}K$, $z_k \in \partial_{\text{out}}K' $ (and $|y_k -z_k|_1 = 1$)}.
\end{split}
\end{equation}
Among the family $U'_k$, $1 \leq k \leq m'$, we may select $m \geq c d^{-7}|\mathbb{H}|$ sets $U_1, \dots , U_m$ satisfying the additional assumption
\begin{equation} \label{EQ:isoperimetry_H_2}
d_1(U_k, U_l) \geq 2, \text{ for all $1 \leq k< l \leq m$}.
\end{equation}
Now, if the event $\mathcal{G} \cap \{ K \nleftrightarrow K' \text{ in $E_\varphi^{\geq h_\text{as}(1-5\varepsilon)}$} \}$ occurs (conditionally on $\mathcal{C}(K_+)$), then by construction, the field value in at least one of the sites in each set $U_k$ cannot exceed $h_\text{as}(1-5\varepsilon)$, for otherwise $U_k$ forms a path connecting $K$ to $K'$ in the level set $E_\varphi^{\geq h_\text{as}(1-5\varepsilon)}$, cf. \eqref{EQ:isoperimetry_H_1}. Thus, setting $\mathcal{F}_k = \bigcup_{y\in U_k} \{ \varphi_y < h_\text{as}(1-5\varepsilon) \}$, for $1 \leq k \leq m$, and $\widetilde{\mathcal{G}}=  \bigcap_{k=1}^m \mathcal{G}_{U_k}$, which contains $\mathcal{G}$, see \eqref{EQ:giant_good_event}, we obtain, for all $d \geq c$,
\begin{equation} \label{EQ:giant_event_H_bound2}
\begin{split}
&\mathbb{P}[\mathcal{G}, K \nleftrightarrow K' \text{ in $E_\varphi^{\geq h_\text{as}(1-5\varepsilon)}$} \, | \, \mathcal{C}(K_+)] \\
&\qquad \leq \mathbb{P} [\widetilde{\mathcal{G}} ,  \mathcal{F}_k, 1\leq k \leq m \, | \, \mathcal{C}(K_+)] \\
& \qquad = \Big( \prod_{k=1}^m \mathbb{P} [ \mathcal{F}_k \, | \,  \widetilde{\mathcal{G}}, \,\mathcal{F}_l, k < l \leq m, \, \mathcal{C}(K_+) ] \Big) \cdot \mathbb{P}[\widetilde{\mathcal{G}} \, | \, \mathcal{C}(K_+)].
\end{split}
\end{equation}
Next, we consider a single factor $\mathbb{P} [ \mathcal{F}_k \, | \,  \widetilde{\mathcal{G}}, \,\mathcal{F}_l, k < l \leq m, \, \mathcal{C}(K_+) ]$ in this product, and show that it doesn't converge too rapidly (in terms of $d$) to $1$ as $d \to \infty$. By definition of $\mathcal{G}_{U_k}$, cf. \eqref{EQ:giant_E_U}, and on account of \eqref{EQ:isoperimetry_H_2}, the event $\widetilde{\mathcal{G}}$ is measurable with respect to the $\sigma$-algebra generated by $\varphi_z$,  $z \in \overline{\mathbb{H}} \setminus \big( \bigcup_{1\leq k \leq m} U_k \big)$, (recall that $\overline{\mathbb{H}}$ denotes the $\ell^1$-closure of $\mathbb{H}$), hence for every $1 \leq k \leq m$,
\begin{equation*}
\mathcal{D}_k = \{ \widetilde{\mathcal{G}}, \mathcal{F}_l, k < l \leq m, \varphi_{|_{K_+}} \geq h_{\text{as}}(1-2\varepsilon), \varphi_{|_{\mathbb{H} \setminus (K_+ \cup U_k)}} < h_{\text{as}}(1-2\varepsilon) \} \in \sigma(\varphi_z ; \; z \in \overline{\mathbb{H}} \setminus U_k)
\end{equation*}
(here and in what follows, we use the shorthand $\{ \varphi_{|_K} \in B\} = \bigcap_{x \in K} \{ \varphi_x \in B\}$, for any Borel set $B \subset \mathbb{R}$ and $K \subset \mathbb{Z}^d$). Since all elements $y \in U_k$ lie on the exterior boundary of a substantial component of $E_{\varphi}^{\geq h_{\text{as}}(1-2\varepsilon)}$,  we necessarily have that $\varphi_y < h_{\text{as}}(1-2\varepsilon)$. Thus, for all $1 \leq k \leq m$,
\begin{equation*}
\mathbb{P} [ \mathcal{F}_k \, | \,  \widetilde{\mathcal{G}}, \mathcal{F}_l, k < l \leq m, \mathcal{C}(K_+) ] =  1-  \frac{\mathbb{P} [ \,  \mathbb{P} [\varphi_{|_{U_k}} < h_{\text{as}}(1-2\varepsilon) , \mathcal{F}_k^c\, | \, (\varphi_x)_{x \in \overline{\mathbb{H}} \setminus U_k}], \, \mathcal{D}_k]}{\mathbb{P} [ \, \mathbb{P} [\varphi_{|_{U_k}} < h_{\text{as}}(1-2\varepsilon) \, | \, (\varphi_x)_{x \in \overline{\mathbb{H}} \setminus U_k}], \, \mathcal{D}_k]}.
\end{equation*}
Recall from \eqref{phi_cond_exps} that $\mathbb{P}[\, \cdot \, | \, (\varphi_x)_{ x \in \mathbb{Z}^d \setminus U_k}] = \widetilde{\mathbb{P}}[\widetilde{\varphi}^{k} +\mu^{U_k^c} \in \cdot \,]$, where we have abbreviated  $\widetilde{\varphi}^{U_k^c}$ by $ \widetilde{\varphi}^{k}$, which we view as a 2-dimensional Gaussian vector with covariance $g_{U_k}(\cdot, \cdot)$ under $\widetilde{\mathbb{P}}$. Thus, by definition of the event $\mathcal{G}_{U_k}$ in \eqref{EQ:giant_E_U}, we obtain, for all $1 \leq k \leq m$, all \textit{decreasing} events $A \subset \{0, 1\}^{U_k}$ and $h \in \mathbb{R}$, that $\mathbb{P}$-a.s. on \nolinebreak $\mathcal{G}_{U_k}$,
\begin{equation*}
\begin{split}
&\widetilde{\mathbb{P}}[A^{h - \varepsilon h_{\text{as}}}(\widetilde{\varphi}^{k})] \leq \widetilde{\mathbb{P}}[A^h(\widetilde{\varphi}^{k}+ |\mu^{U_k^c}|)]  \leq \mathbb{P}[A^h \, | \, (\varphi_x)_{ x \in \mathbb{Z}^d \setminus U_k}] \leq \widetilde{\mathbb{P}}[A^h(\widetilde{\varphi}^{k}- |\mu^{U_k^c}|)] \leq \widetilde{\mathbb{P}}[A^{h + \varepsilon h_{\text{as}}}(\widetilde{\varphi}^{k})] 
\end{split}
\end{equation*}
(cf. \eqref{EQ:event_A_Ah} for notation). We apply this separately to the numerator and denominator above. To do this, we note that $\mathcal{D}_k \subset \widetilde{\mathcal{G}} \subset \mathcal{G}_{U_k}$, for all $1 \leq k \leq m$, and that, despite not being decreasing, the event $\{ \varphi_{|_{U_k}} < h_{\text{as}}(1-2\varepsilon) \} \cap \mathcal{F}_k^c = \{h_{\text{as}}(1-5\varepsilon)  \leq \varphi_{|_{U_k}} < h_{\text{as}}(1-2\varepsilon) \}$ appearing in the numerator can be rewritten using $1\{ \varphi_{|_{U_k}} \in [h,h')\}  = 1\{ \varphi_{|_{U_k}} <h' \} - 1\{ \bigcup_{y \in U_k} \{ \varphi_{|_{U_k}} <h\} \}$, for all $h<h'$ (the events appearing on the right-hand side are both decreasing). All in all, we infer, setting $a = h_{\text{as}}(1-4\varepsilon)$ and $ b = h_{\text{as}}(1-3\varepsilon)$, that 
\begin{equation} \label{EQ:giant_event_H_bound_3}
\begin{split}
\mathbb{P} [ \mathcal{F}_k \, | \,  \widetilde{\mathcal{G}}, \mathcal{F}_l, k < l \leq m, \mathcal{C}(K_+) ] 
&\leq 1-\frac{\widetilde{\mathbb{P}}[\widetilde{\varphi}^{k}  \in [a,b)]\cdot \mathbb{P}[\mathcal{D}_k]}{ \widetilde{\mathbb{P}}[\widetilde{\varphi}^{k} < h_{\text{as}}(1-\varepsilon)] \cdot \mathbb{P}[\mathcal{D}_k]} \\
&\leq 1 -  \widetilde{\mathbb{P}}[\widetilde{\varphi}^{k}  \in [a,b)], \\
\end{split}
\end{equation} 
for all $1 \leq k \leq m$ (in writing $\widetilde{\mathbb{P}}[\widetilde{\varphi}^{k}  \in [a,b)]$, we obviously mean that both components of $\widetilde{\varphi}^{k}$ should lie in $[a,b)$). Moreover, letting $U_k = \{y_k, z_k \}$, and denoting by $\Phi$ the distribution function of a standard Gaussian variable, we can bound this probability as
\begin{equation*}
\begin{split}
\widetilde{\mathbb{P}}[\widetilde{\varphi}^{k}  \in [a,b)] &=  \widetilde{\mathbb{P}}\big[ \, \widetilde{\mathbb{P}}[a<\widetilde{\varphi}^{k}_{y_k} <b \, | \, \widetilde{\varphi}^{k}_{z_k} ], \,a<\widetilde{\varphi}^{k}_{z_k} <b \big] \\
&= \widetilde{\mathbb{E}}\Big[\Big(\Phi\Big(b - \frac{1}{2d}\widetilde{\varphi}^{k}_{z_k} \Big) - \Phi \Big(a - \frac{1}{2d} \widetilde{\varphi}^{k}_{z_k}\Big)\Big)1\{a<\widetilde{\varphi}^{k}_{z_k} <b \} \Big]\\
& \geq  (\Phi(b)-\Phi(a)) \cdot \widetilde{\mathbb{P}}[ a<\widetilde{\varphi}^{k}_{z_k} <b] \geq c \big((b - a) e^{- {b}^2/ 2 } \big)^2 \geq c' h_{\text{as}}^2 \, d^{-2} \geq d^{-2},
\end{split}
\end{equation*}
for all $d\geq c$ and $1 \leq k \leq m$ (here, the first inequality in the third line follows because the linear shift $\frac{1}{2d}\widetilde{\varphi}^{k}_{z_k}$ produced by conditioning on $\widetilde{\varphi}^{k}_{z_k}$ is in fact between, say, $0$ and $1$, when $d$ is large enough, since $\widetilde{\varphi}^{h}_{z_k} \in [a,b)$, while $a,b \to \infty$ as $d \to \infty$). Inserting this bound into \eqref{EQ:giant_event_H_bound_3}, and in view of \eqref{EQ:giant_E_U_bound}, \eqref{EQ:giant_event_H_bound2}, we obtain, for all $d \geq c$, $K_+ \subset \mathbb{H}$ and $\{ K, K' \} \in \mathbf{P}(S_{K_+})$,
\begin{equation*}
\mathbb{P}[\mathcal{G}, K \nleftrightarrow K' \text{ in $E_\varphi^{\geq h_\text{as}(1-3\varepsilon)}$} \; | \;\mathcal{C}(K_+)] \leq (1-d^{-2})^{c d^{-7} |\mathbb{H}|} \leq 2^{-cd^{-9}|\mathbb{H}|},
\end{equation*}
where we have also used that $m$ as appearing in \eqref{EQ:giant_event_H_bound2} is bounded from below by $c d^{-7} |\mathbb{H}|$, cf. above \eqref{EQ:isoperimetry_H_2}. Finally, going back to \eqref{EQ:giant_event_H_bound1}, and on account of \eqref{EQ:giant_good_event_bound}, we see that $\mathbb{P}[\mathcal{H}^c] \leq ce^{-c' d^\varepsilon}$, for all $d \geq c''$, and thus for all $d \geq 3$ by adjusting the constant $c$. This completes the proof of Lemma \ref{L:giant_event_H}. \hfill $\square$

\bigskip
\noindent
\end{proof6}
\noindent As we will now see, \eqref{EQ:giant_comp} follows from Lemma \ref{L:giant_event_H} by virtue of a counting argument. Specifically, it suffices to show that for all $d\geq c$,
\begin{equation}\label{EQ:giant_comb_1}
(\mathcal{H} \cap \{ |B_\varepsilon| \leq |\mathbb{H}|e^{-c_4' d^\varepsilon} \}) \subseteq \{ E_\varphi^{\geq h_{\text{as}}(1-5\varepsilon)} \cap \mathbb{H} \text{ contains a giant component} \},
\end{equation}
which, on account of \eqref{EQ:number_subst_comps} and \eqref{EQ:giant_event_H_bound}, implies \eqref{EQ:giant_comp}. We now show \eqref{EQ:giant_comb_1}. Recall that $S_1,\dots, S_{N_{\text{subst}}}$ are the (ordered) substantial components of $E_\varphi^{\geq h_{\text{as}}(1-2\varepsilon)} \cap \mathbb{H}$, and denote their union by $S$. By definition of $B_{\varepsilon}$, cf. \eqref{EQ:event_E}, 
\begin{equation*}
|S| = |\overline{S}^{\mathbb{H}}| - |\partial_{\text{out}}S \cap \mathbb{H}| \geq |\mathbb{H}\setminus B_\varepsilon| - d|S|,
\end{equation*}
and therefore certainly
\begin{equation}\label{EQ:giant_comb_2}
|S| \geq 4 d^{-4}|\mathbb{H}|, \text{ on the event $ \{ |B_\varepsilon| \leq |\mathbb{H}|e^{-c_4' d^\varepsilon} \}$, for $d\geq c$}.
\end{equation}
Consider the equivalence relation $\sim$ on the set $\{1,\dots, N_{\text{subst}}\}$ defined as $i \sim j$ if and only if $\{ S_i \stackrel{ \geq h_{\text{as}}(1-5\varepsilon)}{\longleftrightarrow} S_j \}$, and let $\pi$ be the (random) partition of $\{1,\dots, N_{\text{subst}}\}$ induced by its equivalence classes. Denoting by $J_{\text{max}}$ the equivalence class maximizing the quantity $\Big| \bigcup_{j\in J}  S_j \Big|$, for $J \in \pi$, we claim that on the event $\mathcal{H} \cap \{ |B_\varepsilon| \leq |\mathbb{H}|e^{-c_4' d^\varepsilon} \}$ and for $d\geq c$,
\begin{align}
& \Big| \bigcup_{j\in J_{\text{max}}}  S_j \Big| \geq d^{-4}|\mathbb{H}|, \label{EQ:giant_comb_3} \\
& \Big| \bigcup_{j\in  \{1,\dots, N_{\text{subst}} \} \setminus J_{\text{max}}}  S_j \Big| < d^{-4}|\mathbb{H}| \label{EQ:giant_comb_4}.
\end{align}
Indeed, if \eqref{EQ:giant_comb_3} did not hold, i.e. $\Big| \bigcup_{j\in J}  S_j \Big| < d^{-4}|\mathbb{H}|$ for all $J \in \pi$, on account of \eqref{EQ:giant_comb_2}, one could easily construct a set $I \subset \{ 1,\dots, N_{\text{subst}} \}$  obtained as union of some of the sets $J \in \pi$ with the property $ d^{-4}|\mathbb{H}| \leq \Big| \bigcup_{i\in I}  S_i \Big| < 2 d^{-4}|\mathbb{H}|$. But by \eqref{EQ:giant_comb_2}, this would yield, setting $K =  \bigcup_{i\in I}  S_i, $
$K' = \bigcup_{i\in  \{1,\dots, N_{\text{subst}} \} \setminus  I}  S_i$, that $\{ K,K' \} \in \mathbf{P}(S)$. Because $\mathcal{H}$ is assumed to occur, this would imply that one could find substantial components $S$ and $S'$ belonging to $K$ and $K'$ respectively, such that $\{ S \stackrel{ \geq h_{\text{as}}(1-5\varepsilon)}{\longleftrightarrow} S' \}$, thus contradicting the definition of $\sim$ (observe that by construction, $S$ and $S'$ belong to different equivalence classes). Hence \eqref{EQ:giant_comb_3} holds. Similarly, \eqref{EQ:giant_comb_3} implies \eqref{EQ:giant_comb_4}, for the joint occurrence of $\mathcal{H}$, \eqref{EQ:giant_comb_3} and $ \{ \big| \bigcup_{j\in  \{1,\dots, N_{\text{subst}} \} \setminus J_{\text{max}}}  S_j \big| \geq d^{-4}|\mathbb{H}| \}$ would also violate the definition of $\sim$. Having established \eqref{EQ:giant_comb_3} and \eqref{EQ:giant_comb_4}, we observe that on $\mathcal{H} \cap \{ |B_\varepsilon| \leq |\mathbb{H}|e^{-c_4' d^\varepsilon} \}$ and for $d\geq c$,
\begin{equation*}
\begin{split}
\Big|\overline{\bigcup_{j \in J_{\text{max}}} S_j}^{\mathbb{H}} \Big| &\geq |\mathbb{H}| - |B_\varepsilon| - \Big|\overline{\bigcup_{j \in  \{1,\dots, N_{\text{subst}} \} \setminus J_{\text{max}}} S_j}^{\mathbb{H}} \Big| \\
&\geq |\mathbb{H}| - |\mathbb{H}|e^{-c_4' d^\varepsilon} - d^{-3}|\mathbb{H}| \\
&\geq |\mathbb{H}|(1-d^{-2}).
\end{split}
\end{equation*}
Thus, on the event $\mathcal{H} \cap \{ |B_\varepsilon| \leq |\mathbb{H}|e^{-c_4' d^\varepsilon} \}$ and for $d\geq c$, the set $C = \bigcup_{j \in J_{\text{max}}} S_j$ forms a connected component of $E_\varphi^{\geq h_{\text{as}}(1-5\varepsilon)} \cap \mathbb{H}$ with $|C| \geq  |\mathbb{H}|(1-d^{-2})$. By definition, cf. \eqref{EQ:goodEventDef}, it is therefore a giant component of $E_\varphi^{\geq h_{\text{as}}(1-5\varepsilon)} \cap \mathbb{H}$. This yields \eqref{EQ:giant_comb_1}, which completes the proof of \eqref{EQ:giant_comp} (for $x = 0$, and thus all $x \in \mathbb{Z}^2$ by translation invariance).

\bigskip

\noindent With \eqref{EQ:giant_comp} at hand, we proceed with the proof of \eqref{EQ:giant_comps_conn}, which is similar, but simpler. For $x \in \mathbb{Z}^2$, we denote by $C_x$ the giant component of $E_\varphi^{\geq h_{\text{as}}(1-5\varepsilon)} \cap \mathbb{H}_x$, cf. \eqref{D:hypercube}, whenever it exists. By \eqref{EQ:goodEventDef},
\begin{equation*}
\begin{split}
\mathbb{P}[(G_0^{h_{\textrm{as}}(1-5\varepsilon)})^c] \leq \sum_{\substack{x \in \mathbb{Z}^2:}{|x|_1 \leq 1}} & (\mathbb{P}[\text{$C_{x}$ does not exist}] \\
&+ \mathbb{P}[\text{$C_0$, $C_{x}$ exist and } C_0 \nleftrightarrow C_{x} \text{ in $E_\varphi^{\geq h_{\text{as}}(1-8\varepsilon)}$}]).
\end{split}
\end{equation*}
On account of \eqref{EQ:giant_comp}, it thus suffices to prove that for all $d \geq 3$ and $x \in \mathbb{Z}^2$ with $|x|_1 = 1$,
\begin{equation} \label{EQ:giant_conn_1}
\mathbb{P}[\text{$C_0$, $C_{x}$ exist and } C_0 \nleftrightarrow C_{x} \text{ in $E_\varphi^{\geq h_{\text{as}}(1-8\varepsilon)}$}] \leq ce^{-c'd^\varepsilon}.
\end{equation}
For arbitrary $x \in \mathbb{Z}^2$ with $|x|_1 = 1$, let $\Pi_x$ denote the singletons and nearest-neighbor edges in $x + \{ 0,1 \}^{\mathbb{Z}^d}$, and denote by $\hat{\mathcal{G}} = \bigcap_{U \in \Pi_x} \{ \sup_{x \in U} |\mu_x^{U^c}| < \varepsilon h_{\text{as}}\}$, with $\mu_x^{U^c}$ as defined in \eqref{mu}. As in \eqref{EQ:giant_good_event_bound}, we have $\mathbb{P}[\hat{\mathcal{G}}^c] \leq c e^{-c' d^{\varepsilon}}$. Thus, we see that \eqref{EQ:giant_conn_1} follows at once if we show that 
\begin{equation} \label{EQ:giant_conn_2}
\sup_{K_+ \subset (\mathbb{H} \cup \mathbb{H}_x)} \mathbb{P}[\hat{\mathcal{G}}, \text{ $C_0$ and $C_{x}$ exist and } C_0 \nleftrightarrow C_{x} \text{ in $E_\varphi^{\geq h_{\text{as}}(1-8\varepsilon)}$}\, | \, \hat{\mathcal{C}}(K_+)] \leq ce^{-c'd^\varepsilon},
\end{equation}
for all $d\geq 3$, where $\hat{\mathcal{C}}(K_+) =  \{E_\varphi^{\geq h_{\text{as}}(1-5\varepsilon)} \cap (\mathbb{H} \cup \mathbb{H}_x) = K_+  \}$ specifies the configuration of $E_\varphi^{\geq h_{\text{as}}(1-5\varepsilon)}$ in $\mathbb{H} \cup \mathbb{H}_x$. Fix $x\in \mathbb{Z}^2$ with $|x|_1 = 1$ and $K_+ \subset \mathbb{H} \cup \mathbb{H}_x$. Letting $x + \{0,1 \}^{\mathbb{Z}^d} =  \mathbb{H}^0 \cup  \mathbb{H}^1$, with $ \mathbb{H}^0 = (x + \{0,1 \}^{\mathbb{Z}^d}) \cap \mathbb{H}$, $ \mathbb{H}^1 = (x + \{0,1 \}^{\mathbb{Z}^d}) \cap \mathbb{H}_x$, cf. \eqref{D:hypercube}, observe that, whenever $C_0$ and $C_{x}$ exist, the (disjoint) sets $\mathbb{H}^0 \setminus \overline{C}_0^{\mathbb{H}}$ and $\mathbb{H}^1 \setminus \overline{C}_{x}^{\mathbb{H}_{x}}$ each contain at most $d^{-2}|\mathbb{H}|$ elements. Moreover, the joint occurrence of $C_0$ and $C_{x}$ implies immediately that for all $d \geq 3$,
\begin{equation} \label{EQ:giant_conn_3}
\begin{split}
&\text{there exist disjoint sets $U_k = \{ y_k, z_k \} \subset (x + \{ 0,1 \}^{\mathbb{Z}^d})$, for $1 \leq k \leq m$,} \\
&\text{with $m \geq cd^{-1}( |\mathbb{H}|/2 - 2 d^{-2} |\mathbb{H}|) \ ( \geq c'd^{-1}|\mathbb{H}|)$, and such that} \\
&\text{$x_k \in \overline{C}_0^{\mathbb{H}} \cap \mathbb{H}^1$, $y_k \in \overline{C}_{x}^{\mathbb{H}_{x}} \cap \mathbb{H}^2$, $|x_k -y_k|_1 =1$, for all $1 \leq k \leq m$} \\
&\text{and $d_1(U_k,U_l) \geq 2$ for $1\leq k< l \leq m$}.
\end{split}
\end{equation} 
For $U_k$ not to form a path joining $C_0$ and $C_{x}$ in the level set $E_\varphi^{\geq h_{\text{as}}(1-8\varepsilon)}$, at least one of the two sites in $U_k$ must have a field value smaller than $h_{\text{as}}(1-8\varepsilon)$, i.e. setting $\hat{\mathcal{F}}_k = \bigcup_{y \in U_k} \{ \varphi_y < h_{\text{as}}(1-8\varepsilon) \}$, we obtain
\begin{align*}
&\mathbb{P}[\hat{\mathcal{G}}, \text{ $C_0$ and $C_{x}$ exist and } C_0 \nleftrightarrow C_{x} \text{ in $E_\varphi^{\geq h_{\text{as}}(1-8\varepsilon)}$}\, | \, \hat{\mathcal{C}}(K_+)] \\
&\qquad \leq \mathbb{P}[\hat{\mathcal{G}}, \hat{\mathcal{F}}_k, 1 \leq k \leq m \, | \, \hat{\mathcal{C}}(K_+)].
\end{align*}
Now, on account of \eqref{EQ:giant_conn_3}, an analysis similar to the one below \eqref{EQ:giant_event_H_bound2} yields that this last quantity is bounded from above by $2^{-cd^{-3}|\mathbb{H}|}$, for all $d\geq c$, which is more than enough to deduce \eqref{EQ:giant_conn_2}. This completes the proof of Theorem \ref{T:giant_comp}.
 \end{proof}

\noindent We now conclude the proof of the lower bound.

\begin{proof7}
Let $\varepsilon \in (0,1/3)$. It follows from Theorem \ref{T:giant_comp} that $\mathbb{P} [ G_0^{h_{\textrm{as}}(1-8\varepsilon)}] \geq 1 - c e^{-c'd^{\varepsilon}}$, for all $d\geq 3$. In particular, $\limsup_{d\to \infty} d^{2+(3/10)} \mathbb{P} [ (G_0^{h_{\textrm{as}}(1-8\varepsilon)})^c] = 0$, and Theorem \ref{T:LOCAL_GLOBAL} yields that $h_*(d) \geq h_{\textrm{as}}(d)(1-9\varepsilon)$, for all $d\geq c \ (=c(\varepsilon))$. The claim \eqref{EQ:ASYMPTOTICS_LB} follows. Moreover, by construction, cf. \eqref{EQ:good_sites}, percolation then occurs in $\mathbb{H} + \mathbb{Z}^2$, which implies \eqref{EQ:ASYMPTOTICS_LB_COR}. \hfill $\square$
\end{proof7}

\medskip 

\begin{proof8}
The claim \eqref{EQ:ASYMPTOTICS_DENSITY} is a direct consequence of \eqref{GRV_basic_estimate} and the fact that $h_*(d) = h_{\text{as}}(d)(1 + o(1))$ as $d \to \infty$, which follows immediately from \eqref{EQ:ASYMPTOTICS_UB} and \eqref{EQ:ASYMPTOTICS_LB}. \hfill $\square$
\end{proof8}

\begin{remark} $\quad$ 

\medskip
\noindent1) A thorough review of the proofs reveals that the long-range dependence present in the model is a serious impediment and considerably hinders any efforts to obtain more precise results than the leading exponential order given by Theorem \ref{T:ASYMPTOTICS_DENSITY}.

\noindent 2) The heuristic paradigm requiring for this model to exhibit strong similarities to a corresponding one on the $(2d)$-regular tree is evidently inherent to many of the above proofs. Specifically, this behavior manifests itself when we investigate the local connectivity of the level set, essentially because the ``tree-like'' structure of the lattice determines the behavior of the random walk at short range, see the proofs of Lemma \ref{L:zeta_i_bound} and Theorem \ref{T:LOCAL_CB} (which crucially rely on \eqref{return_probas1}) for the upper bound, and the proof of Theorem \ref{P:SUBST_COMP} for the lower bound. Accordingly, we would like to compare \eqref{EQ:ASYMPTOTICS_DENSITY} to the critical density of the same model on this tree. We hope to come back to this point in the future. \hfill $\square$
\end{remark}

\noindent{\textit{Note added in proof:}} a recent preprint of Lupu \cite{Lu14}, see Theorem 3 therein, shows that for any $u>0$, there exists a coupling between random interlacements $\mathcal{I}^u$ at level $u$, see \cite{S1} for the definition, and the free field $\varphi$ such that $E_{\varphi}^{\geq \sqrt{2u}} \subseteq \mathcal{V}^u$, where $ \mathcal{V}^u = \mathbb{Z}^d \setminus  \mathcal{I}^u$ is the so-called vacant set (at level $u$). In particular, this readily implies that $h_{*} \leq \sqrt{2u_*}$, where $u_*$ denotes the critical parameter for interlacement percolation. Together with our lower bound \eqref{EQ:ASYMPTOTICS_LB}, this shows that $\liminf_{d \to \infty} u_*(d)/ \log(d) \geq 1$, which matches precisely the lower bound obtained by Sznitman in \cite{SLB}, and thus provides an alternative proof of this result. Similarly, the inequality $h_{**} \leq \sqrt{2u_{**}}$ and Sznitman's asymptotic upper bound on $u_{**}$, cf. \cite{SUB}, Theorem 0.1 and Remark 4.1, yield another proof of \eqref{EQ:ASYMPTOTICS_UB}.

\bigskip

\noindent{\textbf{Acknowledgements.}} A. D. and P.-F. R. would like to thank the Forschungsinstitut f\"ur Mathematik at ETH Z\"urich and the Columbia University Mathematics Department, respectively, for their hospitality and financial support. We are  grateful to R. Rosenthal, A. Sapozhnikov and A.-S. Sznitman for their comments on an earlier version of this paper. We also thank an anonymous referee for various remarks on the manuscript.

\end{document}